\documentclass[letterpaper,oneside,english]{amsart}
\usepackage[T1]{fontenc}
\usepackage[latin9]{inputenc}
\usepackage{units}
\usepackage{mathrsfs}
\usepackage{amsthm}
\usepackage{amsbsy}
\usepackage{amstext}
\usepackage{amssymb}
\usepackage{esint}

\usepackage{graphicx}
\usepackage{lipsum}

\makeatletter

\numberwithin{equation}{section}
\numberwithin{figure}{section}
\theoremstyle{plain}
\newtheorem{thm}{\protect\theoremname}
\theoremstyle{definition}
\newtheorem{defn}[thm]{\protect\definitionname}
\theoremstyle{remark}
\newtheorem{rem}[thm]{\protect\remarkname}
\theoremstyle{plain}
\newtheorem{lem}[thm]{\protect\lemmaname}
\ifx\proof\undefined
\newenvironment{proof}[1][\protect\proofname]{\par
\normalfont\topsep6\p@\@plus6\p@\relax
\trivlist
\itemindent\parindent
\item[\hskip\labelsep
\scshape
#1]\ignorespaces
}{%
\endtrivlist\@endpefalse
}
\providecommand{\proofname}{Proof}
\fi
\theoremstyle{plain}
\newtheorem{prop}[thm]{\protect\propositionname}

\makeatother

\usepackage{babel}
\providecommand{\definitionname}{Definition}
\providecommand{\lemmaname}{Lemma}
\providecommand{\propositionname}{Proposition}
\providecommand{\remarkname}{Remark}
\providecommand{\theoremname}{Theorem}

\numberwithin{equation}{section} 
\numberwithin{thm}{section}

\begin{document}

\title[Uniqueness of $F$ self-shrinkers with a tangent cone at infinity]{Uniqueness of self-shrinkers to the degree-one curvature flow with
a tangent cone at infinity}

\author{Siao-Hao Guo}
\begin{abstract}
Given a smooth, symmetric, homogeneous of degree one function $f=f\left(\lambda_{1},\cdots,\,\lambda_{n}\right)$
satisfying $\partial_{i}f>0$ for all $i=1,\cdots,\, n$, and an
oriented, properly embedded smooth cone $\mathcal{C}^n$ in $\mathbb{R}^{n+1}$,
we show that under some suitable conditions on
$f$ and the covariant derivatives of the
second fundamental form of $\mathcal{C}$, there is at most one $f$
self-shrinker (i.e. an oriented hypersurface $\Sigma^n$ in $\mathbb{R}^{n+1}$
for which $f\left(\kappa_{1},\cdots,\,\kappa_{n}\right)+\frac{1}{2}X\cdot N=0$
holds, where $X$ is the position vector, $N$ is the unit normal vector, and $\kappa_{1},\cdots,\,\kappa_{n}$ are principal curvatures
of $\Sigma$) that is asymptotic to the given cone $\mathcal{C}$
at infinity. 
\end{abstract}
\maketitle

\section{Introduction}

Let $\mathcal{C}^n$ be an oriented, properly embedded smooth cone
(excluding the vertex $O$) in $\mathbb{R}^{n+1}$. Suppose that $\boldsymbol{\Sigma}^n$
is an oriented, properly embedded smooth hypersurface in $\mathbb{R}^{n+1}$
which satisfies 
\[
H\,+\,\frac{1}{2}X\cdot N=0\quad\forall\, X\in\boldsymbol{\Sigma}
\]
\[
\varrho\boldsymbol{\Sigma}\overset{C_{\textrm{loc}}^{\infty}}{\longrightarrow}\mathcal{C}\quad\textrm{as}\;\varrho\searrow0
\]
where $X$ is the position vector, $N$ is the unit normal vector and $H=-\nabla_{\boldsymbol{\Sigma}}\cdot N$
is the mean curvature of $\boldsymbol{\Sigma}$. Then $\boldsymbol{\Sigma}$ is
called a self-shrinker to the mean curvature flow (MCF, an one-parameter family of hypersurfaces for which $\partial_{t}X_t^{\bot}=HN$ holds)
which is smoothly asymptotic to the cone $\mathcal{C}$ at infinity.
It follows that the rescaled family of hypersurfaces $\left\{ \boldsymbol{\Sigma}_{t}=\sqrt{-t}\,\boldsymbol{\Sigma}\right\} $
forms a mean curvature flow starting from $\boldsymbol{\Sigma}$ (when
$t=-1$) and converging locally smoothly to $\mathcal{C}$ as $t\nearrow0$.
Wang in \cite{W} proves the uniqueness of such self-shrinkers by
showing the following: suppose $\tilde{\boldsymbol{\Sigma}}$ is also a self-shrinker
which is asymptotic to the same cone $\mathcal{C}$ at infinity, then outside a large ball $B^{n+1}_{R}$,
each $\tilde{\boldsymbol{\Sigma}}_{t}=\sqrt{-t}\,\tilde{\boldsymbol{\Sigma}}$
can be regarded as a normal graph of $\mathit{\mathtt{h}}_{t}$ defined
on $\boldsymbol{\Sigma}_{t}\setminus\bar{B}_{R}$ for some $R>0$;
moreover, given $\varepsilon>0$ and choose $R$ sufficiently large,
there holds 
\[
\Big|\partial_{t}\mathit{\mathtt{h}}-\triangle_{\boldsymbol{\Sigma}_{t}}\mathtt{h\Big|}\leq\varepsilon\left(|\nabla_{\boldsymbol{\Sigma}_{t}}\mathtt{h}|\,+\,|\mathtt{h}|\right)
\]
\[
\mathtt{h}\Big|_{t=0}=0
\]
Then using the idea of \cite{ESS}, Wang derives a Carleman's inequality
for the heat operator on $\left\{ \boldsymbol{\Sigma}_{t}\right\} $,
applies it to the localization of $\mathtt{h}$, and then uses the unique
continuation principle (see \cite{EF}, for instance) to conclude
that $\mathtt{h}=0$. 

On the other hand, Andrews in \cite{A} consider the motion of hypersurfaces
in $\mathbb{R}^{n+1}$ moved by some degree one curvature. More precisely,
given a smooth, symmetric and homogeneous of degree-one function $f=f\left(\lambda_{1},\cdots,\,\lambda_{n}\right)$
which satisfies 
\[
\partial_{i}f>0,\quad\forall\: i=1,\cdots, n
\]
consider the following evolution of hypersurfaces:
\[
\partial_{t}X_t^{\bot}=f\left(\kappa_{1},\cdots,\,\kappa_{n}\right)N
\]
where $\kappa_{1},\cdots,\,\kappa_{n}$ are the principal curvatures
of the evolving hypersurface. For example, if we take the curvature
function to be $f\left(\lambda_{1},\cdots,\,\lambda_{n}\right)=\lambda_{1}+\cdots+\lambda_{n}$,
then it corresponds to the mean curvature flow. We call an oriented
$C^{2}$ hypersurface $\Sigma^n$ in $\mathbb{R}^{n+1}$to be a ``$f$
self-shrinker'' to the above ``$f$ curvature flow'' provided
that
\[
f\left(\kappa_{1},\cdots,\,\kappa_{n}\right)\,+\,\frac{1}{2}X\cdot N=0
\]
holds on $\Sigma$. Examples of $f$ self-shrinker can be found in \cite{G}. Just like the MCF, the rescaled family of ``$f$ self-shrinkers''
is a self-similar solution to the $f$ curvature flow; that is, the
rescaled family of hypersurfaces $\left\{ \Sigma_{t}=\sqrt{-t}\,\Sigma\right\} _{t<0}$
forms a ``$f$ curvature flow''. In the case when $\Sigma$ is smoothly asymptotic
to the cone $\mathcal{C}$ at infinity, the rescaled flow $\left\{ \Sigma_{t}\right\} _{t<0}$
converges locally smoothly to $\mathcal{C}$ as $t\nearrow0$.

This paper is an extension of the uniqueness result of \cite{W} to the class of $f$ self-shrinkers
with a tangent cone at infinity. Based on Wang's idea of proving the uniqueness, we need to have
some additional treatments to the nonlinearity of $f$ (which is not a concern in
Wang's case because the curvature function there is linear) in order to generalize the result.  
The crucial step is to derive Carleman's inequality for the associated parabolic
operator to the ``$f$ curvature flow" under some assumptions on the nonlinearity
of $f$, the uniform positivity of $\partial_{i}f$ and some
curvature bounds of $\mathcal{C}$ (see Proposition \ref{p26}). For this part, we are motivated
by the work of Nguyen in \cite{N} as well as Wu and Zhang in \cite{WZ}
for deriving Carleman's inequality for parabolic operator with variable
coefficients (see Remark \ref{r25}). 

In order to state our main results, Theorem \ref{t5}, we have to first introduce some notations and definitions
regarding the $f$ self-shrinkers, the tangent cone of a hypersurface
at infinity, and also some basic assumptions on the curvature
function $f$. We put all of these in Section 2.

In Section 3, we essentially follow the idea of \cite{W}: if $\Sigma^n$
and $\tilde{\Sigma}^n$ are $f$ self-shrinkers which are asymptotic
to the given cone $\mathcal{C}^n$ at infinity, then outside a large ball $B^{n+1}_R$,
$\tilde{\Sigma}_{t}=\sqrt{-t}\,\tilde{\Sigma}$ can be regarded
as a normal graph of $h_{t}$ defined on $\Sigma_{t}\setminus\bar{B}_{R}$ (see Lemma \ref{l8}),
which satisfies some parabolic equation and vanishes
at time 0 (see Proposition \ref{p14}). We also give some estimates on the coefficients of
the parabolic equation (see Proposition \ref{p15}). 

In Section 4, we follow the idea of \cite{ESS} for treating the backward
uniqueness of the heat equation (which is also used in \cite{W} to
deal with the uniqueness of self-shrinkers to the MCF)
to show that the function $h_{t}$ vanishes outside a large ball $B_R$.
We would first apply the mean value inequality for parabolic equations
and a local type of Carleman's inequalities to show the exponential
decay of $h_{t}$ to 0 as $t\nearrow0$ as in \cite{N} (see Proposition \ref{p22}). Then we are devoted to derive a global type of Carleman's inequalities
(based on the estimates of the coefficients of the parabolic equation
derived in Section 3, see Proposition \ref{p26}) and use it to show that $h_{t}$ vanishes
outside a ball $B_R$; in other words, the two shrinkers coincide outside a ball $B_R$. 
In the end, we use the unique continuation
principle to characterize the overlap region of $\Sigma$ and $\tilde{\Sigma}$.

\section*{Acknowledgement}
I am grateful to my advisor, Natasa Sesum, for proposing this problem and for her great support, as well as many helpful suggestions to this paper.

\vspace{.5in} 
\section{Assumptions and main results}
\begin{defn}[A regular cone]\label{d1}
Let $\mathcal{C}^n$ be an oriented and properly embedded smooth cone
(excluding the vertex $O$) in $\mathbb{R}^{n+1}$; that is, $\mathcal{C}$
is an oriented and properly embedded hypersurface in $\mathbb{R}^{n+1}$
satisfying $\varrho\,\mathcal{C}=\mathcal{C}$ for all $\varrho\in\mathbb{R}_{+}$
, and we assume that $O\notin\mathcal{C}$. 
\end{defn}
We then define what it means for a hypersurface to be asymptotic to
the cone $\mathcal{C}$ at infinity.
\begin{defn}[Tangent cone at infinity] \label{d2}
A $C^{k}$ hypersurface $\Sigma^n$ in $\mathbb{R}^{n+1}$ (with $k\in\mathbb{N}$)
is said to be $C^{k}$ asymptotic to $\mathcal{C}$ at infinity provided
that $\varrho\Sigma\overset{C_{\textrm{loc}}^{k}}{\longrightarrow}\mathcal{C}$
as $\varrho\searrow0$. In this case, $\mathcal{C}$
is called the tangent cone of $\Sigma$ at infinity.
\end{defn}
For a given $C^{2}$ oriented hypersurface $\Sigma^n$ in $\mathbb{R}^{n+1}$,
its Weingarten map $A^{\#}$ sends tangent vectors
to tangent vectors in such a way that
\[
A^{\#}\left(V\right)=-D_{V}N
\]
for any tangent vector field $V$ on $\Sigma$, where $N$ is the
unit-normal of $\Sigma$. The second fundamental form $A$ is defined
to be a 2 tensor on $\Sigma$ so that 
\[
A\left(V,\, W\right)=A^{\#}\left(V\right)\cdot W
\]
for any tangent vector fields $V$ and $W$ on $\Sigma$. The components
of $A^{\#}$ and $A$ with respect to a given local frame $\left\{ e_{1},\cdots,\, e_{n}\right\} $
of the tangent bundle of $\Sigma$ are defined by 
\[
A^{\#}\left(e_{i}\right)=A_{i}^{j}e_{j},\quad A\left(e_{i},e_{j}\right)=A_{ij}
\]
For simplicity, $A^{\#}$ and $A$ are usually denoted by their components:
$A^{\#}\sim A_{i}^{j}$ and $A\sim A_{ij}$. Note that $A^{\#}$ is
a self-adjoint operator with respect to the induced metric of the hypersurface (or equivalently, $A$ is a symmetric 2 tensor),
so $A^{\#}$ is diagonalizable. The eigenvectors of $A^{\#}$ are
called principal vectors and the associated eigenvalues are called principal curvatures,
which we denote by $\kappa_{1},\cdots,\,\kappa_{n}$. The mean curvature
is defined to be $H=\textrm{tr}\left(A^{\#}\right)=\kappa_{1}+\cdots+\kappa_{n}$,
which is a linear, symmetric and homogeneous of degree-one function
of the shape operator (or the principal curvatures). In this paper,
we consider a generic degree-one curvature.

\begin{defn}[The degree-one curvature function]\label{d3}
Let $F=F\left(S\right)$ be a conjugation-invariant, homogeneous of
degree-one function whose domain $\boldsymbol{\Omega}$ (in the space
of $n\times n$ matrices) containing a neighborhood of the set consisting
of all the values of shape operator $A_{\mathcal{C}}^{\#}$ of $\mathcal{C}$;
besides, $F$ can be written as a $C^{3}$ function composed with
the the elementary symmetric functions $\mathcal{E}_{1},\cdots,\,\mathcal{E}_{n}$
(for instance, $\mathcal{E}_{1}=\textrm{tr}$ and $\mathcal{E}_{n}=\textrm{det}$)
and $\frac{\partial F}{\partial S_{i}^{j}}>0$ (i.e. $\frac{\partial F}{\partial S_{i}^{j}}$
is a positive matrix). In particular, we require the curvature function $F$ to be defined 
and $C^{3}$ on the curvature of $\mathcal{C}$. 

Note that by the conjugation-invariant and homogeneous property of
$F$, we may assume that $\boldsymbol{\Omega}$ is closed under conjugation
and homothety; that is, if $S\in\boldsymbol{\Omega}$, then so are
$RSR^{-1}$ and $\varrho S$ for any invertible $n\times n$ matrix
$R$ and positive number $\varrho$.

Also, by the condition that $F$ can be written as a $C^{3}$ function
composed with the the elementary symmetric functions, it induces a
symmetric, homogeneous of degree-one function $f$ so that 
\[
F\left(S\right)=f\left(\lambda_{1},\cdots,\,\lambda_{n}\right)
\]
whenever $\lambda_{1},\cdots,\,\lambda_{n}$ are the eigenvalues of
the matrix $S$. The function $f$ is defined and $C^{3}$ on an open
set $\mho$ (in $\mathbb{R}^{n}$) containing a neighborhood of the
set consisting of all the values of the princinpal curvature vector
$\left(\kappa_{1}^{\mathcal{C}},\cdots,\,\kappa_{n}^{\mathcal{C}}\right)$
of $\mathcal{C}$. Likewise, we may assume that the domain $\mho$
is closed under permutation and homothety.

In fact, at a diagonal matrix $S=\textrm{diag}\left(\lambda_{1},\cdots,\,\lambda_{n}\right)$,
there holds (see \cite{A}):
\begin{equation}\label{1}
\frac{\partial F}{\partial S_{i}^{j}}\left(S\right)=\partial_{i}f\left(\lambda_{1},\cdots,\,\lambda_{n}\right)\delta_{ij}
\end{equation}
\begin{equation}
\frac{\partial^{2}F}{\partial S_{i}^{j}\partial S_{i}^{l}}\left(S\right)=\partial_{ii}^{2}f\left(\lambda_{1},\cdots,\,\lambda_{n}\right)\delta_{ij}\delta_{il}\label{2}
\end{equation}
\begin{equation}
\frac{\partial^{2}F}{\partial S_{i}^{j}\partial S_{k}^{l}}\left(S\right)=\partial_{ik}^{2}f\left(\lambda_{1},\cdots,\,\lambda_{n}\right)\delta_{ij}\delta_{kl}\,+\,\frac{\partial_{i}f-\partial_{k}f}{\lambda_{i}-\lambda_{k}}\delta_{il}\delta_{kj}\quad\textrm{if}\; i\neq k\label{3}
\end{equation}
Since $F$ is well-defined on conjugacy classes, (\ref{1}), (\ref{2}),
(\ref{3}) can be applied to any diagonalizable matrix in $\boldsymbol{\Omega}$.
For example, by (\ref{1}), we have
\[
\frac{\partial F}{\partial S_{i}^{j}}\left(A_{\mathcal{C}}^{\#}\right)\sim\partial_{i}f\left(\kappa_{1}^{\mathcal{C}},\cdots,\,\kappa_{n}^{\mathcal{C}}\right)\delta_{ij}
\]
where $A_{\mathcal{C}}^{\#}\sim\kappa_{\mathcal{C}}^{i}\delta_{ij}$
are the shape operator and principal curvatures of $\mathcal{C}$, respectively.
Besides, by the condition that $\frac{\partial F}{\partial S_{i}^{j}}>0$
on $\boldsymbol{\Omega}$, we may assume that $\partial_{i}f>0\quad\forall\, i=1,\cdots,\, n$
on $\mho$.

Let $U$ be an open neighborhood of the set consisting of the
all the shape operator $A_{\mathcal{C}}^{\#}$ of $\mathcal{C}$ at each
$X_{\mathcal{C}}\in\mathcal{C}\cap\left(B_{3}\setminus\bar{B}_{\frac{1}{3}}\right)$
in $\boldsymbol{\Omega}$. Note that we may assume that $U$ is closed
under conjugation and that $\frac{\partial F}{\partial S_{i}^{j}}$
is uniformly positive on $U$; that is, there exist a constant $\lambda\in(0,\,1]$
so that
\begin{equation}
\lambda\delta_{j}^{i}\,\leq\frac{\partial F}{\partial S_{i}^{j}}\leq\,\frac{1}{\lambda}\delta_{j}^{i}\label{4}
\end{equation}
Also, we have
\begin{equation}
\varkappa\,\equiv\sup_{X_{\mathcal{C}}\in\mathcal{C}\cap\left(B_{3}\setminus\bar{B}_{\frac{1}{3}}\right)}\Big|\nabla_{\mathcal{C}}\left(\frac{\partial F}{\partial S_{i}^{j}}\left(A_{\mathcal{C}}^{\#}\right)\right)\Big|\label{5}
\end{equation}
\[
=\sup_{X_{\mathcal{C}}\in\mathcal{C}\cap\left(B_{3}\setminus\bar{B}_{\frac{1}{3}}\right)}\Big|\sum_{k,\, l}\frac{\partial^{2}F}{\partial S_{i}^{j}\partial S_{k}^{l}}\left(A_{\mathcal{C}}^{\#}\right)\,\left(\nabla_{\mathcal{C}}A_{\mathcal{C}}^{\#}\right)_{k}^{l}\Big|\,
\]
\[
\leq\, C\left(n,\,\mathcal{C},\,\parallel F\parallel_{C^{2}\left(U\right)}\right)
\]
where $A_{\mathcal{C}}^{\#}$ and $\nabla_{\mathcal{C}}A_{\mathcal{C}}^{\#}$
are the shape operator of $\mathcal{C}$ and its covariant derivative
at $X_{\mathcal{C}}$, respectively; $B_{\varrho}=B^{n+1}_{\varrho}$ is the ball of
radius $\varrho$ in $\mathbb{R}^{n+1}$. We would give a more precise estimate
of $\varkappa$ in Section 5 (see(\ref{247})) in the case when $\mathcal{C}$ is
rotationally symmetric.
\end{defn}
Now let's define the $F$ self-shrinker (or $f$ self-shrinker): 
\begin{defn}[$F$ self-shrinker/ $f$ self-shrinker]\label{d4}

An oriented $C^{2}$ hypersurface $\Sigma^n$ 
in $\mathbb{R}^{n+1}$ is called a $F$ self-shrinker (or $f$ self-shrinker)
provided that $F$ is defined on the shape operator $A^{\#}$ of $\Sigma$
(i.e. $A^{\#}\in\boldsymbol{\Omega}$) and there holds
\[
F\left(A^{\#}\right)\,+\,\frac{1}{2}X\cdot N=0
\]
where $X$ is the position vector, $N$ is the unit-normal, and $A^{\#}$
is the shape operator of $\Sigma$; or equivalently, $f$ is defined
on the principal curvatures of $\Sigma$ (i.e. $\left(\kappa_{1},\cdots,\,\kappa_{n}\right)\in\mho$)
and there holds
\[
f\left(\kappa_{1},\cdots,\,\kappa_{n}\right)\,+\,\frac{1}{2}X\cdot N=0
\]
where $\kappa_{1},\cdots,\,\kappa_{n}$ are the principal curvatures
of $\Sigma$. 

Note that the rescaled family of $F$ self-shrinkers forms a self-similar
solution to the $F$ curvature flow. More precisely, the one-parameter
family $\left\{ \Sigma_{t}=\sqrt{-t}\,\Sigma\right\} _{-1\leqslant t<0}$
is a motion of a hypersurface moved by $F$ curvature vector. That
is, 
\[
\partial_{t}X_t^{\bot}=F\left(A^{\#}\right)N
\]
where $\partial_{t}X_t^{\bot}$ is the normal projection of $\partial_{t}X_t$.
Besides, for each time-slice $\Sigma_{t}=\sqrt{-t}\,\Sigma$, there
holds
\[
F\left(A^{\#}\right)\,+\,\frac{X\cdot N}{2(-t)}=0
\]

\end{defn}
We would prove the following uniqueness result for $F$ self-shrinkers
with a tangent cone in Section 4.
\begin{thm}[Uniqueness of self-shrinkers with a conical end]\label{t5}

Assume that $\varkappa\leq6^{-4}\lambda^{3}$ (see (\ref{4}), (\ref{5})).
Then for any properly embedded $F$ self-shrinkers $\Sigma^n$ and $\tilde{\Sigma}^n$
which are $C^{5}$ asymptotic to the cone $\mathcal{C}$ at infinity,
there exists \\ $R=R\left(\Sigma,\,\tilde{\Sigma},\mathcal{\, C},\, U,\,\parallel F\parallel_{C^{3}\left(U\right)},\,\lambda,\,\varkappa\right)\geq1$
so that $\Sigma\setminus B_{R}=\tilde{\Sigma}\setminus B_{R}$. \\Moreover,
let

\[
\Sigma^{0}=\left\{ X\in\Sigma\cap\tilde{\Sigma}\Big|\,\Sigma\textrm{ coincides with }\tilde{\Sigma}\textrm{ in a neighborhood of }X\right\} 
\]
then $\Sigma^{0}$ is a nonempty hypersurface, which satisfies $\partial\Sigma^{0}\subseteq\left(\partial\Sigma\,\cup\,\partial\tilde{\Sigma}\right)$. \end{thm}
\begin{rem}
In the case of \cite{W}, $F=\mathcal{E}_{1}$ (or equivalently,
$f\left(\lambda_{1},\cdots,\,\lambda_{n}\right)=\lambda_{1}+\cdots+\lambda_{n}$)
is a linear function, so (by (\ref{5}), (\ref{2}), (\ref{3})) $\varkappa\equiv0$
and the hypothesis of Theorem \ref{t5} is trivially satisfied. On the other
hand, consider
\[
F=\mathcal{E}_{1}\,\pm\,\epsilon\frac{\mathcal{E}_{n}}{\mathcal{E}_{n-1}}
\]
or equivalently, 
\[
f\left(\lambda_{1},\cdots,\,\lambda_{n}\right)=\left(\lambda_{1}+\cdots+\lambda_{n}\right)\,\pm\,\epsilon\frac{\prod_{i=1}^{n}\lambda_{i}}{\sum_{i=1}^{n}\left(\prod_{j\neq i}\lambda_{j}\right)}
\]
and take $\mathcal{C}$ to be a rotationally symmetric cone. Then
by Theorem \ref{t5} and (\ref{247}) in the Section 5, the uniqueness
holds when $0<\epsilon\ll1$. 
\end{rem}

\vspace{.5in}
\section{Deviation between two $F$ self-shrinkers with the same asymptotic
behavior at infinity}

Let $\Sigma^n$ be a properly embedded $F$ self-shrinker (in Definition
\ref{d4}) which is $C^{5}$ asymptotic to the cone $\mathcal{C}$ at infinity. 

By Definition \ref{d2}, $\varrho\Sigma$ can be arbitrary $C^{5}$ close
to $\mathcal{C}$ on any fixed bounded set of $\mathbb{R}^{n+1}$
which is away from the origin (e.g. on $B_{2}\setminus\bar{B}_{\frac{1}{2}}$)
as long as $\varrho$ is sufficiently small. Below we would like to use this condition to show that any ``rescaled'' $C^{5}$
quantities of $\Sigma\setminus\bar{B}_{R}$ can estimated by that
of $\mathcal{C}$ (if $R$ is sufficiently large). 

First, choose $R\gg1$ (depending on $\Sigma$, $\mathcal{C}$)
so that outside a compact set, $\Sigma$ is a normal graph over
$\mathcal{C}\setminus\bar{B}_{R}$, say $X=\boldsymbol{\Psi}\left(X_{\mathcal{C}}\right)=X_{\mathcal{C}}+\boldsymbol{\psi}N_{\mathcal{C}}$,
where $X_{\mathcal{C}}$ is the position vector of $\mathcal{C}$,
 $N_{\mathcal{C}}$ is the unit-normal of $\mathcal{C}$ at $X_{\mathcal{C}}$ and 
$\boldsymbol{\psi}$ is a real-valued function of $X_{\mathcal{C}}$.
Consequently, it's natural to define the ''normal projecton'' $\Pi$ to
be the inverse map of $\boldsymbol{\Psi}$, which sends $X\in\Sigma$
to $X_{\mathcal{C}}\in\mathcal{C}$. Also, by the rescaling argument,
we may assume that 
\[
\mathcal{H}^{n}\left(\Sigma\cap\left(B_{2\mathrm{r}}\setminus\bar{B}_{\mathrm{r}}\right)\right)\leq C\left(n,\,\mathcal{C}\right)\mathrm{r}^{n}
\]
for all $\mathrm{r}\geq R$ (i.e. $\Sigma$ has polynomial volume
growth).

On the other hand, for each fixed $\hat{X}_{\mathcal{C}}\in\mathcal{C}\setminus\bar{B}_{R}$, we have $|\hat{X}_{\mathcal{C}}|{}^{-1}\hat{X}_{\mathcal{C}}\in\mathcal{C}$. 
So near $|\hat{X}_{\mathcal{C}}|{}^{-1}\hat{X}_{\mathcal{C}}$,
 $\mathcal{C}$ is locally a graph
over its tangent hyperplane at $|\hat{X}_{\mathcal{C}}|{}^{-1}\hat{X}_{\mathcal{C}}$.
By Definition \ref{d2}, $|\hat{X}_{\mathcal{C}}|{}^{-1}\Sigma$ is $C^{5}$ close
to $\mathcal{C}$, so it must also
be a local graph over $T_{|\hat{X}_{\mathcal{C}}|^{-1}\hat{X}_{\mathcal{C}}}\mathcal{C}$,
and the graph must be $C^{5}$ close to the corresponding graph of $\mathcal{C}$.

Furthermore, by \cite{L}, there exits a uniform radius $\rho\in(0,\,\frac{1}{8}]$
(depending on the dimension $n$, the volume of $\mathcal{C}\cap\left(B_{3}\setminus\bar{B}_{\frac{1}{3}}\right)$ and the $C^{3}$ bound
of the curvature of $\mathcal{C}\cap\left(B_{3}\setminus\bar{B}_{\frac{1}{3}}\right)$)
so that near $|\hat{X}_{\mathcal{C}}|^{-1}\hat{X}_{\mathcal{C}}$,
the graph of $\mathcal{C}$ and
the graph of $|\hat{X}_{\mathcal{C}}|^{-1}\Sigma$ are defined on $B_{\rho}^{n} \subset T_{|\hat{X}_{\mathcal{C}}|^{-1}\hat{X}_{\mathcal{C}}}\mathcal{C}$.
We may also assume that the $C^{1}$ norm of the local graph of $\mathcal{C}$ on $B_{\rho}^{n} \subset T_{|\hat{X}_{\mathcal{C}}|^{-1}\hat{X}_{\mathcal{C}}}\mathcal{C}$ is small (by choosing $\rho$ small).

After undoing the rescaling (from $|\hat{X}_{\mathcal{C}}|{}^{-1}\Sigma$ to $\Sigma$), the above translates into the following: there exists
$R=R\left(\Sigma,\mathcal{\, C}\right)\geq1$ so that near each $\hat{X}_{\mathcal{C}}\in\mathcal{C}\setminus\bar{B}_{R}$,
$\mathcal{C}$ and $\Sigma$ can be repectively parametrized by
\[
X_{\mathcal{C}}=X_{\mathcal{C}}\left(x\right)\equiv\hat{X}_{\mathcal{C}}+\left(x,\,\mathtt{w}\left(x\right)\right)
\]
 
\[
X=X\left(x\right)\equiv\hat{X}_{\mathcal{C}}+\left(x,\,\mathtt{u}\left(x\right)\right)
\]
for $x=\left(x_{1},\cdots,\, x_{n}\right)\in B^n_{\rho|\hat{X}_{\mathcal{C}}|} \subset T_{\hat{X}_{\mathcal{C}}}\mathcal{C}$,
such that $\mathtt{w}\left(0\right)=0$, $\partial_{x}\mathtt{w}\left(0\right)=0$
and 
\begin{equation}
|\hat{X}_{\mathcal{C}}|^{-1}\parallel\mathtt{w}\parallel_{L^{\infty}(B_{\rho|\hat{X}_{\mathcal{C}}|}^{n})}+\parallel\partial_{x}\mathtt{w}\parallel_{L^{\infty}(B_{\rho\mid\hat{X}_{\mathcal{C}}\mid}^{n})}\,\leq\frac{1}{16}\label{6}
\end{equation}
\begin{equation}
|\hat{X}_{\mathcal{C}}|\parallel\partial_{x}^{2}\mathtt{w}\parallel_{L^{\infty}(B_{\rho|\hat{X}_{\mathcal{C}}|}^{n})}+\cdots+|\hat{X}_{\mathcal{C}}|{}^{4}\parallel\partial_{x}^{5}\mathtt{w}\parallel_{L^{\infty}(B_{\rho|\hat{X}_{\mathcal{C}}|}^{n})}\,\leq C\left(n,\,\mathcal{C}\right)\label{7}
\end{equation}

\begin{equation}
\{|\hat{X}_{\mathcal{C}}|^{-1}\parallel\mathtt{u}-\mathtt{w}\parallel_{L^{\infty}(B_{\rho|\hat{X}_{\mathcal{C}}|}^{n})}+\parallel\partial_{x}\mathtt{u}-\partial_{x}\mathtt{w}\parallel_{L^{\infty}(B_{\rho|\hat{X}_{\mathcal{C}}|}^{n})}\label{8}
\end{equation}

\[
+|\hat{X}_{\mathcal{C}}|\parallel\partial_{x}^{2}\mathtt{u}-\partial_{x}^{2}\mathtt{w}\parallel_{L^{\infty}(B_{\rho|\hat{X}_{\mathcal{C}}|}^{n})}+\cdots
\]

\[
+\,|\hat{X}_{\mathcal{C}}|{}^{4}\parallel\partial_{x}^{5}\mathtt{u}-\partial_{x}^{5}\mathtt{w}\parallel_{L^{\infty}(B_{\rho|\hat{X}_{\mathcal{C}}|}^{n})}\}\,\,\leq\frac{1}{16}
\]
where we assume the unit-normal of $\mathcal{C}$ at $\hat{X}_{\mathcal{C}}$
to be $\left(0,\,1\right)$ for ease of notation (and hence $\Pi\left(X\left(0\right)\right)=\hat{X}_{\mathcal{C}}$).
Note that (\ref{6}) is the rescale of the smallness of the $C^{1}$
norm of the local graph of $\mathcal{C}$, and (\ref{8}) is the
rescale of the small $C^{5}$ difference between the local graphs
of $\mathcal{C}$ and $|\hat{X}_{\mathcal{C}}|^{-1}\Sigma$. 

By Definition \ref{d2} and the rescaling argument, the same thing holds for
each rescaled hypersurface $\Sigma_{t}=\sqrt{-t}\,\Sigma$, $t\in[-1,\,0)$
as well. That is, outside a compact set, $\Sigma_{t}$ is a normal
graph over $\mathcal{C}\setminus\bar{B}_{R}$ (with $R\gg1$ depending
on $\Sigma$, $\mathcal{C}$); besides, near each $\hat{X}_{\mathcal{C}}\in\mathcal{C}\setminus\bar{B}_{R}$,
$\Sigma_{t}$ is a graph over $T_{|\hat{X}_{\mathcal{C}}|^{-1}\hat{X}_{\mathcal{C}}}\mathcal{C}$
and it can be parametrized by 
\[
X_{t}\left(x\right)=X\left(x,\, t\right)\equiv\hat{X}_{\mathcal{C}}+\left(x,\,\mathtt{u}_{t}\left(x\right)\right)=\hat{X}_{\mathcal{C}}+\left(x,\,\mathtt{u}\left(x,\, t\right)\right)
\]
which satisfies 

\begin{equation}
\{|\hat{X}_{\mathcal{C}}|^{-1}\parallel\mathtt{u}\left(\cdot,\, t\right)-w\parallel_{L^{\infty}(B_{\rho|\hat{X}_{\mathcal{C}}|}^{n})}+\parallel\partial_{x}\mathtt{u}\left(\cdot,\, t\right)-\partial_{x}\mathtt{w}\parallel_{L^{\infty}(B_{\rho|\hat{X}_{\mathcal{C}}|}^{n})}\label{9}
\end{equation}
\[
+\,|\hat{X}_{\mathcal{C}}|\parallel\partial_{x}^{2}\mathtt{u}\left(\cdot,\, t\right)-\partial_{x}^{2}\mathtt{w}\parallel_{L^{\infty}(B_{\rho|\hat{X}_{\mathcal{C}}|}^{n})}+\cdots
\]

\[
+\,|\hat{X}_{\mathcal{C}}|^{4}\parallel\partial_{x}^{5}\mathtt{u}\left(\cdot,\, t\right)-\partial_{x}^{5}\mathtt{w}\parallel_{L^{\infty}(B_{\rho|\hat{X}_{\mathcal{C}}|}^{n})}\}\,\,\leq\frac{1}{16}
\]
We call $t\mapsto X(x,\, t)=\hat{X}_{\mathcal{C}}+\left(x,\,\mathtt{u}\left(x,\, t\right)\right)$
to be the ``vertical parametrization'' of the flow $\left\{ \Sigma_{t}\right\} _{-1\leq t<0}$.
By (\ref{6}), (\ref{9}) and $0<\rho\leq\frac{1}{8}$,
we have 
\[
\frac{3}{4}|\hat{X}_{\mathcal{C}}|\,\leq\,|X\left(x,\, t\right)|\,=\,|\hat{X}_{\mathcal{C}}+\left(x,\,\mathtt{u}\left(x,\, t\right)\right)|\,\leq\,\frac{5}{4}|\hat{X}_{\mathcal{C}}|
\]
for $x\in B_{\rho\mid\hat{X}_{\mathcal{C}}\mid}^{n}\subset T_{\hat{X}_{\mathcal{C}}}\mathcal{C}$, $t\in[-1,\,0)$;
that is, $|X|$ is comparable with $|\hat{X}_{\mathcal{C}}|$. Note that
we still have the polynomial volume growth for each $\Sigma_{t}$:
\begin{equation}
\mathcal{H}^{n}\left(\Sigma_{t}\cap\left(B_{2\mathrm{r}}\setminus\bar{B}_{\mathrm{r}}\right)\right)\leq C\left(n,\,\mathcal{C}\right)\mathrm{r}^{n}\label{10}
\end{equation}
for all $\mathrm{r}\geq R$.

On the other hand, we could use the $F$ self-shrinker condition to improve (\ref{9}). 
To see this, observe that under the conditions
of being a $F$ self-shrinker and having a tangent cone $\mathcal{C}$
at infinity, the rescaled flow $\left\{ \Sigma_{t}=\sqrt{-t}\,\Sigma\right\} _{-1\leq t<0}$
moves by $F$ curvature vector and converges (in the locally $C^{5}$
sense) to the cone $\mathcal{C}$ as $t\nearrow0$. In other words,
we could define a $F$ curvature flow $\left\{ \Sigma_{t}\right\} _{-1\leq t\leq0}$
with $\Sigma_{t}=\sqrt{-t}\,\Sigma$ for $t\in[-1,\,0)$ and $\Sigma_{0}=\mathcal{C}$,
which is continuous upto $t=0$ (in the locally $C^{5}$ sense). Besides,
near each $\hat{X}_{\mathcal{C}}\in\mathcal{C}\setminus\bar{B}_{R}$
(with $R\gg1$ depending on $\Sigma$, $\mathcal{C}$), we have the
vertical parametrization of the flow (as above) for $t\in\left[-1,\,0\right]$.
by Definition \ref{d4}, the evolution of $u_{t}$ satisfies
\begin{equation}
\partial_{t}\mathtt{u}=\sqrt{1+|\partial_{x}\mathtt{u}|{}^{2}}\, F\left(A_{i}^{j}\left(x,\, t\right)\right)\label{11}
\end{equation}
for $\left(x_{1},\cdots,\, x_{n}\right)\in B_{\rho|\hat{X}_{\mathcal{C}}|}^{n}\subset T_{\hat{X}_{\mathcal{C}}}\mathcal{C}$, $-1\leq t<0$, and
\begin{equation}
\mathtt{u}\left(\cdot,\, t\right)\overset{C^{5}}{\longrightarrow}w\quad\textrm{on}\; B_{\rho\hat{|X}_{\mathcal{C}}|}^{n}\quad\textrm{as}\; t\nearrow0\label{12}
\end{equation}
where the shape operator $A_{t}^{\#}\left(x\right)\sim A_{i}^{j}\left(x,\, t\right)$
of $\Sigma_{t}$ (with respect to the local coordinate frame $\left\{ \partial_{1}X_{t},\cdots,\,\partial_{n}X_{t}\right\} $)
is equal to 
\begin{equation}
A_{i}^{j}\left(x,\, t\right)=\partial_{i}\left(\frac{\partial_{j}\mathtt{u}\left(x,\, t\right)}{\sqrt{1+|\partial_{x}\mathtt{u}|{}^{2}}}\right)\label{13}
\end{equation}\\
It follows (by (\ref{11}), (\ref{9}), (\ref{6}), (\ref{7}), (\ref{13})) that 
\[
|\partial_{t}\mathtt{u}|\,=\,|\hat{X}_{\mathcal{C}}|{}^{-1}\sqrt{1+|\partial_{x}\mathtt{u}|{}^{2}}\,\Big|F\left(|\hat{X}_{\mathcal{C}}|A_{i}^{j}\left(x,\, t\right)\right)\Big|
\]
\[
\leq\,|\hat{X}_{\mathcal{C}}|^{-1}\left(1+\parallel\partial_{x}\mathtt{u}_{t}\parallel_{L^{\infty}(B_{\rho\hat{|X}_{\mathcal{C}}|}^{n})}\right)\parallel F\parallel_{L^{\infty}(U)}
\]
in which we use the homogeneity of $F$. \\
Similarly, by differentiating
(\ref{11}) and using the homogeneity of $F$ (and its derivatives), we get

\begin{equation}
\{ |\hat{X}_{\mathcal{C}}|\parallel\partial_{t}\mathtt{u}\left(\cdot,\, t\right)\parallel_{L^{\infty}(B_{\rho\hat{|X}_{\mathcal{C}}|}^{n})}\,+\,|\hat{X}_{\mathcal{C}}|^{2}\parallel\partial_{t}\partial_{x}\mathtt{u}\left(\cdot,\, t\right)\parallel_{L^{\infty}(B_{\rho\hat{|X}_{\mathcal{C}}|}^{n})}\label{14}
\end{equation}
\[
+\,|\hat{X}_{\mathcal{C}}|^{3}\parallel\partial_{t}\partial_{x}^{2}\mathtt{u}\left(\cdot,\, t\right)\parallel_{L^{\infty}(B_{\rho\hat{|X}_{\mathcal{C}}|}^{n})}
\]
\[
+\,|\hat{X}_{\mathcal{C}}|^{4}\parallel\partial_{t}\partial_{x}^{3}\mathtt{u}\left(\cdot,\, t\right)\parallel_{L^{\infty}(B_{\rho\hat{|X}_{\mathcal{C}}|}^{n})}\} \,\leq C\left(n,\,\mathcal{C},\,\parallel F\parallel_{C^{3}\left(U\right)}\right)
\]
which implies (by (\ref{14}) and (\ref{11})) 
\[
|\mathtt{u}\left(\cdot,\, t\right)-w|\,\leq\int_{t}^{0}|\partial_{t}\mathtt{u}\left(\cdot,\,\tau\right)|\,d\tau\,\,\leq C\left(n,\,\mathcal{C},\,\parallel F\parallel_{C^{3}\left(U\right)}\right)\,|\hat{X}_{\mathcal{C}}|^{-1}\left(-t\right)
\]
Likewise, integrate the estimates for derivatives in (\ref{14}) to
get: $\forall\, t\in\left[-1,\,0\right]$

\begin{equation}
\{|\hat{X}_{\mathcal{C}}|\parallel\mathtt{u}\left(\cdot,\, t\right)-\mathtt{w}\parallel_{L^{\infty}(B_{\rho\hat{|X}_{\mathcal{C}}|}^{n})}+\,|\hat{X}_{\mathcal{C}}|^{2}\parallel\partial_{x}\mathtt{u}\left(\cdot,\, t\right)-\partial_{x}\mathtt{w}\parallel_{L^{\infty}(B_{\rho\hat{|X}_{\mathcal{C}}|}^{n})}\label{15}
\end{equation}
\[
+\,|\hat{X}_{\mathcal{C}}|^{3}\parallel\partial_{x}^{2}\mathtt{u}\left(\cdot,\, t\right)-\partial_{x}^{2}\mathtt{w}\parallel_{L^{\infty}(B_{\rho\hat{|X}_{\mathcal{C}}|}^{n})}
\]
\[
+\,|\hat{X}_{\mathcal{C}}|^{4}\parallel\partial_{x}^{3}\mathtt{u}\left(\cdot,\, t\right)-\partial_{x}^{3}\mathtt{w}\parallel_{L^{\infty}(B_{\rho\hat{|X}_{\mathcal{C}}|}^{n})}\}\,\,\leq C\left(n,\,\mathcal{C},\,\parallel F\parallel_{C^{3}\left(U\right)}\right)\left(-t\right)
\]
which is the improvement of (\ref{9}) by using the equation (\ref{11}). 

In view of the pull-back metric 
$$
g_{ij}\left(x,\, t\right)=\delta_{ij}+\partial_{i}u\left(x,\, t\right)\,\partial_{j}u\left(x,\, t\right)
$$
and the associated Christoffel symbols 
\begin{equation}
\Gamma_{ij}^{k}\left(x,\, t\right)=\frac{\partial_{k}\mathtt{u}\left(x,\, t\right)\,\partial_{ij}^{2}\mathtt{u}\left(x,\, t\right)}{1+|\partial_{x}\mathtt{u}\left(x,\, t\right)|{}^{2}}\label{16}
\end{equation}
together with (\ref{13}), (\ref{15}), the comparability of $|X|$
and $|\hat{X}_{\mathcal{C}}|$, (\ref{4}), (\ref{5}) and the continuity
and homogeneity of $F$ (and its derivatives), there exits $R\geq1$
(depending on $\Sigma,\mathcal{\, C},\, U,\,\parallel F\parallel_{C^{3}\left(U\right)},\,\lambda,\,\varkappa$)
such that for $X_{t}\in\Sigma_{t}\setminus\bar{B}_{R}$, the following
hold:
\begin{equation}
|X_{t}|\, A_{t}^{\#}\in U\label{17}
\end{equation}

\begin{equation}
\frac{\lambda}{2}\delta_{j}^{i}\,\leq\frac{\partial F}{\partial S_{i}^{j}}\left(A_{t}^{\#}\right)=\frac{\partial F}{\partial S_{i}^{j}}\left(|X_{t}|\, A_{t}^{\#}\right)\leq\,\frac{2}{\lambda}\delta_{j}^{i}\label{18}
\end{equation}

\noindent \resizebox{1.0\linewidth}{!}{
 \begin{minipage}{\linewidth}
  \begin{align}
|X_{t}|\Big|\sum_{k,\, l}\frac{\partial^{2}F}{\partial S_{i}^{j}\partial S_{k}^{l}}\left(A_{t}^{\#}\right)\,\left(\nabla_{\Sigma_{t}}A_{t}^{\#}\right)_{k}^{l}\Big|\, 
=\,\Big|\sum_{k,\, l}\frac{\partial^{2}F}{\partial S_{i}^{j}\partial S_{k}^{l}}\left(|X_{t}|A_{t}^{\#})\right)\cdot\left(|X_{t}|^{2}\nabla_{\Sigma_{t}}A_{t}^{\#}\right)_{k}^{l}\Big|\, 
\leq\,2\varkappa  \label{19}
\end{align}
  \end{minipage}
}

\begin{equation}
|X_{t}|\,|A_{t}^{\#}|\,+\,|X_{t}|^{2}|\nabla_{\Sigma_{t}}A_{t}^{\#}|\,+\,|X_{t}|^{3}\mid|\nabla_{\Sigma_{t}}^{2}A_{t}^{\#}|\,\leq C\left(n,\,\mathcal{C}\right)\label{20}
\end{equation}
where $A_{t}^{\#}$ is the shape operator of $\Sigma_{t}$ at $X_{t}$
and $\nabla_{\Sigma_{t}}A_{t}^{\#}$ is the covariant derivative of
$A_{t}^{\#}$ (with respect to $\Sigma_{t}$). Note that $F$ is
homogeneous of degree $1$, $\frac{\partial F}{\partial S_{i}^{j}}$
is of degree $0$ and $\frac{\partial^{2}F}{\partial S_{i}^{j}\partial S_{k}^{l}}$
is of degree $-1$.

Now let $\tilde{\Sigma}^n$ to be any other $F$ self-shrinker which is also
$C^{5}$ asymptotic to $\mathcal{C}$ at infinity. By the same limiting
behavior, $\tilde{\Sigma}$ is $C^{5}$ close to $\Sigma$ (in the
blow-down sense), and hence it can be regarded as a normal
graph of a function $h$ over $\Sigma$ outside a large ball $B^{n+1}_R$. Later we would derive
an elliptic equation which is satisfied by $h$. To this end, we need
the following two lemmas (Lemma \ref{l8} \& Lemma  \ref{l10}). The first one gives
the decay rate of the function $h$ and the difference of the shape
operators between $\Sigma$ and $\tilde{\Sigma}$ as $|X|\nearrow\infty$;
in the second lemma, we estimate the coefficients of the differential
equation to be satisfied by $h$. 
\begin{lem}\label{l8}
There exits $R=R\left(\Sigma,\,\tilde{\Sigma},\, n,\,\mathcal{C},\,\parallel F\parallel_{C^{3}\left(U\right)}\right)\geq1$
so that outside a compact set, $\tilde{\Sigma}$ is a normal graph
over $\Sigma\setminus\bar{B}_{R}$ and can be parametrized as 
\[
\tilde{X}=X+hN\,\,\,\textrm{for}\,\, X\in\Sigma\setminus\bar{B}_{R}
\]
where $N$ is the unit-normal of $\Sigma$ and $h$ is the
deviation of $\tilde{\Sigma}$ from $\Sigma$. Besides, there hold

\noindent \resizebox{1.0\linewidth}{!}{
  \begin{minipage}{\linewidth}
  \begin{align}
\parallel|X|\, h\parallel_{L^{\infty}(\Sigma\setminus\bar{B}_{R})}+\parallel|X|{}^{2}\nabla_{\Sigma}h\parallel_{L^{\infty}(\Sigma\setminus\bar{B}_{R})}+\parallel|X|^{3}\nabla_{\Sigma}^{2}h\parallel_{L^{\infty}(\Sigma\setminus\bar{B}_{R})}\,\leq C\left(n,\,\mathcal{C},\,\parallel F\parallel_{C^{3}\left(U\right)}\right)\label{21}
\end{align}
  \end{minipage}
}

\noindent \resizebox{1.0\linewidth}{!}{
  \begin{minipage}{\linewidth}
  \begin{align}
\parallel|X|{}^{3}\left(\tilde{A}^{\#}-A^{\#}\right)\parallel_{L^{\infty}(\Sigma\setminus\bar{B}_{R})}+\parallel|X|^{4}\left(\nabla_{\Sigma}\tilde{A}^{\#}-\nabla_{\Sigma}A^{\#}\right)\parallel_{L^{\infty}(\Sigma\setminus\bar{B}_{R})}\,\leq C\left(n,\,\mathcal{C},\,\parallel F\parallel_{C^{3}\left(U\right)}\right)\label{22}
\end{align}
  \end{minipage}
}

\begin{equation}
\parallel|X|^{3}\nabla_{\Sigma}^{2}\tilde{A}^{\#}\parallel_{L^{\infty}(\Sigma\setminus\bar{B}_{R})}\,\leq C\left(n,\,\mathcal{C},\,\parallel F\parallel_{C^{3}\left(U\right)}\right)\label{23}
\end{equation}
where $\tilde{A}^{\#}$ is the shape operator of $\tilde{\Sigma}$
at $\tilde{X}=X+hN$ and $\nabla_{\Sigma}\tilde{A}^{\#}$ is the covariant
derivative of $\tilde{A}^{\#}$ (which can be regarded as a 2-tensor
on $\Sigma$ via the normal graphic parametrization) with respect
to $\Sigma$. \end{lem}
\begin{proof}
Choose $R\gg1$ (depending on $\Sigma$, $\tilde{\Sigma}$, $n$,
$\mathcal{C}$, $\parallel F\parallel_{C^{3}\left(U\right)}$) so
that $\Sigma\setminus\bar{B}_{R}$ and $\tilde{\Sigma}\setminus\bar{B}_{R}$
have the local graph coordinates over tangent hyperplanes of $\mathcal{C}$
with appropriate estimates for the graphs as before. That is, for
each $\hat{X}\in\Sigma\setminus\bar{B}_{R}$, we can respectively
parametrize $\Sigma$ and $\tilde{\Sigma}$ locally (near $\Pi\left(\hat{X}\right)=\hat{X}_{\mathcal{C}}\in\mathcal{C}$)
by 

\[
X=X(x)\equiv\Pi\left(\hat{X}\right)+\left(x,\,\mathtt{u}(x)\right)
\]
\[
\tilde{X}=\tilde{X}(x)\equiv\Pi\left(\hat{X}\right)+\left(x,\,\tilde{\mathtt{u}}(x)\right)
\]
for $x=\left(x_{1},\cdots,\, x_{n}\right)\in B_{\rho|\Pi\left(\hat{X}\right)|}^{n}\subset T_{\Pi\left(\hat{X}\right)}\mathcal{C}$,
which satisfy (by ($\ref{6}$), (\ref{7}), (\ref{8}) and the comparability
of $|\hat{X}|$ and $|\hat{X}_{\mathcal{C}}|$)
\[
\{|\hat{X}|^{-1}\parallel\mathtt{u}\parallel_{L^{\infty}(B_{\rho|\Pi\left(\hat{X}\right)|}^{n})}+\parallel\partial_{x}\mathtt{u}\parallel_{L^{\infty}(B_{\rho|\Pi\left(\hat{X}\right)|}^{n})}+\,|\hat{X}|\parallel\partial_{x}^{2}\mathtt{u}\parallel_{L^{\infty}(B_{\rho|\Pi\left(\hat{X}\right)|}^{n})}+\cdots
\]
\begin{equation}
+\,|\hat{X}|{}^{4}\parallel\partial_{x}^{5}\mathtt{u}\parallel_{L^{\infty}(B_{\rho|\Pi\left(\hat{X}\right)|}^{n})}\}\,\,\leq C\left(n,\,\mathcal{C}\right)\label{24}
\end{equation}
\[
\{|\hat{X}|^{-1}\parallel\tilde{\mathtt{u}}\parallel_{L^{\infty}(B_{\rho|\Pi\left(\hat{X}\right)|}^{n})}+\parallel\partial_{x}\tilde{\mathtt{u}}\parallel_{L^{\infty}(B_{\rho|\Pi\left(\hat{X}\right)|}^{n})}+\,|\hat{X}|\parallel\partial_{x}^{2}\tilde{\mathtt{u}}\parallel_{L^{\infty}(B_{\rho|\Pi\left(\hat{X}\right)|}^{n})}+\cdots
\]
\begin{equation}
+\,|\hat{X}|^{4}\parallel\partial_{x}^{5}\tilde{\mathtt{u}}\parallel_{L^{\infty}(B_{\rho|\Pi\left(\hat{X}\right)|}^{n})}\}\,\,\leq C\left(n,\,\mathcal{C}\right)\label{25}
\end{equation}
Also, by applying the triangle inequality to (\ref{15}), we get
\[
\{|\hat{X}|\parallel\tilde{\mathtt{u}}-\mathtt{u}\parallel_{L^{\infty}(B_{\rho|\Pi\left(\hat{X}\right)|}^{n})}\,+|\hat{X}|^{2}\parallel\partial_{x}\tilde{\mathtt{u}}-\partial_{x}\mathtt{u}\parallel_{L^{\infty}(B_{\rho|\Pi\left(\hat{X}\right)|}^{n})}\,+|\hat{X}|^{3}\parallel\partial_{x}^{2}\tilde{\mathtt{u}}-\partial_{x}^{2}\mathtt{u}\parallel
\]
\begin{equation}
+|\hat{X}|^{4}\parallel\partial_{x}^{3}\tilde{\mathtt{u}}-\partial_{x}^{3}\mathtt{u}\parallel_{L^{\infty}(B_{\rho|\Pi\left(\hat{X}\right)|}^{n})}\}\,\,\leq C\left(n,\,\mathcal{C},\,\parallel F\parallel_{C^{3}\left(U\right)}\right)\label{26}
\end{equation}
By (\ref{26}), we may assume that $\tilde{\Sigma}$ is a normal graph
of $h$ defined on $\Sigma\setminus\bar{B}_{R}$; that is, for each
$x\in B_{\frac{\rho}{2}\mid\Pi\left(\hat{X}\right)\mid}^{n}\subset T_{\Pi\left(\hat{X}\right)}\mathcal{C}$, there
is a unique $y\in B_{\rho\mid\Pi\left(\hat{X}\right)\mid}^{n}\subset T_{\Pi\left(\hat{X}\right)}\mathcal{C}$ such
that

\begin{equation}
\Pi\left(\hat{X}\right)\,+\,\left(x,\,\mathtt{u}\left(x\right)\right)\,+\, h(x)\frac{\left(-\partial_{x}\mathtt{u},\,1\right)}{\sqrt{1+|\partial_{x}\mathtt{u}|{}^{2}}}=\Pi\left(\hat{X}\right)\,+\,\left(y,\,\tilde{\mathtt{u}}\left(y\right)\right)\label{27}
\end{equation}
or equivalently,
\[
\left(x-h(x)\frac{\partial_{x}\mathtt{u}}{\sqrt{1+|\partial_{x}\mathtt{u}|{}^{2}}},\:\mathtt{u}\left(x\right)+\frac{h(x)}{\sqrt{1+|\partial_{x}\mathtt{u}|{}^{2}}}\right)=\left(y,\,\tilde{\mathtt{u}}\left(y\right)\right)
\]
where $\frac{\left(-\partial_{x}\mathtt{\mathtt{u}},\,1\right)}{\sqrt{1+|\partial_{x}\mathtt{\mathtt{u}}|^{2}}}$
is the unit normal $N$ of $\Sigma$ at $X(x)=\Pi\left(\hat{X}\right)+\left(x,\,\mathtt{u}\left(x\right)\right)$.
In other words, $h$ is defined implicitly by the following equation
\begin{equation}
\tilde{\mathtt{u}}\left(\psi(x)\right)-\left(\mathtt{u}+\frac{h(x)}{\sqrt{1+|\partial_{x}\mathtt{u}|{}^{2}}}\right)=0\label{28}
\end{equation}
where 
\begin{equation}
\psi(x)=x-h(x)\frac{\partial_{x}\mathtt{u}}{\sqrt{1+|\partial_{x}\mathtt{u}|{}^{2}}}\label{29}
\end{equation}
defines a map from $B_{\frac{\rho}{2}|\Pi\left(\hat{X}\right)|}^{n}\subset T_{\Pi\left(\hat{X}\right)}\mathcal{C}$
into $B_{\rho|\Pi\left(\hat{X}\right)|}^{n}\subset T_{\Pi\left(\hat{X}\right)}\mathcal{C}$. Since $|h\left(x\right)|$
stands for the distance from the point in (\ref{27}):
$$
\tilde{X}\left(\psi(x)\right)=\Pi\left(\hat{X}\right)+\left(\psi(x),\,\tilde{u}\left(\psi(x)\right)\right)
$$
to $\Sigma$, we
immediately have 
\[
|h(x)|\,\leq\,|\tilde{\mathtt{u}}\left(\psi(x)\right)-\mathtt{u}(\psi(x))|\,\leq C\left(n,\,\mathcal{C},\,\parallel F\parallel_{C^{3}\left(U\right)}\right)|\hat{X}|{}^{-1}
\]

To proceed further, notice that for the unit normal vectors
of $\Sigma$ and $\tilde{\Sigma}$: 
\begin{equation}
N\left(x\right)=\frac{\left(-\partial_{x}\mathtt{u},\,1\right)}{\sqrt{1+|\partial_{x}\mathtt{u}|{}^{2}}},\quad\tilde{N}\left(x\right)=\frac{\left(-\partial_{x}\tilde{\mathtt{u}},\,1\right)}{\sqrt{1+|\partial_{x}\tilde{\mathtt{u}}|{}^{2}}}\label{30}
\end{equation}
we may assume (by (\ref{26}), (\ref{24})) that 
\[
\parallel\tilde{N}-N\parallel_{L^{\infty}(B_{\rho|\Pi\left(\hat{X}\right)|}^{n})}+\parallel N\circ\psi-N\parallel_{L^{\infty}(B_{\frac{\rho}{2}|\Pi\left(\hat{X}\right)|}^{n})}\,\,\leq\frac{1}{3}
\]
which implies that for each $x\in B_{\frac{\rho}{2}|\Pi\left(\hat{X}\right)|}^{n}\subset T_{\Pi\left(\hat{X}\right)}\mathcal{C}$,
\begin{equation}
\tilde{N}\left(\psi(x)\right)\cdot N(x)\geq\, N(x)\cdot N(x)\,-\,|\tilde{N}\left(\psi(x)\right)-N(x)|\,|N(x)|\label{31}
\end{equation}
\[
\geq1\,-\,\left(\tilde{|N}\left(\psi(x)\right)-N\left(\psi(x)\right)|\,+\,|N\left(\psi(x)\right)-N(x)|\right)\,\geq\frac{2}{3}
\]
Let

\[
\Theta\left(x,\, s\right)=\tilde{\mathtt{u}}\left(x-s\frac{\partial_{x}\mathtt{u}}{\sqrt{1+\mid\partial_{x}\mathtt{u}\mid^{2}}}\right)-\left(\mathtt{u}+\frac{s}{\sqrt{1+\mid\partial_{x}\mathtt{u}\mid^{2}}}\right)
\]
then by (\ref{28}), (\ref{29}) and (\ref{31}), we have $\Theta\left(x,\, h(x)\right)=0$
and 
\[
\partial_{s}\Theta\left(x,\, h(x)\right)=-\sqrt{1+\mid\partial_{y}\tilde{\mathtt{u}}\left(\psi(x)\right)\mid^{2}}\;\tilde{N}\left(\psi(x)\right)\cdot N(x)\leq-\frac{2}{3}
\]
Therefore, by the implicit function theorem, we have $h\in C^{2}\left(B_{\frac{\rho}{2}\mid\Pi\left(\hat{X}\right)\mid}^{n}\right)$.
Besides, by doing the implicit differentiation of (\ref{28}) (or equivalently
$\Theta\left(x,\, h(x)\right)=0$), we get
\begin{equation}
\frac{1+\partial_{j}\tilde{\mathtt{u}}\circ\psi\cdot\partial_{j}\mathtt{u}}{\sqrt{1+|\partial_{x}\mathtt{u}|{}^{2}}}\partial_{i}h\,=\,\left(\partial_{i}\tilde{\mathtt{u}}\circ\psi-\partial_{i}\mathtt{u}\right)\label{32}
\end{equation}
\[
-\left(\partial_{j}\tilde{\mathtt{u}}\circ\psi\cdot\partial_{i}\frac{\partial_{j}\mathtt{u}}{\sqrt{1+|\partial_{x}\mathtt{u}|{}^{2}}}+\partial_{j}\mathtt{u}\frac{\partial_{ij}^{2}\mathtt{u}}{\left(1+|\partial_{x}\mathtt{u}|{}^{2}\right)^{\frac{3}{2}}}\right)h
\]
in which we sum over repeated indices. Note that we can use (\ref{32}),
together with (\ref{24}) and (\ref{26}), to estimate $\partial_{x}h$.
For instance, for the first term on the RHS of the equation, we have
\[
|\partial_{i}\tilde{\mathtt{u}}\circ\psi-\partial_{i}\mathtt{u}|\,\leq\,|\partial_{i}\tilde{\mathtt{u}}\circ\psi-\partial_{i}\mathtt{u}\circ\psi|\,+|\partial_{i}\mathtt{u}\circ\psi-\partial_{i}\mathtt{u|}
\]

\noindent \resizebox{1.009\linewidth}{!}{
  \begin{minipage}{\linewidth}
  \begin{align*}
\leq C\left(n,\,\mathcal{C},\,\parallel F\parallel_{C^{3}\left(U\right)}\right)\,|\hat{X}|^{-2}+\,\sum_{j}\int_{0}^{1}\Big|\partial_{ij}^{2}\mathtt{u}\left(x-\theta h\frac{\partial_{x}\mathtt{u}}{\sqrt{1+\mid\partial_{x}\mathtt{u}\mid^{2}}}\right)\Big|\, d\theta\;\frac{|\partial_{j}\mathtt{u}|}{\sqrt{1+|\partial_{x}\mathtt{u}|{}^{2}}}|h|
\end{align*}
  \end{minipage}
}

\[
\leq C\left(n,\,\mathcal{C},\,\parallel F\parallel_{C^{3}\left(U\right)}\right)\,|\hat{X}|{}^{-2}
\]
Thus we get 
\[
\parallel\partial_{x}h\parallel_{L^{\infty}(B_{\frac{\rho}{2}|\Pi\left(\hat{X}\right)|}^{n})}\,\leq C\left(n,\,\mathcal{C},\,\parallel F\parallel_{C^{3}\left(U\right)}\right)\,|\hat{X}|{}^{-2}
\]
Similarly, doing the implicit differentiation of (\ref{32}) and using
(\ref{24}) and (\ref{26}) yields 
\[
\parallel\partial_{x}^{2}h\parallel_{L^{\infty}(B_{\frac{\rho}{2}|\Pi\left(\hat{X}\right)|}^{n})}\,\leq C\left(n,\,\mathcal{C},\,\parallel F\parallel_{C^{3}\left(U\right)}\right)\,|\hat{X}|{}^{-3}
\]
The bounds on the covariant derivatives of $h$ follow from the the
following estimates on the pull-back metric $g_{ij}=\partial_{i}X\cdot\partial_{j}X$
and the Christoffel symbols $\Gamma_{ij}^{k}$ in (\ref{16}) associated
with the local coordinates $x=\left(x_{1},\cdots,\, x_{n}\right)$:
\begin{equation}
\delta_{ij}\leq\, g_{ij}=1+\partial_{i}\mathtt{u}\,\partial_{j}\mathtt{u}\,\leq\frac{5}{4}\delta_{ij}\label{33}
\end{equation}
\begin{equation}
|\Gamma_{ij}^{k}|\,=\frac{|\partial_{k}\mathtt{u}|}{1+|\partial_{x}\mathtt{u}|{}^{2}}|\partial_{ij}^{2}\mathtt{u}|\,\leq C\left(n,\,\mathcal{C},\, F\right)\,|\hat{X}|{}^{-1}\label{34}
\end{equation}
where we have used (\ref{24}). This completes the derivation of (\ref{21}).

As for (\ref{22}), notice that the normal graph reparametrization
of $\tilde{\Sigma}$ amounts to the following change of variables:
\begin{equation}
\tilde{X}=\Pi\left(\hat{X}\right)+\left(y,\,\tilde{\mathtt{u}}(y)\right)\quad\textrm{with}\quad y=\psi(x)=x-h(x)\frac{\partial_{x}\mathtt{u}}{\sqrt{1+|\partial_{x}\mathtt{u}|{}^{2}}}\label{35}
\end{equation}
So from (\ref{35}), (\ref{24}) and (\ref{21}), we have
\begin{equation}
\frac{\partial y_{k}}{\partial x_{i}}\,=\delta_{i}^{k}-\, h\cdot\partial_{x_{i}}\left(\frac{\partial_{x_{j}}\mathtt{u}}{\sqrt{1+|\partial_{x}\mathtt{u}|{}^{2}}}\right)-\,\partial_{x_{i}}h\frac{\partial_{k}\mathtt{u}}{\sqrt{1+|\partial_{x}\mathtt{u}|{}^{2}}}=\,\delta_{i}^{k}+\, O\left(|\hat{X}|{}^{-2}\right)\label{36}
\end{equation}
By taking $R$ sufficiently large, we may assume that $\psi:\, B_{\frac{\rho}{2}|\Pi\left(\hat{X}\right)|}^{n}\rightarrow\textrm{Im}\psi\subset B_{\rho|\Pi\left(\hat{X}\right)|}^{n}$
is a $C^{2}$ diffeomorphism and the inverse of $\frac{\partial y_{k}}{\partial x_{i}}$
satisfies
\[
\frac{\partial x_{i}}{\partial y_{k}}=\delta_{k}^{i}+\, O\left(|\hat{X}|{}^{-2}\right)
\]
It follows that the components of shape operators $\tilde{A}^{\#}$
of $\tilde{\Sigma}$ and $A^{\#}$ of $\Sigma$ with respect to the
local coordinates $x=\left(x_{1},\cdots,\, x_{n}\right)$ are respectively
equal to 
\begin{equation}
\tilde{A}_{i}^{j}=\frac{\partial y_{k}}{\partial x_{i}}\frac{\partial x_{j}}{\partial y_{l}}\,\partial_{y_{k}}\left(\frac{\partial_{y_{l}}\tilde{\mathtt{u}}}{\sqrt{1+|\partial_{y}\tilde{\mathtt{u}}|{}^{2}}}\right)\Big|_{y=\varphi(x)},\quad A_{i}^{j}=\partial_{x_{i}}\left(\frac{\partial_{x_{j}}\mathtt{u}}{\sqrt{1+|\partial_{x}\mathtt{u}|{}^{2}}}\right)\label{37}
\end{equation}
in which we sum over repeated indices. Using the triangle inequality,
combined with (\ref{24}), (\ref{26}), (\ref{35}), (\ref{21}) and
(\ref{36}), we then get from (\ref{37}) that 
\[
\Big|\tilde{A}_{i}^{j}-A_{i}^{j}\Big|\,\leq C\left(n,\,\mathcal{C},\,\parallel F\parallel_{C^{3}\left(U\right)}\right)\,|\hat{X}|{}^{-3}
\]
Due to (\ref{33}), the above implies that
\[
\Big|\tilde{A}^{\#}-A^{\#}\Big|\,\leq C\left(n,\,\mathcal{C},\,\parallel F\parallel_{C^{3}\left(U\right)}\right)\,|\hat{X}|{}^{-3}
\]
Also, in view of $\nabla_{\Sigma}\tilde{A}^{\#}\sim\nabla_{r}\tilde{A}_{i}^{j}$,
$\nabla_{\Sigma}A^{\#}\sim\nabla_{r}A_{i}^{j}$ and
\begin{equation}
\nabla_{r}\tilde{A}_{i}^{j}=\partial_{r}\tilde{A}_{i}^{j}-\,\Gamma_{ri}^{s}\tilde{A}_{s}^{j}+\,\Gamma_{rs}^{j}\tilde{A}_{i}^{s},\quad\nabla_{r}A_{i}^{j}=\partial_{r}A_{i}^{j}-\,\Gamma_{ri}^{s}A_{s}^{j}+\,\Gamma_{rs}^{j}A_{i}^{s}\label{38}
\end{equation}
(in which we sum over repeated indices), we can similarly derive
\[
|\nabla_{\Sigma}\tilde{A}^{\#}-\nabla_{\Sigma}A^{\#}|\,\leq C\left(n,\,\mathcal{C},\,\parallel F\parallel_{C^{3}\left(U\right)}\right)\,|\hat{X}|{}^{-4}
\]
This completes (\ref{22}). 

(\ref{23}) follows from taking one more derivative of (\ref{38})
and use (\ref{37}), (\ref{34}), (\ref{24}), (\ref{26}) and (\ref{33}).
\end{proof}
Next, we would like to define a 2-tensor $\mathbf{a}$ on $\Sigma$ (outside
a large ball), which would be served as the coefficients of the differential
equation to be satisfied by the deviation $h$. Note that by (\ref{17}),
Lemma  \ref{l8} (in particular (\ref{22})), we may assume that 
\begin{equation}
\left(1-\theta\right)|X|\, A^{\#}+\theta|X|\,\tilde{A}^{\#}\in U\qquad\forall\, X\in\Sigma\setminus\bar{B}_{R},\;\theta\in\left[0,\,1\right]\label{39}
\end{equation}
where $\tilde{A}^{\#}$ is the shape operator of $\tilde{\Sigma}$
at $\tilde{X}=X+hN$.
\begin{defn}\label{d9}
In the setting of Lemma \ref{l8}, let's take a local coordinate $x=\left(x_{1},\cdots,\, x_{n}\right)$
of $\Sigma$ (outside a larger ball) so that $\Sigma$ and $\tilde{\Sigma}$
can be respectively parametrized as
\[
X=X\left(x\right),\quad\tilde{X}\left(x\right)=X\left(x\right)+h\left(x\right)N\left(x\right)
\]
where $h\left(x\right)$is the deviation and $N\left(x\right)$ is
the unit-normal of $\Sigma$ at $X\left(x\right)$. Then we define 
$$\bar{\mathtt{a}}^{ij}\left(x\right)=\sum_{k}\bar{\mathtt{a}}_{k}^{i}\left(x\right)\, g^{kj}\left(x\right)$$ 
$$\bar{\mathtt{a}}_{j}^{i}\left(x\right)=\int_{0}^{1}\frac{\partial F}{\partial S_{i}^{j}}\left(\left(1-\theta\right)|X|\, A^{\#}\left(x\right)+\theta|X|\,\tilde{A}^{\#}\left(x\right)\right)\, d\theta$$

\noindent and its symmetrization 
\[
\mathbf{a}^{ij}\left(x\right)=\frac{1}{2}\left(\bar{\mathtt{a}}^{ij}\left(x\right)+\bar{\mathtt{a}}^{ji}\left(x\right)\right)
\]
where $g^{ij}\left(x\right)$ is the inverse of the pull-back metric
$g_{ij}=\partial_{i}X\cdot\partial_{j}X$, $A^{\#}\left(x\right)\sim A_{i}^{j}\left(x\right)=-\partial_{i}N\cdot\partial_{j}X$
is the shape operator of $\Sigma$ at $X\left(x\right),$ $\tilde{A}_{t}^{\#}\left(x\right)\sim\tilde{A}_{i}^{j}\left(x,\, t\right)=-\partial_{i}\tilde{N}\cdot\partial_{j}\tilde{X}$
is the shape operator of $\tilde{\Sigma}$ at $\tilde{X}\left(x\right)$
with $\tilde{N}\left(x\right)$ being the unit-normal of $\tilde{\Sigma}$
at $\tilde{X}\left(x\right)$.\\
Note that 
\[
\bar{\mathtt{a}}_{j}^{i}\left(x\right)=\int_{0}^{1}\frac{\partial F}{\partial S_{i}^{j}}\left(\left(1-\theta\right)|X|\, A^{\#}\left(x\right)+\theta|X|\,\tilde{A}^{\#}\left(x\right)\right)\, d\theta
\]
\[
=\int_{0}^{1}\frac{\partial F}{\partial S_{i}^{j}}\left(\left(1-\theta\right)A^{\#}\left(x\right)+\theta\tilde{A}^{\#}\left(x\right)\right)\, d\theta
\]
since $\frac{\partial F}{\partial S_{i}^{j}}$ is homogeneous of degree
0; besides, the operator $\mathbf{a}$ is independent of the choice
of local coordinates and hence defines a 2-tensor on $\Sigma$.
\end{defn}
We have the following estimates for the tensor $\mathbf{a}$, which
is based on (\ref{18}), (\ref{19}), (\ref{20}), (\ref{22}), (\ref{23})
and the homogeneity of $F$ and its derivatives.
\begin{lem}\label{l10}
There exits $R=R\left(\Sigma,\,\tilde{\Sigma},\mathcal{\, C},\, U,\,\parallel F\parallel_{C^{3}\left(U\right)},\,\lambda,\,\varkappa\right)\geq1$
such that
\begin{equation}
\frac{\lambda}{3}\leq\mathbf{a}\leq\frac{3}{\lambda}\label{40}
\end{equation}
\begin{equation}
|X|\Big|\nabla_{\Sigma}\mathbf{a}\Big|\,\leq3\varkappa\label{41}
\end{equation}
\begin{equation}
|X|{}^{2}\Big|\nabla_{\Sigma}^{2}\mathbf{a\Big|}\,\leq C\left(n,\,\mathcal{C},\,\parallel F\parallel_{C^{3}\left(U\right)}\right)\label{42}
\end{equation}
for all $X\in\Sigma\setminus\bar{B}_{R}$. \end{lem}
\begin{proof}
By (\ref{18}), (\ref{19}), (\ref{39}), (\ref{22}), the homogeneity
and continuity of $F$ (and its derivatives), there exists $R=R\left(\Sigma,\,\tilde{\Sigma},\mathcal{\, C},\, U,\,\parallel F\parallel_{C^{3}\left(U\right)},\,\lambda,\,\varkappa\right)\geq1$
such that
\[
\frac{\lambda}{3}\delta_{j}^{i}\,\leq\,\bar{\mathtt{a}}_{j}^{i}=\int_{0}^{1}\frac{\partial F}{\partial S_{i}^{j}}\left(\left(1-\theta\right)|X|\, A^{\#}+\theta|X|\,\tilde{A}^{\#}\right)\, d\theta\,\leq\frac{3}{\lambda}\delta_{j}^{i}
\]

\noindent \resizebox{1.009\linewidth}{!}{
  \begin{minipage}{\linewidth}
  \begin{align*}
|X|\Big|\nabla_{r}\bar{\mathtt{a}}_{i}^{j}\Big|\,=\,|X|\,\Big|\int_{0}^{1}\sum_{k,\, l}\frac{\partial^{2}F}{\partial S_{i}^{j}\partial S_{k}^{l}}\left(\left(1-\theta\right)A^{\#}+\theta\tilde{A}^{\#}\right)\cdot\left(\left(1-\theta\right)\nabla_{r}A_{k}^{l}+\theta\nabla_{r}\tilde{A}_{k}^{l}\right)\, d\theta\Big|
\end{align*}
  \end{minipage}
}

\noindent \resizebox{1.009\linewidth}{!}{
  \begin{minipage}{\linewidth}
  \begin{align*}
=\Big|\int_{0}^{1}\sum_{k,\, l}\frac{\partial^{2}F}{\partial S_{i}^{j}\partial S_{k}^{l}}\left(\left(1-\theta\right)|X|\, A^{\#}+\theta|X|\,\tilde{A}^{\#}\right)\cdot\left(\left(1-\theta\right)|X|{}^{2}\nabla_{r}A_{k}^{l}+\theta|X|^{2}\nabla_{r}\tilde{A}_{k}^{l}\right)\, d\theta\Big|\,\leq\,3\varkappa
\end{align*}
  \end{minipage}
}\\

Likewise, with the help of (\ref{20}), (\ref{23}), we get 
\[
|X|{}^{2}\Big|\nabla_{\Sigma}^{2}\bar{\mathtt{a}}\Big|\,\leq C\left(n,\,\mathcal{C},\,\parallel F\parallel_{C^{3}\left(U\right)}\right)
\]
The conclusion follows immediately.
\end{proof}

Now we are in a position to derive an equation for $h$.
\begin{prop}\label{p11}
There exits $R=R\left(\Sigma,\,\tilde{\Sigma},\mathcal{\, C},\, U,\,\parallel F\parallel_{C^{3}\left(U\right)},\,\lambda,\,\varkappa\right)\geq1$
such that the deviation $h$ satisfies
\begin{equation}
\nabla_{\Sigma}\cdot\left(\mathbf{a}\, dh\right)-\frac{1}{2}\left(X\cdot\nabla_{\Sigma}h\,-\, h\right)=O\left(|X|^{-1}\right)|\nabla_{\Sigma}h|\,+\, O\left(|X|^{-2}\right)|h|\label{43}
\end{equation}
for $X\in\Sigma\setminus\bar{B}_{R}$, where 
\[
\nabla_{\Sigma}\cdot\left(\mathbf{a}\, dh\right)=\sum_{i,\, j}\nabla_{i}\left(\mathbf{a}^{ij}\nabla_{j}h\right)
\]
and the notation $O\left(|X|{}^{-1}\right)$
means 
\[
\Big|O\left(|X|^{-1}\right)\Big|\,\leq C\left(n,\,\mathcal{C},\,\parallel F\parallel_{C^{3}\left(U\right)}\right)\,|X|{}^{-1}
\]
\end{prop}
\begin{proof}
Fix $\hat{X}\in\Sigma\setminus\bar{B}_{R}$ and take a local coordinate
$x=\left(x_{1},\cdots,\, x_{n}\right)$ of $\Sigma$ which is normal
and principal (w.r.t. $\Sigma$) at $\hat{X}=X\left(0\right)$. That
is 
\[
g_{ij}\Big|_{x=0}\,=\delta_{ij},\,\,\Gamma_{ij}^{k}\Big|_{x=0}=0,\,\, A_{i}^{j}\Big|_{x=0}=\kappa_{i}\delta_{ij}
\]
where $g_{ij}$ is the pull-back metric, $\Gamma_{ij}^{k}$ is the
Christoffel symbols and $A_{i}^{j}$ is the shape operator of $\Sigma$
at $X\left(x\right)$. Denote the principal direction of $\Sigma$
at $\hat{X}$ by
\[
\partial_{i}X\Big|_{x=0}\,=e_{i}
\]
Throughout the proof, we adopt the Einstein summation convention (i.e.
summing over repeated indices). Recall that we regard $\tilde{\Sigma}$
(outside a large ball) as a normal graph over $\Sigma\setminus\bar{B}_{R}$
and parametrize it by $\tilde{X}=X\left(x\right)+h(x)N(x)$. We then
want to compute some geometric quantities of $\tilde{\Sigma}$ in terms
of this local coordinate at $\tilde{X}\left(0\right)=\hat{X}+hN\Big|_{\hat{X}}$.
First, let's compute 
\[
\partial_{i}\tilde{X}\Big|_{x=0}\,=\left(\delta_{i}^{k}-A_{i}^{k}h\right)\partial_{k}X\,+\,\partial_{i}h\, N\Big|_{x=0}\,=\left(1-\kappa_{i}h\right)e_{i}+\nabla_{i}h\, N
\]
\begin{equation}
\partial_{ij}^{2}\tilde{X}\Big|_{x=0}\,=-\left(A_{i}^{k}\nabla_{j}h+A_{j}^{k}\nabla_{i}h+\nabla_{i}A_{j}^{k}\cdot h\right)e_{k}\,+\,\left(A_{ij}+\nabla_{ij}^{2}h-A_{ij}^{2}h\right)N\label{44}
\end{equation}
which (together with Lemma \ref{l8}) gives the metric of $\tilde{\Sigma}$,
its inverse and determinant as follows: 
\[
\tilde{g}_{ij}\Big|_{x=0}\,=\left(1-\kappa_{i}h\right)^{2}\delta_{ij}+\nabla_{i}h\,\nabla_{j}h=\left(1-\kappa_{i}h\right)^{2}\left(\delta_{ij}+\frac{\nabla_{i}h\,\nabla_{j}h}{\left(1-\kappa_{i}h\right)^{2}}\right)
\]
\begin{equation}
\tilde{g}^{ij}\Big|_{x=0}\,=\left(1-\kappa_{i}h\right)^{-2}\left(\delta_{ij}+\frac{\nabla_{i}h\,\nabla_{j}h}{\left(1-\kappa_{i}h\right)^{2}}\right)^{-1}\label{45}
\end{equation}
\[
=\left(1+2\kappa_{i}h\right)\delta^{ij}+\, O\left(|\hat{X}|^{-2}\right)|\nabla_{\Sigma}h|\,+O\left(|\hat{X}|^{-3}\right)|h|
\]
\[
\det\tilde{g}\Big|_{x=0}\,=\left(1-\kappa_{1}h\right)^{2}\cdots\left(1-\kappa_{n}h\right)^{2}\det\left(\delta_{ij}+\frac{\nabla_{i}h\,\nabla_{j}h}{\left(1-\kappa_{i}h\right)^{2}}\right)
\]
\[
=1-2Hh+\, O\left(|\hat{X}|^{-2}\right)|\nabla_{\Sigma}h|\,+O\left(|\hat{X}|^{-3}\right)|h|
\]
and also the unit-normal of $\tilde{\Sigma}$: 
\begin{equation}
\tilde{N}\Big|_{x=0}\,=\left(\det\tilde{g}\right)^{-\frac{1}{2}}\partial_{1}\tilde{X}\wedge\cdots\wedge\partial_{n}\tilde{X}\label{46}
\end{equation}
\[
=\left(\det\tilde{g}\right)^{-\frac{1}{2}}\left(-\sum_{i=1}^{n}\left(\nabla_{i}h\,\prod_{j\neq i}\left(1-\kappa_{j}h\right)\right)e_{i}\,+\,\left(1-\kappa_{1}h\right)\cdots\left(1-\kappa_{n}h\right)N\right)
\]
\[
=-\sum_{i=1}^{n}\left(1+\kappa_{i}h+O\left(|\hat{X}|^{-2}\right)|\nabla_{\Sigma}h|+O\left(|\hat{X}|^{-3}\right)|h|\right)\nabla_{i}h\cdot e_{i}
\]
\[
+\left(1+O\left(|\hat{X}|^{-2}\right)|\nabla_{\Sigma}h|+O\left(|\hat{X}|^{-3}\right)|h|\right)N
\]

By (\ref{44}), (\ref{45}), (\ref{46}) and Lemma \ref{l8}, we compute the
shape operator of $\tilde{\Sigma}$ at $\tilde{X}\left(0\right)$:

\begin{equation}
\tilde{A}_{i}^{j}\Big|_{x=0}\,=\tilde{A}_{ik\,}\tilde{g}^{kj}=\left(\partial_{ik}^{2}\tilde{X}\cdot\tilde{N}\right)\,\tilde{g}^{kj}\label{47}
\end{equation}

\noindent \resizebox{1.009\linewidth}{!}{
  \begin{minipage}{\linewidth}
  \begin{align*}
=\left(A_{ik}+\nabla_{ik}^{2}h+O\left(|\hat{X}|^{-2}\right)|\nabla_{\Sigma}h|\,+O\left(|\hat{X}|^{-2}\right)|h|\right)\left(\left(1+2\kappa_{j}h\right)\delta^{kj}+O\left(|\hat{X}|^{-2}\right)|\nabla_{\Sigma}h|\right)
\end{align*}
  \end{minipage}
}

\[
+\left(A_{ik}+\nabla_{ik}^{2}h+O\left(|\hat{X}|^{-2}\right)|\nabla_{\Sigma}h|\,+O\left(|\hat{X}|^{-2}\right)|h|\right)O\left(|\hat{X}|^{-3}\right)|h|
\]
\[
=A_{i}^{j}+\delta^{kj}\nabla_{ik}^{2}h+O\left(|\hat{X}|{}^{-2}\right)\left(|\nabla_{\Sigma}h|\,+\,|h|\right)
\]
and 
\begin{equation}
\tilde{X}\cdot\tilde{N}\Big|_{x=0}=X\cdot N\,-\, X\cdot\nabla_{\Sigma}h\,+h+O\left(|\hat{X}|{}^{-1}\right)|\nabla_{\Sigma}h|\,+O\left(|\hat{X}|{}^{-2}\right)|h|\label{48}
\end{equation}

Thus, in view of the $F$ self-shrinker equation satisfied by $\Sigma$
and $\tilde{\Sigma}$, we get
\begin{equation}
0=F\left(\tilde{A}^{\#}\right)-F\left(A^{\#}\right)\,+\,\frac{1}{2}\left(\tilde{X}\cdot\tilde{N}-X\cdot N\right)\Big|_{x=0}\label{49}
\end{equation}
\[
=\int_{0}^{1}\frac{\partial F}{\partial S_{i}^{j}}\left(\left(1-\theta\right)A^{\#}+\theta\tilde{A}^{\#}\right)\, d\theta\cdot\left(\tilde{A}_{i}^{j}-A_{i}^{j}\right)\,-\,\frac{1}{2}\left(X\cdot\nabla_{\Sigma}h-h\right)
\]
\[
+O\left(|\hat{X}|{}^{-1}\right)|\nabla_{\Sigma}h|\,+O\left(|\hat{X}|{}^{-2}\right)|h|
\]
\[
=\bar{\mathtt{a}}_{j}^{i}\delta^{jk}\nabla_{ik}^{2}h\,-\,\frac{1}{2}\left(X\cdot\nabla_{\Sigma}h\,-\, h\right)+O\left(|\hat{X}|{}^{-1}\right)|\nabla_{\Sigma}h|\,+O\left(|\hat{X}|{}^{-2}\right)|h|
\]
\[
=\bar{\mathtt{a}}^{ik}\nabla_{ik}^{2}h\,-\,\frac{1}{2}\left(X\cdot\nabla_{\Sigma}h\,-\, h\right)+O\left(|\hat{X}|{}^{-1}\right)|\nabla_{\Sigma}h|\,+O\left(|\hat{X}|{}^{-2}\right)|h|
\]
\[
=\langle\bar{\mathtt{a}},\,\nabla_{\Sigma}^{2}h\rangle\,-\,\frac{1}{2}\left(X\cdot\nabla_{\Sigma}h\,-\, h\right)+O\left(|\hat{X}|{}^{-1}\right)|\nabla_{\Sigma}h|\,+O\left(|\hat{X}|{}^{-2}\right)|h|
\]
Note that by the symmetry of the Hessian and Lemma  \ref{l10}, we have
\begin{equation}
\langle\bar{\mathtt{a}},\,\nabla_{\Sigma}^{2}h\rangle\,=\,\bar{\mathtt{a}}^{ij}\nabla_{ij}^{2}h=\frac{1}{2}\left(\bar{\mathtt{a}}^{ij}+\bar{\mathtt{a}}^{ji}\right)\nabla_{ij}^{2}h=\langle\mathbf{a},\,\nabla_{\Sigma}^{2}h\rangle\label{50}
\end{equation}
\[
=\nabla_{i}\left(\mathbf{a}^{ij}\nabla_{j}h\right)-\left(\nabla_{i}\mathbf{a}^{ij}\right)\nabla_{j}h=\nabla_{\Sigma}\cdot\left(\mathbf{a}\, dh\right)+\, O\left(|\hat{X}|{}^{-1}\right)|\nabla_{\Sigma}h|
\]
(\ref{43}) follows from combining (\ref{49}) and (\ref{50}).
\end{proof}
Our goal is to show that $h$ vanishes on $\Sigma\setminus\bar{B}_{R}$
for some $R\gg1$, which would be done in the next section through
Carleman's inequality. For that purpose, we first observe that for
each $t\in[-1,\,0)$, $\tilde{\Sigma}_{t}=\sqrt{-t}\tilde{\,\Sigma}$
is also a normal graph over $\Sigma_{t}\setminus\bar{B}_{R}$
and it can be parametrized as $\tilde{X}_{t}=X_{t}+h_{t}N_{t}$. For
the rest of this section, we would show that each $h_{t}=h\left(\cdot,\, t\right)$
satisfies a similar equation as that of $h\left(\cdot,\,-1\right)$ in
Proposition \ref{p11}. Due to the property that $\left\{ \Sigma_{t}\right\} _{-1\leq t<0}$
forms a $F$ curvature flow, it turns out that the evolution of $h_{t}$
satisfies a parabolic equation. We then give some estimates for the
coefficients of the parabolic equations (as in Lemma \ref{l10}) , which is
crucial for deriving the Carleman's inequality in the next section.

Now fix $t\in[-1,\,0)$ and define a 2-tensor $\mathbf{a}_{t}$ on
$\Sigma_{t}=\sqrt{-t}\,\Sigma$ as in Definition \ref{d9}. First, take a
local coordinate $x=\left(x_{1},\cdots,\, x_{n}\right)$ of $\Sigma_{t}$
(outside a large ball) so that $\Sigma_{t}$ and $\tilde{\Sigma}_{t}$
can be respectively parametrized as
\[
X_{t}=X_{t}\left(x\right),\quad\tilde{X}_{t}\left(x\right)=X_{t}\left(x\right)+h_{t}\left(x\right)N_{t}\left(x\right)
\]
We then define 
$$\bar{\mathtt{a}}_{t}^{ij}\left(x\right)=\sum_{k}\bar{\mathtt{a}}_{k}^{i}(x,\, t)\, g_{t}^{kj}\left(x\right)$$
$$\bar{\mathtt{a}}_{j}^{i}(x,\, t)=\int_{0}^{1}\frac{\partial F}{\partial S_{i}^{j}}\left(\left(1-\theta\right)A_{t}^{\#}\left(x\right)+\theta\tilde{A_{t}}^{\#}\left(x\right)\right)\, d\theta$$
and its symmetrization
\[
\mathbf{a}_{t}^{ij}\left(x\right)=\frac{1}{2}\left(\bar{\mathtt{a}}_{t}^{ij}\left(x\right)+\bar{\mathtt{a}}_{t}^{ji}\left(x\right)\right)
\]
where $g_{t}^{ij}\left(x\right)$ is the inverse of the pull-back
metric $g_{ij}\left(x,\, t\right)=\partial_{i}X_{t}\left(x\right)\cdot\partial_{j}X_{t}\left(x\right)$,
$A_{t}^{\#}\left(x\right)\sim A_{i}^{j}\left(x,\, t\right)=-\partial_{i}N_{t}\left(x\right)\cdot\partial_{j}X_{t}\left(x\right)$
is the shape operator of $\Sigma_{t}$ at $X_{t}\left(x\right)$ with
$N_{t}\left(x\right)$ being the unit-normal of $\Sigma_{t}$ at $X_{t}\left(x\right)$,
$\tilde{A_{t}}^{\#}\sim\tilde{A}_{i}^{j}\left(x,\, t\right)=-\partial_{i}\tilde{N_{t}}\left(x\right)\cdot\partial_{j}\tilde{X_{t}}\left(x\right)$
is the shape operator of $\tilde{\Sigma}_{t}$ at $\tilde{X}_{t}\left(x\right)$
with $\tilde{N}_{t}\left(x\right)$ being the unit-normal of $\tilde{\Sigma}_{t}$
at $\tilde{X}_{t}\left(x\right)$.

Then we have the following lemma, which is analogous to Proposition \ref{p11}:
\begin{lem}\label{l12}
There exits $R=R\left(\Sigma,\,\tilde{\Sigma},\mathcal{\, C},\, U,\,\parallel F\parallel_{C^{3}\left(U\right)},\,\lambda,\,\varkappa\right)\geq1$
such that for each $t\in[-1,\,0)$, the deviation $h_{t}$ satisfies
\begin{equation}
\nabla_{\Sigma_{t}}\cdot\left(\mathbf{a}_{t}\, dh_{t}\right)-\frac{1}{2\left(-t\right)}\left(X_{t}\cdot\nabla_{\Sigma_{t}}h_{t}\,-\, h_{t}\right)=O\left(|X_{t}|{}^{-1}\right)|\nabla_{\Sigma_{t}}h_{t}|\,+O\left(|X_{t}|{}^{-2}\right)|h_{t}|\label{51}
\end{equation}
for $X_{t}\in\Sigma_{t}\setminus\bar{B}_{R}$, where 
$$\nabla_{\Sigma_{t}}\cdot\left(\mathbf{a}_{t}\, dh_{t}\right)=\sum_{i,\, j}\nabla_{i}\left(\mathbf{a}_{t}^{ij}\nabla_{j}h_{t}\right)$$
and 
\[
\Big|O\left(|X_{t}|{}^{-1}\right)\Big|\,\leq C\left(n,\,\mathcal{C},\,\parallel F\parallel_{C^{3}\left(U\right)}\right)\,|X_{t}|{}^{-1}
\]
Also, we have 
\[
\parallel|X_{t}|\, h_{t}\parallel_{L^{\infty}(\Sigma_{t}\setminus\bar{B}_{R})}\,+\parallel|X_{t}|{}^{2}\nabla_{\Sigma_{t}}h_{t}\parallel_{L^{\infty}(\Sigma_{t}\setminus\bar{B}_{R})}\,+\parallel|X_{t}|{}^{3}\nabla_{\Sigma_{t}}^{2}h_{t}\parallel_{L^{\infty}(\Sigma_{t}\setminus\bar{B}_{R})}
\]
\begin{equation}
\leq C\left(n,\,\mathcal{C},\,\parallel F\parallel_{C^{3}\left(U\right)}\right)\left(-t\right)\label{52}
\end{equation}
\end{lem}
\begin{proof}
Fix $t\in[-1,\,0)$ and $\hat{X}_{t}\in\Sigma_{t}\setminus B_{R}$,
then we have $\hat{X}=\frac{\hat{X}_{t}}{\sqrt{-t}}\in\Sigma\setminus\bar{B}_{R}$
and 

\noindent \resizebox{1.009\linewidth}{!}{
  \begin{minipage}{\linewidth}
  \begin{align*}
\left(\nabla_{\Sigma_{t}}\cdot\left(\mathbf{a}_{t}\, dh_{t}\right)\,-\,\frac{1}{2\left(-t\right)}\left(X_{t}\cdot\nabla_{\Sigma_{t}}h_{t}\,-\, h_{t}\right)\right)\Big|_{\hat{X}_{t}}=\frac{1}{\sqrt{-t}}\left(\nabla_{\Sigma}\cdot\left(\mathbf{a}\, dh\right)\,-\,\frac{1}{2}\left(X\cdot\nabla_{\Sigma}h\,-\, h\right)\right)\Big|_{\hat{X}_{t}}
\end{align*}
  \end{minipage}
}

\[
=\frac{1}{\sqrt{-t}}\left(O\left(\hat{|X}|{}^{-1}\right)|\nabla_{\Sigma}h|\,+O\left(|\hat{X}|{}^{-2}\right)|h|\right)\Big|_{\hat{X}_{t}}
\]
\[
=\left(O\left(|\hat{X}_{t}|{}^{-1}\right)|\nabla_{\Sigma_{t}}h_{t}|\,+O\left(|\hat{X}_{t}|{}^{-2}\right)|h_{t}|\right)\Big|_{\hat{X}_{t}}
\]
Similarly, to derive (\ref{52}), it suffices to rescale (\ref{21})
to get

\[
|\hat{X}_{t}|\,|h_{t}|\,+\,|\hat{X}_{t}|{}^{2}|\nabla_{\Sigma_{t}}h_{t}|\,+\,|\hat{X}_{t}|{}^{3}|\nabla_{\Sigma_{t}}^{2}h_{t}|\Big|_{\hat{X}_{t}}
\]
\[
=(-t)\left(\hat{|X}|\,|h|\,+\,|\hat{X}|{}^{2}|\nabla_{\Sigma}h|\,+\,|\hat{X}|{}^{3}|\nabla_{\Sigma}^{2}h|\right)\Big|_{\hat{X}_{t}}
\]
\[
\leq C\left(n,\,\mathcal{C},\,\parallel F\parallel_{C^{3}\left(U\right)}\right)\left(-t\right)
\]
 
\end{proof}
Next, we define ``normal parametrizations'' of the flow:
\begin{defn}\label{d13}
$X_{t}=X\left(\cdot,\, t\right)$ is called a ``normal parametrization''
for the motion of a hypersurface $\left\{ \Sigma_{t}\right\} $ provided
that 
\[
\partial_{t}X_t=F\left(A^{\#}\right)N
\]
That is, each particle on the hypersurface moves in normal direction
during the flow (see also Definition \ref{d4}).
\end{defn}
In the derivation of the parabolic equation to be satisfied by $h_{t}=h\left(\cdot,\, t\right)$,
we start with a ``radial parametrization'' of the flow $\left\{ \Sigma_{t}\right\} _{-1\leq t<0}$
(i.e. each particles on the hypersurface moves in the radial direction
along the flow, see the proof of Proposition \ref{p14} for details),
then we make a transition to the ``normal parametrization'' by using
a time-dependent tangential diffeomorphism. Note that in general,
the ``radial parametrization'' exists only for a short period of
time (unlike the ``vertical parametrization''), so later in the
proof, we would do a local argument, which is sufficient for deriving the equation.
\begin{prop}\label{p14}
There exits $R=R\left(\Sigma,\,\tilde{\Sigma},\mathcal{\, C},\, U,\,\parallel F\parallel_{C^{3}\left(U\right)},\,\lambda,\,\varkappa\right)\geq1$
so that in the normal parametrization of the $F$ curvature flow $\left\{ \Sigma_{t}\right\} _{-1\leq t<0}$
, the deviation $h_{t}$ satisfies 
\begin{equation}
\mathbf{P}h\,\equiv\,\partial_{t}h-\nabla_{\Sigma_{t}}\cdot\left(\mathbf{a}\left(\cdot,\, t\right)\, dh\right)=O\left(|X_{t}|{}^{-1}\right)|\nabla_{\Sigma_{t}}h|\,+\, O\left(|X_{t}|{}^{-2}\right)|h|\label{53}
\end{equation}
\begin{equation}
h\left(\cdot,\,0\right)=0\quad\textrm{as}\; t\nearrow0\label{54}
\end{equation}
for $X_{t}\in\Sigma_{t}\setminus\bar{B}_{R},\,-1\leq t<0$, where
$\mathbf{a}\left(\cdot,\, t\right)=\mathbf{a}_{t}$.\end{prop}
\begin{proof}
Fix $\hat{t}\in[-1,\,0)$, $\hat{X}\in\Sigma_{\hat{t}}\setminus\bar{B}_{R}$,
and take a local coordinate $x=\left(x_{1},\cdots,\, x_{n}\right)$
of $\Sigma_{\hat{t}}$ around $\hat{X}$. Define the ``radial parametrization''
of the flow starting at time $\hat{t}$ near the point $\hat{X}$
by
\[
X(x,\, t)\,=\,\frac{\sqrt{-t}}{\sqrt{-\hat{t}}}\, X_{\hat{t}}(x)
\]
For this parametrization, we can decompose the velocity vector into
the normal part and the tangential part as follows: 
\begin{equation}
\partial_{t}X(x,\, t)=\,\frac{-1}{2\sqrt{-\hat{t}}\sqrt{-t}}X_{\hat{t}}(x)\label{55}
\end{equation}
\[
=\frac{-1}{2\sqrt{-\hat{t}}\sqrt{-t}}\left(\left(X_{\hat{t}}\left(x\right)\cdot N_{\hat{t}}\left(x\right)\right)N_{\hat{t}}\left(x\right)\,+\,\sum_{i,\, j}g_{\hat{t}}^{ij}\left(x\right)\,\left(X_{\hat{t}}\left(x\right)\cdot\partial_{j}X_{\hat{t}}\left(x\right)\right)\,\partial_{i}X_{\hat{t}}\left(x\right)\right)
\]
\[
=F\left(A_{i}^{j}(x,\, t)\right)N\left(x,\, t\right)-\,\sum_{i,\, j}\frac{1}{2\left(-t\right)}g^{ij}\left(x,\, t\right)\,\left(X\left(x,\, t\right)\cdot\partial_{j}X\left(x,\, t\right)\right)\,\partial_{i}X\left(x,\, t\right)
\]
in which we use the $F$ self-shrinker equation of $\Sigma_{\hat{t}}=\sqrt{-\hat{t}}\,\Sigma$
(in Definition \ref{d4}) and the homogeneity of $F$. Now consider
the following ODE system:
\begin{equation}
\partial_{t}x_{i}=\sum_{i,\, j}\frac{1}{2(-t)}\, g^{ij}\left(x,\, t\right)\,\left(X\left(x,\, t\right)\cdot\partial_{j}X\left(x,\, t\right)\right)\label{56}
\end{equation}
\[
x_{i}\Big|_{t=\hat{t}}=\xi_{i},\,\, i=1,\cdots,\, n
\]
Let the solution (which exists at least for a while) to be $x=\varphi_{t}\left(\xi\right)$.
In other words, $\varphi_{t}$ is the local diffeomorphism on $\Sigma_{t}$
generated by the tangent vector field $\frac{1}{2\left(-t\right)}X\left(x,\, t\right)^{\top}$.
By (\ref{55}) and (\ref{56}), the reparametrization $X\left(\varphi_{t}\left(\xi\right),\, t\right)$
of the flow becomes a normal parametrization. 

On the other hand, in the radial parametrization, we have
$$h(x,\, t)=\frac{\sqrt{-t}}{\sqrt{-\hat{t}}}h_{\hat{t}}(x)$$
Thus, by (\ref{56}) and Lemma  \ref{l12}, we get 

\noindent \resizebox{1.009\linewidth}{!}{
  \begin{minipage}{\linewidth}
  \begin{align*}
\frac{\partial}{\partial t}\left\{ h\left(\varphi_{t}\left(\xi\right),t\right)\right\} =\partial_{t}h\left(x,\, t\right)+\sum_{i,\, j}\frac{1}{2\left(-t\right)}\, g^{ij}\left(x,\, t\right)\,\left(X\left(x,\, t\right)\cdot\partial_{j}X\left(x,\, t\right)\right)\,\partial_{i}h\left(x,\, t\right)\Big|_{x=\varphi_{t}\left(\xi\right)}
\end{align*}
  \end{minipage}
}

\[
=\frac{1}{2\left(-t\right)}\left\{ -h\left(x,\, t\right)+\, X\left(x,\, t\right)\cdot\nabla_{\Sigma_{t}}h\right\} \Big|_{x=\varphi_{t}\left(\xi\right)}
\]
\[
=\nabla_{\Sigma_{t}}\cdot\left(\mathbf{a}\left(\cdot,\, t\right)\, dh_{t}\right)\,+\, O\left(|X_{t}|{}^{-1}\right)|\nabla_{\Sigma_{t}}h_{t}|\,+\, O\left(|X_{t}|{}^{-2}\right)|h_{t}|\Big|{}_{x=\varphi_{t}\left(\xi\right)}
\]
which proves (\ref{53}).\\
(\ref{54}) follows from (\ref{52}).
\end{proof}
Lastly, we conclude this section by giving some estimates on the 2-tensor
$\mathbf{a}\left(\cdot,\, t\right)$ for each time-slice $\Sigma_{t}$.
\begin{prop}\label{p15}
There exits $R=R\left(\Sigma,\,\tilde{\Sigma},\mathcal{\, C},\, U,\,\parallel F\parallel_{C^{3}\left(U\right)},\,\lambda,\,\varkappa\right)\geq1$
so that for $t\in[-1,\,0)$, $X_{t}\in\Sigma_{t}\setminus\bar{B}_{R}$,
there hold 
\begin{equation}
\frac{\lambda}{3}\,\leq\mathbf{a}\left(\cdot,\, t\right)\leq\,\frac{3}{\lambda}\label{57}
\end{equation}
\begin{equation}
|X_{t}|\Big|\nabla_{\Sigma_{t}}\mathbf{a}\left(\cdot,\, t\right)\Big|\,\leq3\varkappa\label{58}
\end{equation}
\begin{equation}
|X_{t}|{}^{2}\Big|\nabla_{\Sigma_{t}}^{2}\mathbf{a}\left(\cdot,\, t\right)\Big|\,\leq C\left(n,\,\mathcal{C},\,\parallel F\parallel_{C^{3}\left(U\right)}\right)\label{59}
\end{equation}

\end{prop}
\begin{equation}
|X_{t}|{}^{2}\Big|\partial_{t}\mathbf{a}\left(\cdot,\, t\right)\Big|\,\leq C\left(n,\,\mathcal{C},\,\parallel F\parallel_{C^{3}\left(U\right)}\right)\label{60}
\end{equation}
where the time derivative in the last term is taken with respect to
a normal parametrization of the flow $\left\{ \Sigma_{t}\right\} _{-1\leq t<0}$.
\begin{proof}
We adopt the Einstein summation convention throughout the proof. 

By using the rescaling argument and the homogeneity of the derivatives
of $F$, (\ref{57}), (\ref{58}), (\ref{59}) follow from (\ref{40}),
(\ref{41}), (\ref{42}), respectively. As for (\ref{60}), note that
in a normal parametrization, we have 
\begin{equation}
\partial_{t}\bar{\mathtt{a}}^{ij}\left(t\right)=\partial_{t}\left(\bar{\mathtt{a}}_{k}^{i}(t)\, g_{t}^{kj}\right)=\left(\partial_{t}\bar{\mathtt{a}}_{k}^{i}(t)\right)g_{t}^{kj}+2\bar{\mathtt{a}}_{k}^{i}(t)\, F\left(A_{t}^{\#}\right)A_{t}^{kj}\label{61}
\end{equation}
in which we use the following evolution equation for the metric along
the $F$ curvature flow $\left\{ \Sigma_{t}\right\} _{-1\leq t<0}$
(see \cite{A}):
\begin{equation}
\partial_{t}g_{ij}\left(t\right)=-2F\left(A_{t}^{\#}\right)\, A_{ij}\left(t\right),\quad\partial_{t}g_{t}^{ij}=2F\left(A_{t}^{\#}\right)\, A_{t}^{ij}\label{62}
\end{equation}
By the rescaling argument, (\ref{17}), and the homogeneity of $F$
and its derivatives, we can estimate each term in (\ref{61}) by 
\[
|X_{t}|{}^{2}\Big|F\left(A_{t}^{\#}\right)A_{t}^{ij}\Big|\,=\,\Big|\, F\left(|X_{t}|\, A_{t}^{\#}\right)\cdot|X_{t}|\, A_{t}^{ij}\,\Big|\,\leq C\left(n,\,\mathcal{C},\,\parallel F\parallel_{C^{3}\left(U\right)}\right)
\]
and 
\[
|X_{t}|{}^{2}|\partial_{t}\bar{\mathtt{a}}_{j}^{i}|\,=\,|X_{t}|{}^{2}\Big|\int_{0}^{1}\frac{\partial^{2}F}{\partial S_{i}^{j}\partial S_{k}^{l}}\left(\left(1-\theta\right)A_{t}^{\#}+\theta\tilde{A_{t}}^{\#}\right)\cdot\left(\left(1-\theta\right)\partial_{t}A_{k}^{l}+\theta\partial_{t}\tilde{A}_{k}^{l}\right)\, d\theta\Big|
\]

\noindent \resizebox{1.009\linewidth}{!}{
  \begin{minipage}{\linewidth}
  \begin{align*}
=\Big|\int_{0}^{1}\frac{\partial^{2}F}{\partial S_{i}^{j}\partial S_{k}^{l}}\left(\left(1-\theta\right)|X_{t}|\, A_{t}^{\#}+\theta|X_{t}|\,\tilde{A_{t}}^{\#}\right)\cdot\left(\left(1-\theta\right)|X_{t}|^{3}\partial_{t}A_{k}^{l}+\theta|X_{t}|{}^{3}\partial_{t}\tilde{A}_{k}^{l}\right)\, d\theta\Big|
\end{align*}
  \end{minipage}
}

\[
\leq C\left(n,\,\mathcal{C},\,\parallel F\parallel_{C^{3}\left(U\right)}\right)\Big|\int_{0}^{1}\left(\left(1-\theta\right)|X_{t}|{}^{3}\partial_{t}A_{k}^{l}+\theta|X_{t}|{}^{3}\partial_{t}\tilde{A}_{k}^{l}\right)\, d\theta\Big|
\]
Thus, to establish (\ref{60}), it suffices to show that 
\begin{equation}
|X_{t}|{}^{3}|\partial_{t}A_{t}^{\#}|\,\leq C\left(n,\,\mathcal{C},\,\parallel F\parallel_{C^{3}\left(U\right)}\right)\label{63}
\end{equation}
\begin{equation}
|X_{t}|^{3}|\partial_{t}\tilde{A}_{t}^{\#}-\partial_{t}A_{t}^{\#}|\,\leq C\left(n,\,\mathcal{C},\,\parallel F\parallel_{C^{3}\left(U\right)}\right)\label{64}
\end{equation}
for all $X_{t}\in\Sigma_{t}\setminus\bar{B}_{R}$, $t\in[-1,\,0)$.

Firstly, let's recall the evolution equation for the shape operator
$A_{t}^{\#}$ in the normal parametrization along the flow (see \cite{A}):
\begin{equation}
\partial_{t}A_{i}^{j}(t)=\frac{\partial F}{\partial S_{k}^{l}}\left(A_{t}^{\#}\right)\cdot g_{t}^{lm}\nabla_{km}^{2}A_{i}^{j}+\frac{\partial F}{\partial S_{k}^{l}}\left(A_{t}^{\#}\right)\cdot\left(A_{t}^{2}\right)_{k}^{l}A_{i}^{j}(t)\label{65}
\end{equation}
\[
+\frac{\partial^{2}F}{\partial S_{k}^{l}\partial S_{p}^{q}}\left(A_{t}^{\#}\right)\cdot g_{t}^{jm}\nabla_{i}A_{k}^{l}(t)\nabla_{m}A_{p}^{q}(t)
\]
which yields (\ref{63}) by the rescaling argument, (\ref{20}) and
the homogeneity of $F$ and its derivatives.

Secondly, we would compute $\partial_{t}\left(\tilde{A}_{t}^{\#}-A_{t}^{\#}\right)$
in the normal parametrization (of $\left\{ \Sigma_{t}\right\} _{-1\leq t<0}$)
by using the same trick as in the proof of Proposition \ref{p14}. Fix $\hat{t}\in[-1,\,0)$,
$\hat{X}\in\Sigma_{\hat{t}}\setminus\bar{B}_{R}$, and take a local
coordinate $x=\left(x_{1},\cdots,\, x_{n}\right)$ of $\Sigma_{\hat{t}}$
which is normal at $\hat{X}=X\left(0\right)$. Consider the radial
parametrization of the flow starting at time $\hat{t}$ near the point
$\hat{X}$ by 
$$X(x,\, t)=\frac{\sqrt{-t}}{\sqrt{-\hat{t}}}\, X_{\hat{t}}(x)$$
Then we have 
\[
\tilde{A}_{i}^{j}\left(x,\, t\right)-A_{i}^{j}\left(x,\, t\right)=\frac{\sqrt{-\hat{t}}}{\sqrt{-t}}\,\left(\tilde{A}_{i}^{j}\left(x,\,\hat{t}\right)-A_{i}^{j}\left(x,\,\hat{t}\right)\right)
\]
Let $x=\varphi_{t}\left(\xi\right)$ with $\varphi_{\hat{t}}=\textrm{id}$
to be the local diffeomorphism on $\Sigma_{t}$ generated by the tangent
vector field $\frac{1}{2\left(-t\right)}\, X\left(\cdot,\, t\right)^{\top}$
as before. Then the reparametrization $X\left(\varphi_{t}\left(\xi\right),\, t\right)$
of the flow becomes a normal parametrization and we have 
\begin{equation}
\partial_{t}\left(\tilde{A}_{i}^{j}\left(\varphi_{t}\left(\xi\right),\, t\right)-A_{i}^{j}\left(\varphi_{t}\left(\xi\right),\, t\right)\right)\Big|_{\xi=0,\, t=\hat{t}}=\left(\partial_{t}\tilde{A}_{i}^{j}-\partial_{t}A_{i}^{j}\right)\left(\varphi_{t}\left(\xi\right),\, t\right)\label{66}
\end{equation}

\noindent \resizebox{1.009\linewidth}{!}{
  \begin{minipage}{\linewidth}
  \begin{align*}
+\frac{1}{2\left(-t\right)}g^{kl}\left(\varphi_{t}\left(\xi\right),\, t\right)\,\left(X_{t}\left(\varphi_{t}\left(\xi\right),\, t\right)\cdot\partial_{l}X_{t}\left(\varphi_{t}\left(\xi\right),\, t\right)\right)\left(\partial_{k}\tilde{A}_{i}^{j}\left(\varphi_{t}\left(\xi\right),\, t\right)-\partial_{k}A_{i}^{j}\left(\varphi_{t}\left(\xi\right),\, t\right)\right)\Big|_{\xi=0,\, t=\hat{t}}
\end{align*}
  \end{minipage}
}

\[
=\frac{1}{2\left(-\hat{t}\right)}\left\{ \left(\tilde{A}_{i}^{j}(\hat{t})-A_{i}^{j}(\hat{t})\right)+\, g_{\hat{t}}^{kl}\left(X_{\hat{t}}\cdot\partial_{l}X_{\hat{t}}\right)\left(\nabla_{k}\tilde{A}_{i}^{j}(\hat{t})-\nabla_{k}A_{i}^{j}(\hat{t})\right)\right\} \Big|_{\hat{X}}
\]
Note that for each $t\in[-1,\,0)$, by the rescaling argument and
(\ref{22}), we have
\begin{equation}
\parallel|X_{t}|{}^{3}\left(\tilde{A_{t}}^{\#}-A_{t}^{\#}\right)\parallel_{L^{\infty}(\Sigma_{t}\setminus\bar{B}_{R})}+\parallel|X_{t}|{}^{4}\left(\nabla_{\Sigma_{t}}\tilde{A_{t}}^{\#}-\nabla_{\Sigma_{t}}A_{t}^{\#}\right)\parallel_{L^{\infty}(\Sigma_{t}\setminus\bar{B}_{R})}\label{67}
\end{equation}
\[
\leq\left\{ \parallel|X|^{3}\left(\tilde{A}^{\#}-A^{\#}\right)\parallel_{L^{\infty}(\Sigma\setminus\bar{B}_{R})}+\parallel|X|{}^{4}\left(\nabla_{\Sigma}\tilde{A}^{\#}-\nabla_{\Sigma}A^{\#}\right)\parallel_{L^{\infty}(\Sigma\setminus\bar{B}_{R})}\right\} \left(-t\right)
\]
\[
\leq C\left(n,\,\mathcal{C},\,\parallel F\parallel_{C^{3}\left(U\right)}\right)(-t)
\]
Combining (\ref{66}) and (\ref{67}) to get (\ref{64}).
\end{proof}

\vspace{.5in}
\section{Carleman's inequalities and uniqueness of $F$ self-shrinkers
with a tangent cone at infinity}

This section is a continuation of the previous section. Here we still
assume that $\Sigma^n$ and $\tilde{\Sigma}^n$ are properly embedded
$F$ self-shrinkers (in Definition \ref{d4}) which are $C^{5}$ asymptotic
to the cone $\mathcal{C}^n$ at infinity, and they induce $F$ curvature
flows $\left\{ \Sigma_{t}\right\} _{-1\leq t\leq0}$ and $\left\{ \tilde{\Sigma}_{t}\right\} _{-1\leq t\leq0}$,
 where $\Sigma_{t}=\sqrt{-t}\,\Sigma,\;\tilde{\Sigma}_{t}=\sqrt{-t}\tilde{\,\Sigma}$
for $t\in[-1,\,0)$ and $\Sigma_{0}=\mathcal{C}=\tilde{\Sigma}_{0}$.
We also consider the deviation $h_{t}=h\left(\cdot,\, t\right)$ of
$\tilde{\Sigma}_{t}$ from $\Sigma_{t}$ for $t\in\left[-1,\,0\right]$
(we set $h_{0}=0$), which is defined on $\Sigma_{t}\setminus\bar{B}_{R}$ (see Lemma \ref{l8}),
where $R\gg1$ (depending on $\Sigma,\,\tilde{\Sigma},\mathcal{\, C},\, U,\,\parallel F\parallel_{C^{3}\left(U\right)},\,\lambda,\,\varkappa$).
For the function $h$, recall that we have Proposition \ref{p14} and Proposition
\ref{p15}. The Einstein summation convention is adopted throughout
this section (i.e. summing over repeated indices).

At the beginning, we would improve the rate of decay of $h_{t}$
as $t\nearrow0$ in (\ref{52}) to be exponential. To achieve that,
we need Proposition \ref{p20}, which is due to \cite{EF} and \cite{N} for
different cases. The proof of Proposition \ref{p20} would be included here
for readers' convenience, and it is based on two crucial lemmas (which we state without proof). The
first one is the mean value inequality for parabolic equations from \cite{LSU}.

\begin{lem}[Mean value inequality]\label{l16}

Let $P=\partial_{t}-\partial_{i}\left(a^{ij}\left(x,\, t\right)\,\partial_{j}\right)$
be a differential operator such that $a_{t}^{ij}=a^{ij}(\cdot,\, t)\in C^{1}\left(B_{1}^{n}\right)$
for $t\in\left[-1,\,0\right]$, $a^{ij}=a^{ji}$, and
\[
\boldsymbol{\lambda}\delta^{ij}\leq a^{ij}\leq\frac{1}{\boldsymbol{\lambda}}\delta^{ij}
\]
\[
|a^{ij}\left(x,\, t\right)-a^{ij}\left(\tilde{x},\,\tilde{t}\right)|\,\leq L\left(|x-\tilde{x}|\,+\,|t-\tilde{t}|{}^{\frac{1}{2}}\right)
\]
for some $\boldsymbol{\lambda}\in(0,\,1]$, $L>0$, where $B_{1}^{n}=\left\{ x\in\mathbb{R}^{n}\Big|\,|x|<1\right\} $. 

Suppose that $\mathrm{u}\in C^{2,1}\left(B_{1}^{n}\times\left[-T,\,0\right]\right)$
satisfies
\[
|P\mathscr{\mathrm{u}}|\,\leq\, L\left(\frac{1}{\sqrt{T}}|\partial_{x}\mathscr{\mathrm{u}}|\,+\,\frac{1}{T}|\mathscr{\mathrm{u}}|\right)
\]
for some $T\in(0,\,1]$, then there holds
\[
|\mathscr{\mathrm{u}}\left(x,\, t\right)|\,+\,\sqrt{-t}\,|\partial_{x}\mathscr{\mathrm{u}}\left(x,\, t\right)|\,\leq C\left(n,\,\boldsymbol{\lambda},\, L\right)\,\fint_{Q\left(x,\, t;\,\sqrt{-t}\right)}\,|\mathscr{\mathrm{u}}|
\]
for $\left(x,\, t\right)\in Q\left(0,\,0;\,\frac{\sqrt{T}}{2}\right)$, where
$Q\left(x,\, t;\, r\right)=B_{r}^{n}\left(x\right)\times(t-r^{2},\,t]$
is the parabolic cylinder centered at $\left(x,\, t\right)$ and $\fint_{\mathcal{D}}$
means taking the average of a function over the region $\mathcal{D}$.\end{lem}
\begin{rem}\label{r17}
To prove the above lemma, we may consider the following change of
variables:
\[
\left(x,\, t\right)=\left(\sqrt{T}\,\widehat{x},\, T\,\widehat{t}\right)
\]
In the new variables, the equation in Lemma \ref{l16} becomes 
\[
\Big|\partial_{\widehat{t}}\mathscr{\mathrm{u}}-\partial_{\widehat{x}_{i}}\left(a^{ij}\left(\sqrt{T}\,\widehat{x},\: T\,\widehat{t}\right)\partial_{\tilde{x}_{j}}\mathscr{\mathrm{u}}\right)\Big|\,\leq\, L\left(|\partial_{\widehat{x}}\mathscr{\mathrm{u}}|\,+\,|\mathscr{\mathrm{u}}|\right)
\]
for $\widehat{x}\in B_{\nicefrac{1}{\sqrt{T}}}^{n}$, $\widehat{t}\in\left[-1,\,0\right]$.
Then apply standard estimates from \cite{LSU} to the normalized equation (note that $T\in(0,\,1]$).
\end{rem}
The second lemma is a local type of Carleman's inequalities from \cite{EFV}.
\begin{lem}[Local Carleman's inequality]\label{l18}

Let $P=\partial_{t}-\partial_{i}\left(a^{ij}\left(x,\, t\right)\,\partial_{j}\right)$
be a differential operator such that $a_{t}^{ij}=a^{ij}(\cdot,\, t)\in C^{1}\left(B_{1}^{n}\right)$
for $t\in\left[-1,\,0\right]$ , $a^{ij}=a^{ji}$, $a^{ij}\left(0,\,0\right)=\delta^{ij}$
and
\[
\boldsymbol{\lambda}\delta^{ij}\leq a^{ij}\leq\frac{1}{\boldsymbol{\lambda}}\delta^{ij}
\]
\[
|a^{ij}\left(x,\, t\right)-a^{ij}\left(\tilde{x},\,\tilde{t}\right)|\,\leq L\left(|x-\tilde{x}|\,+\,|t-\tilde{t}|{}^{\frac{1}{2}}\right)
\]
for some $\boldsymbol{\lambda}\in(0,\,1]$, $L>0$, where $B_{1}^{n}=\left\{ x\in\mathbb{R}^{n}\Big|\,|x|<1\right\} $. 

Then for any fixed constant $M\geq4$, there exists a non-increasing
function $\boldsymbol{\varphi}:\left(-\frac{4}{M},\,0\right)\rightarrow\mathbb{R}_{+}$
satisfying $\frac{-t}{\sigma}\leq\boldsymbol{\varphi}\left(t\right)\leq-t$
for some constant $\sigma=\sigma\left(n,\,\boldsymbol{\lambda},\, L\right)\geq1$,
so that for any constant $\delta\in\left(0,\,\frac{1}{M}\right)$
and function $\mathrm{v}\in C_{\textrm{c}}^{2,1}\left(B_{1}^{n}\times(-\frac{2}{M},\,0]\right)$,
there holds
\[
M^{2}\int\,\mathrm{v}^{2}\boldsymbol{\varphi}_{\delta}^{-M}\boldsymbol{\Phi}_{\delta}\, dx\, dt\,+\, M\int\,|\partial_{x}\mathrm{v}|{}^{2}\boldsymbol{\varphi}_{\delta}^{1-M}\boldsymbol{\Phi}_{\delta}\, dx\, dt
\]
\noindent \resizebox{1.009\linewidth}{!}{
  \begin{minipage}{\linewidth}
  \begin{align*}
  \leq\sigma\int\,|P\mathscr{\mathrm{v}}|{}^{2}\boldsymbol{\varphi}_{\delta}^{1-M}\boldsymbol{\Phi}_{\delta}\, dx\, dt\,+\,\left(\sigma M\right)^{M}\sup_{t<0}\int\,\left(|\partial_{x}\mathscr{\mathrm{v}}|{}^{2}+\mathrm{v}^{2}\right)\, dx\,+\,\sigma M\int\,\mathrm{v}^{2}\boldsymbol{\varphi}_{\delta}^{-M}\boldsymbol{\Phi}_{\delta}\, dx\Big|_{t=0}
\end{align*}
  \end{minipage}
}\\

where $\boldsymbol{\varphi}_{\delta}\left(t\right)=\boldsymbol{\varphi}\left(t-\delta\right)$
and $\boldsymbol{\Phi}_{\delta}\left(x,\, t\right)=\boldsymbol{\Phi}\left(x,\, t-\delta\right)=\frac{1}{\left(4\pi\left(-t+\delta\right)\right)^{\frac{n}{2}}}\,\exp\left(-\frac{\mid x\mid^{2}}{4\left(-t+\delta\right)}\right)$.\end{lem}
\begin{rem}\label{r19}
Note that the last term on the RHS of the above inequality vanishes
provided that $\mathrm{v}\Big|_{t=0}=0$.
\end{rem}
Now we state the proposition (of showing the exponential decay) and
include the proof (which is due to  \cite{EF} and \cite{N}).

\begin{prop}[Exponential decay and Unique continuation principle]\label{p20}

Let $P=\partial_{t}-\partial_{i}\left(a^{ij}\left(x,\, t\right)\,\partial_{j}\right)$
be a differential operator such that $a_{t}^{ij}=a^{ij}(\cdot,\, t)\in C^{1}\left(B_{1}^{n}\right)$
for $t\in\left[-1,\,0\right]$, $a^{ij}=a^{ji}$, and
\[
\boldsymbol{\lambda}\delta^{ij}\leq a^{ij}\leq\frac{1}{\boldsymbol{\lambda}}\delta^{ij}
\]
\[
|a^{ij}\left(x,\, t\right)-a^{ij}\left(\tilde{x},\,\tilde{t}\right)|\,\leq L\left(|x-\tilde{x}|\,+\,|t-\tilde{t}|{}^{\frac{1}{2}}\right)
\]
for some $\boldsymbol{\lambda}\in(0,\,1]$, $L>0$, where $B_{1}^{n}=\left\{ x\in\mathbb{R}^{n}\Big|\,|x|<1\right\} $. 

Suppose that $\mathrm{u}\in C^{2,1}\left(B_{1}^{n}\times\left[-T,\,0\right]\right)$
satisfies 
\begin{equation}
|P\mathscr{\mathrm{u}}|\,\leq\, L\left(\frac{1}{\sqrt{T}}|\partial_{x}\mathscr{\mathrm{u}}|\,+\,\frac{1}{T}|\mathscr{\mathrm{u}}|\right)\label{68}
\end{equation}
for some $T\in(0,\,1]$, and that either $u$ vanishes at $\left(0,\,0\right)$
to infinite order (see \cite{EF}), i.e. 
\begin{equation}
\forall\, k\in\mathbb{N\quad}\exists\, C_{k}>0\quad\textrm{s.t.}\quad|\mathrm{u}\left(x,\, t\right)|\,\leq\, C_{k}\left(|x|\,+\,\sqrt{-t}\right)^{k}\label{69}
\end{equation}
or $\mathrm{u}$ vanishes identically at $t=0$ (see \cite{N}), i.e.
\begin{equation}
\mathscr{\mathrm{u}}\Big|_{t=0}=0\label{70}
\end{equation}

Then there exit $\boldsymbol{\Lambda}=\boldsymbol{\Lambda}\left(n,\,\boldsymbol{\lambda},\, L\right)>0$,
$\boldsymbol{\alpha}=\boldsymbol{\alpha}\left(n,\,\boldsymbol{\lambda},\, L\right)\in\left(0,\,1\right)$
so that
\begin{equation}
|\mathrm{u}\left(x,\, t\right)|\,+\,|\partial_{x}\mathscr{\mathrm{u}}\left(x,\, t\right)|\label{71}
\end{equation}
\[
\leq\boldsymbol{\Lambda}\, e^{\frac{1}{\boldsymbol{\Lambda}t}}\left(\parallel\partial_{x}\mathrm{u}\parallel_{L^{\infty}\left(B_{1}\times\left[-T,\,0\right]\right)}+\parallel\mathrm{u}\parallel_{L^{\infty}\left(B_{1}\times\left[-T,\,0\right]\right)}\right)
\]
for $x\in B_{\nicefrac{1}{4}}^{n}$, $t\in[-\boldsymbol{\alpha}T,\,0)$.\end{prop}
\begin{rem}\label{r21}
Later we would apply Proposition \ref{p20} under the condition (\ref{70})
to show the exponential decay of the deviation $h_{t}=h\left(\cdot,\, t\right)$ as $t\nearrow0$.
On the other hand, Proposition \ref{p20}  implies that under the condition
(\ref{69}), the function $\mathrm{u}$ in ($\ref{68}$) must vanish identically
at $t=0$. In particular, in the case when $\mathrm{u}$ is time-independent (so $\mathrm{u}$ would satisfy an elliptic equation, and we could take $t=0$ in (\ref{69})), 
it implies that $\mathrm{u}$ vanishes identically.
Such phenomenon is called the ``unique continuation principle'' and would be used at the end
of this section.\end{rem}

\begin{proof}
For simplicity, we may assume that $a^{ij}\left(0,\,0\right)=\delta^{ij}$. 
Otherwise, we could do change of variables like $\widehat{x}=a^{ij}\left(0,\,0\right)^{-\frac{1}{2}}x$ to achieve that.

In the proof, we mainly focus on dealing with the case of (\ref{69}),
since the same argument also applies to the case of (\ref{70}) with only
a slight modification, which we would point out on the way of proof.

Fix a constant $M\in[\frac{4L^{2}(n+\sigma)}{T},\,\infty)$ (to be
chosen), where $\sigma=\sigma\left(n,\,\boldsymbol{\lambda},\, L\right)\geq1$
is the constant that appears in Lemma  \ref{l18}. Then for any $\epsilon\in\left(0,\,\min\left\{ \frac{1}{M},\,1\right\} \right)$,
choose smooth cut-off functions $\zeta=\zeta(x)$, $\eta_{\epsilon}=\eta_{\epsilon}(t)$
and $\eta=\eta(t)$ such that 
\[
\chi_{B_{\nicefrac{1}{2}}^{n}}\leq\zeta\leq\chi_{B_{1}^{n}},\quad\parallel\zeta\parallel_{C^{2}}\leq4
\]
 
\[
\chi_{\left[\frac{-1}{M},\,-\epsilon\right]}\leq\eta_{\epsilon}\leq\chi_{\left[\frac{-2}{M},\,0\right]},\quad\chi_{\left[\frac{-1}{M},\,0\right]}\leq\eta\leq\chi_{\left[\frac{-2}{M},\,0\right]},\quad\eta_{\epsilon}\nearrow\eta\quad\textrm{as}\:\epsilon\searrow0
\]
 
\[
|\partial_{t}\eta_{\epsilon}|\,\leq2M\chi_{\left[\frac{-2}{M},\,\frac{-1}{M}\right]}+\,\frac{2}{\epsilon}\chi_{\left[-\epsilon,\,0\right]}
\]
where $\chi_{B_{1}^{n}}$ is the characteristic function of $B_{1}^{n}$.\\
Let $\mathrm{v}_{\epsilon}\left(x,\, t\right)=\zeta(x)\,\eta_{\epsilon}(t)\,\mathrm{v}\left(x,\, t\right)$
be a localization of $\mathrm{v}$, which satisfies $\mathrm{v}_{\epsilon}\Big|_{t=0}=0$
and converges pointwisely to $\mathrm{v}\left(x,\, t\right)=\zeta(x)\,\eta(t)\,\mathrm{v}\left(x,\, t\right)$
as $\epsilon\searrow0$. Besides, we have
\begin{equation}
|P\mathrm{v}_{\epsilon}|\,\leq\, \{L\zeta\eta_{\epsilon}\left(\frac{1}{\sqrt{T}}|\partial_{x}\mathrm{v}|\,+\,\frac{1}{T}|\mathrm{v}|\right)\label{72}
\end{equation}
\[
+C\left(\boldsymbol{\lambda},\, L\right)\left(|\partial_{x}\mathscr{\mathrm{v}}|\,+\,|\mathscr{\mathrm{v}}|\right)\,\chi_{B_{1}\setminus B_{\frac{1}{2}}}\left(x\right)+\,2LM\,|\mathscr{\mathrm{v}}|\,\chi_{\left[\frac{-2}{M},\,\frac{-1}{M}\right]}\left(t\right)+\frac{2L}{\epsilon}\,|\mathscr{\mathrm{v}}|\,\chi_{\left[-\epsilon,\,0\right]}\left(t\right)\}
\]
\[
\leq L\left(\frac{1}{\sqrt{T}}|\partial_{x}\mathrm{v}_{\epsilon}|\,+\,\frac{1}{T}|\mathrm{v}_{\epsilon}|\right)+\, C\left(\boldsymbol{\lambda},\, L\right)M\left(|\partial_{x}\mathrm{v}|\,+\,|\mathrm{v}|\right)\,\chi_{E}\left(x,\, t\right)+\,\frac{2L}{\epsilon}\,|\mathscr{\mathrm{v}}|\,\chi_{\left[-\epsilon,\,0\right]}\left(t\right)
\]
where $E=\left\{ \left(x,\, t\right)\in B_{1}^{n}\times[-1,\,0)\Big|\;\frac{1}{2}\leq|x|\leq1\:\,\textrm{or}\,\:\frac{-2}{M}\leq t\leq\frac{-1}{M}\right\} $.
Note that in the case of (\ref{70}), it suffices to consider $\mathrm{v}$
(without using the $\epsilon$ cut-off) in order to make the function
vanishing at $t=0$. 

Then for each $\delta\in\left(0,\,\frac{1}{M}\right)$, by Lemma  \ref{l18}
(applied to $\mathrm{v}_{\epsilon}$) and (\ref{72}), there holds
\[
M^{2}\int\,\mathrm{v}_{\epsilon}^{2}\boldsymbol{\varphi}_{\delta}^{-M}\boldsymbol{\Phi}_{\delta}\, dx\, dt\,+\, M\int\,|\partial_{x}\mathrm{v}_{\epsilon}|{}^{2}\boldsymbol{\varphi}_{\delta}^{1-M}\boldsymbol{\Phi}_{\delta}\, dx\, dt
\]
\noindent \resizebox{1.009\linewidth}{!}{
  \begin{minipage}{\linewidth}
  \begin{align*}
\leq \{2\sigma L^{2}\int\,\left(\frac{\mathrm{v}_{\epsilon}^{2}}{T^{2}}\,+\,\frac{|\partial_{x}\mathrm{v}_{\epsilon}|^{2}}{T}\right)\boldsymbol{\varphi}_{\delta}^{1-M}\boldsymbol{\Phi}_{\delta}\, dx\, dt\,+\,2C\left(\boldsymbol{\lambda},\, L\right)\sigma M^{2}\int_{E}\,\left(|\partial_{x}\mathscr{\mathrm{v}}|{}^{2}+\mathscr{\mathrm{v}}^{2}\right)\boldsymbol{\varphi}_{\delta}^{1-M}\boldsymbol{\Phi}_{\delta}\, dx\, dt
\end{align*}
  \end{minipage}
}

\[
+\frac{4\sigma L^{2}}{\epsilon^{2}}\int_{-\epsilon}^{0}\int_{B_{1}}\mathrm{v}^{2}\boldsymbol{\varphi}_{\delta}^{1-M}\boldsymbol{\Phi}_{\delta}\, dx\, dt\,+\,\left(\sigma M\right)^{M}\sup_{t}\int\,\left(|\partial_{x}\mathrm{v}_{\epsilon}|{}^{2}+\mathrm{v}_{\epsilon}^{2}\right)\, dx\}
\]
By our choice of $M$, the first term on the RHS of the above inequality
can be absorbed by its LHS. Thus, we get
\begin{equation}
M^{2}\int\,\mathrm{v}_{\epsilon}^{2}\boldsymbol{\varphi}_{\delta}^{-M}\boldsymbol{\Phi}_{\delta}\, dx\, dt\leq \{C\left(\boldsymbol{\lambda},\, L\right)\sigma M^{2}\int_{E}\,\left(|\partial_{x}\mathscr{\mathrm{v}}|{}^{2}+\mathscr{\mathrm{v}}^{2}\right)\boldsymbol{\varphi}_{\delta}^{1-M}\boldsymbol{\Phi}_{\delta}\, dx\, dt\,\label{73}
\end{equation}
\[
+4\left(\sigma M\right)^{M}\sup_{-T\leq t\leq0}\int_{B_{1}}\,\left(|\partial_{x}\mathscr{\mathrm{v}}|^{2}+\mathscr{\mathrm{v}}^{2}\right)\, dx\,+\,\frac{4\sigma L^{2}}{\epsilon^{2}}\int_{-\epsilon}^{0}\int_{B_{1}}\mathscr{\mathrm{v}}^{2}\boldsymbol{\varphi}_{\delta}^{1-M}\boldsymbol{\Phi}_{\delta}\, dx\, dt\}
\]
Now choose an integer $k\geq M+\frac{n}{2}$, then by (\ref{69})
the last term on the RHS of (\ref{73}) can be estimated by
\begin{equation}
\frac{4\sigma L^{2}}{\epsilon^{2}}\int_{-\epsilon}^{0}\int_{B_{1}}\mathscr{\mathrm{v}}^{2}\boldsymbol{\varphi}_{\delta}^{1-M}\boldsymbol{\Phi}_{\delta}\, dx\, dt\label{74}
\end{equation}
\[
\leq\,\frac{4\sigma L^{2}}{\epsilon^{2}}\int_{-\epsilon}^{0}\int_{B_{1}}\frac{C_{k}\left(|x|\,+\,\sqrt{-t}\right)^{2\left(M+\frac{n}{2}\right)}}{\left(\frac{-t+\delta}{\sigma}\right)^{M-1}}\frac{\exp\left(\frac{-|x|{}^{2}}{4\left(-t+\delta\right)}\right)}{\left(4\pi\left(-t+\delta\right)\right)^{\frac{n}{2}}}\, dx\, dt
\]
\noindent \resizebox{1.009\linewidth}{!}{
  \begin{minipage}{\linewidth}
  \begin{align*}
\leq C\left(n,\, C_{k},\,\sigma,\, M,\, L\right)\frac{1}{\epsilon^{2}}\int_{-\epsilon}^{0}\left\{ \int_{B_{1}}\,\left(\frac{|x|{}^{2}}{-t+\delta}\,+1\right)^{M+\frac{n}{2}}\exp\left(\frac{-|x|{}^{2}}{4\left(-t+\delta\right)}\right)\, dx\right\} \left(-t+\delta\right)\, dt
\end{align*}
  \end{minipage}
}

\noindent \resizebox{1.009\linewidth}{!}{
  \begin{minipage}{\linewidth}
  \begin{align*}
\leq C\left(n,\, C_{k},\,\sigma,\, M,\, L\right)\frac{1}{\epsilon^{2}}\int_{-\epsilon}^{0}\left\{ \int_{0}^{\infty}\,\left(|\xi|{}^{2}+1\right)^{M+\frac{n}{2}}\exp\left(\frac{-|\xi|{}^{2}}{4}\right)\, d\xi\right\} \left(-t+\delta\right)^{\frac{n}{2}+1}\, dt
\end{align*}
  \end{minipage}
}

\[
\leq C\left(n,\, C_{k},\,\sigma,\, M,\, L\right)\frac{\left(\epsilon+\delta\right)^{\frac{n}{2}+2}-\delta^{\frac{n}{2}+2}}{\epsilon^{2}}
\]
In view of (\ref{74}), apply the monotone convergence theorem to
(\ref{73}) by first letting $\delta\searrow0$ and then $\epsilon\searrow0$
to arrive at
\begin{equation}
\int_{B_{\frac{1}{2}\times\left(\frac{-1}{M},\,0\right)}}\mathscr{\mathrm{v}}^{2}\boldsymbol{\varphi}^{-M}\boldsymbol{\Phi}\, dx\, dt\,\label{75}
\end{equation}

\noindent \resizebox{1.009\linewidth}{!}{
  \begin{minipage}{\linewidth}
  \begin{align*}
\leq\, C\left(\Lambda,\, L\right)\sigma\int_{E}\,\left(|\partial_{x}\mathscr{\mathrm{v}}|^{2}+\mathscr{\mathrm{v}}^{2}\right)\boldsymbol{\varphi}^{1-M}\boldsymbol{\Phi}\, dx\, dt\,+\left(4\sigma M\right)^{M}\sup_{-T\leq t\leq0}\int_{B_{1}}\,\left(|\partial_{x}\mathscr{\mathrm{v}}|{}^{2}+\mathscr{\mathrm{v}}^{2}\right)\, dx
\end{align*}
  \end{minipage}
}

\noindent \resizebox{1.009\linewidth}{!}{
  \begin{minipage}{\linewidth}
  \begin{align*}
\leq C\left(n,\,\Lambda,\, L\right)\left(\sigma\int_{E}\,\boldsymbol{\varphi}^{1-M}\boldsymbol{\Phi}\, dx\, dt\,+\,\left(\sigma M\right)^{M}\right)\left(\parallel\partial_{x}\mathscr{\mathrm{v}}\parallel^2_{L^{\infty}\left(B_{1}\times\left[-T,\,0\right]\right)}+\parallel\mathrm{v}\parallel^2_{L^{\infty}\left(B_{1}\times\left[-T,\,0\right]\right)}\right)
\end{align*}
  \end{minipage}
}\\

Note that in the case of (\ref{70}), we can get (\ref{75}) directly
from taking the limit as $\delta\searrow0$ without using (\ref{74}).

Next, we would estimate the first term on the RHS of (\ref{75}).
For $\left(x,\, t\right)\in E$, either $\frac{-2}{M}\leq t\leq\frac{-1}{M}$,
in which case we have
\begin{equation}
\boldsymbol{\varphi}^{1-M}\boldsymbol{\Phi}\left(x,\, t\right)\leq\,\left(\frac{-t}{\sigma}\right)^{1-M}\frac{1}{\left(4\pi\left(-t\right)\right)^{\frac{n}{2}}}\leq\,\frac{\left(\sigma M\right)^{M-1+\frac{n}{2}}}{\left(4\pi\sigma\right)^{\frac{n}{2}}}\label{76}
\end{equation}
or $\frac{1}{2}\leq|x|\leq1\:\textrm{and}\:\frac{-2}{M}\leq t<0$,
in which case we have
\begin{equation}
\boldsymbol{\varphi}^{1-M}\boldsymbol{\Phi}\left(x,\, t\right)\leq\,\left(\frac{\sigma M}{\left(-t\right)M}\right)^{M-1}\frac{M^{\frac{n}{2}}}{\left(4\pi\left(-t\right)M\right)^{\frac{n}{2}}}\exp\left(\frac{-M}{16\left(-tM\right)}\right)\label{77}
\end{equation}
\[
=\frac{\left(\sigma M\right)^{M-1}\left(\frac{M}{4\pi}\right)^{\frac{n}{2}}}{\left(-tM\right)^{M-1+\frac{n}{2}}\exp\left(\frac{\nicefrac{M}{16}}{-tM}\right)}\leq\,\left(\sigma M\right)^{M-1}\left(\frac{M}{4\pi}\right)^{\frac{n}{2}}\left(\frac{M-1+\frac{n}{2}}{e\,\nicefrac{M}{16}}\right)^{M-1+\frac{n}{2}}
\]
\[
\leq\left(\frac{16\sigma}{\textrm{e}}\left(M-1+\frac{n}{2}\right)\right)^{M-1+\frac{n}{2}}
\]
Note that in (\ref{77}) we have used the fact that the function $\vartheta\left(\xi\right)=\xi^{M-1+\frac{n}{2}}\exp\left(\frac{\nicefrac{M}{16}}{\xi}\right)$
achieves its minimum on $\mathbb{R}_{+}$ at $\xi=\frac{\nicefrac{M}{16}}{M-1+\frac{n}{2}}$. 

On the other hand, for any $\left(y,\, s\right)\in B_{\frac{1}{4}}\times[\frac{-1}{8M},\,0)$,
the parabolic cylinder $Q\left(y,\, s;\,\sqrt{-s}\right)=B_{\sqrt{-s}}^{n}\left(y\right)\times\left(2s,\, s\right)$
is contained in $B_{\nicefrac{1}{2}}^{n}\times\left(\frac{-1}{M},\,0\right)$
and hence the LHS of (\ref{75}) is bounded below by
\begin{equation}
\int_{B_{\nicefrac{1}{2}}^{n}\times\left(\frac{-1}{M},\,0\right)}\mathscr{\mathrm{v}}^{2}\boldsymbol{\varphi}^{-M}\boldsymbol{\Phi}\, dx\, dt\,\geq\,\frac{\exp\frac{-\nicefrac{1}{4}}{-8s}}{\left(4\pi\right)^{\frac{n}{2}}\left(-2s\right)^{M+\frac{n}{2}}}\int_{Q\left(y,\, s;\,\sqrt{-s}\right)}\,\mathscr{\mathrm{v}}^{2}\, dx\, dt\label{78}
\end{equation}

Combining (\ref{75}), (\ref{76}), (\ref{77}), (\ref{78}), we conclude
that for $\left(y,\, s\right)\in Q\left(0,\,0;\,\frac{-1}{8M}\right)$,
\begin{equation}
\int_{Q\left(y,\, s;\,\sqrt{-s}\right)}\,\mathrm{v}^{2}\, dx\, dt\label{79}
\end{equation}

\noindent \resizebox{1.009\linewidth}{!}{
  \begin{minipage}{\linewidth}
  \begin{align*}
\leq\, C\left(n,\,\boldsymbol{\lambda},\, L,\,\sigma\right)\left(\frac{64\sigma}{\textrm{e}}\left(-sM\right)\right)^{M-1+\frac{n}{2}}\left(\parallel\partial_{x}\mathrm{v}\parallel^2_{L^{\infty}\left(B_{1}\times\left[-T,\,0\right]\right)}+\parallel\mathrm{v}\parallel^2_{L^{\infty}\left(B_{1}\times\left[-T,\,0\right]\right)}\right)
\end{align*}
  \end{minipage}
}

Now let $\beta=\frac{1}{2}\left(\frac{64\sigma}{\textrm{e}}\right)^{-1}$.
For each $\left(y,\, s\right)\in B_{\nicefrac{1}{4}}^{n}\times[\frac{-\beta}{4L^{2}\left(n+\sigma\right)}T,\,0)$,
we choose $M=\frac{\beta}{-s}$ so that $M\geq\frac{4L^{2}\left(n+\sigma\right)}{T}$
(and note that $\frac{-1}{8M}\leq s<0$). By (\ref{79}), we get
\begin{equation}
\fint_{Q\left(y,\, s;\,\sqrt{-s}\right)}\,|\mathscr{\mathrm{v}}|\, dx\, dt\, \leq\, \left(\fint_{Q\left(y,\, s;\,\sqrt{-s}\right)}\,\mathscr{\mathrm{v}}^{2}\, dx\, dt\,\right)^{\frac{1}{2}}\label{80}
\end{equation}

\noindent \resizebox{1.009\linewidth}{!}{
  \begin{minipage}{\linewidth}
  \begin{align*}
\leq\, C\left(n,\,\Lambda,\, L,\,\sigma\right)\sqrt{\left(-s\right)^{-\frac{n}{2}-1}\left(\frac{1}{2}\right)^{-\frac{\beta}{s}-1+\frac{n}{2}}}\left(\parallel\partial_{x}\mathscr{\mathrm{v}}\parallel_{L^{\infty}\left(B_{1}\times\left[-T,\,0\right]\right)}+\parallel\mathrm{v}\parallel_{L^{\infty}\left(B_{1}\times\left[-T,\,0\right]\right)}\right)
\end{align*}
  \end{minipage}
}

\[
\leq C\left(n,\,\boldsymbol{\lambda},\, L,\,\sigma\right)\left(2^{\frac{\beta}{4}}\right)^{\frac{1}{s}}\left(\parallel\partial_{x}\mathrm{v}\parallel_{L^{\infty}\left(B_{1}\times\left[-T,\,0\right]\right)}+\parallel\mathrm{v}\parallel_{L^{\infty}\left(B_{1}\times\left[-T,\,0\right]\right)}\right)
\]\\
Let $\boldsymbol{\alpha}=\frac{\beta}{4L^{2}\left(n+\sigma\right)}$,
$\boldsymbol{\Lambda}=\max\left\{ C\left(n,\,\boldsymbol{\lambda},\, L,\,\sigma\right),\,\left(\frac{\beta}{4}\ln2\right)^{-1}\right\} $,
then (\ref{71}) follows from (\ref{80}) and Lemma \ref{l16}.
\end{proof}
Combining Proposition \ref{p14}, Proposition \ref{p15} with Proposition \ref{p20}, we could
show the exponential decay of $h_{t}=h\left(\cdot,\, t\right)$ as $t\nearrow0$ as in \cite{W} (see also \cite{N}). 
\begin{prop}[Exponential decay of the deviation]\label{p22}
There exist\\ $R=R\left(\Sigma,\,\tilde{\Sigma},\mathcal{\, C},\, U,\,\parallel F\parallel_{C^{3}\left(U\right)},\,\lambda,\,\varkappa\right)\geq1$,
$\Lambda=\Lambda\left(n,\,\mathcal{C},\,\parallel F\parallel_{C^{3}\left(U\right)},\,\lambda\right)>0$,
$\alpha=\alpha\left(n,\,\mathcal{C},\,\parallel F\parallel_{C^{3}\left(U\right)},\,\lambda\right)\in\left(0,\,1\right)$
such that for $X\in\Sigma_{t}\setminus\bar{B}_{R},\; t\in[-\alpha,\,0)$,
there holds
\[
|\nabla_{\Sigma_{t}}h|\,+\,|h|\,\leq\,\Lambda\,\exp\left(\frac{|X|{}^{2}}{\Lambda t}\right)
\]
\end{prop}
\begin{proof}
Fix $\hat{X}\in\Sigma\setminus\bar{B}_{R}$ with $R=R\left(\Sigma,\,\tilde{\Sigma},\mathcal{\, C},\, U,\,\parallel F\parallel_{C^{3}\left(U\right)},\,\lambda,\,\varkappa\right)\geq1$,
we would first show that near $\hat{X}$, there is a ``normal
parametrization'' for the flow $\left\{ \Sigma_{t}\right\} $ for
$t\in\left[-1,\,0\right]$.

Recall that at the beginning of Section 3, we show that there exists
an uniform radius $\boldsymbol{\rho}=\boldsymbol{\rho}\left(n,\,\mathcal{C}\right)\in\left(0,\,1\right)$
so that near $\hat{X}$, each $\Sigma_{t}$ is the graph of the function
$\mathtt{u}_{t}=\mathtt{u}\left(\cdot,\, t\right)$ defined on $B_{\boldsymbol{\rho}|\hat{X}|}^{n}\subset T_{\hat{X}_{\mathcal{C}}}\mathcal{C}$
for $t\in\left[-1,\,0\right]$, where $\hat{X}_{\mathcal{C}}=\Pi\left(\hat{X}\right)$
is the the normal projection of $\hat{X}$ onto $\mathcal{C}$. Note also
that $|\hat{X}_{\mathcal{C}}|$ is comparable with $|\hat{X}|$. In
other words, locally near $\hat{X}$, we have the following ``vertical
parametrization'' of the flow $\left\{ \Sigma_{t}\right\} _{-1\leq t\leq0}$:
\[
X=X\left(x,\, t\right)\equiv\hat{X}_{\mathcal{C}}+\left(x,\,\mathtt{u}\left(x,\, t\right)\right)
\]
Here we assume that the unit-normal of $\mathcal{C}$ at $\hat{X}_{\mathcal{C}}$
to be $\left(0,\,1\right)$ for ease of notation. For this vertical
parametrization, we may decompose the velocity vector into the normal part
and the tangential part as follows:
\[
\partial_{t}X=F\left(A^{\#}\left(x,\, t\right)\right)N\left(x,\, t\right)\,+\,\sum_{i=1}^{n}\frac{\partial_{i}\mathtt{u}\,\partial_{t}\mathtt{u}}{1+|\partial_{x}\mathtt{u}|{}^{2}}\,\partial_{i}X
\]
where $A^{\#}\left(x,\, t\right)$, $N\left(x,\, t\right)$ are the
shape operator and the unit-normal of $\Sigma_{t}$ at $X\left(x,\, t\right)$,
respectively. Note that the normal part is from Definition \ref{d4}. 

Next, we would do suitable change of variables to go from
this ``vertical parametrization'' to the ``normal parametrization''
of the flow (see Definition \ref{d13}). For that purpose, we use the same
trick as in Proposition \ref{p14}.\\
Let $x=\phi_{t}\left(\xi\right)$ with
$\phi_{-1}=\textrm{id}$ to be the local diffeomorphism on $\Sigma_{t}$
generated by the following tangent vector field: 
\[
\mathcal{V}\left(x,\, t\right)=-\sum_{i=1}^{n}\frac{\partial_{i}\mathtt{u}\,\partial_{t}\mathtt{u}}{1+|\partial_{x}\mathtt{u}|{}^{2}}\,\partial_{i}X\,\equiv-\sum_{i=1}^{n}\mathcal{V}^{i}\left(x,\, t\right)\,\partial_{i}X
\]
That is, $\phi_{t}\left(\xi\right)=\phi\left(\xi,\, t\right)$ satisfies
\begin{equation}
\partial_{t}\phi_{t}=\left(\mathcal{V}^{1}\left(\phi_{t},\, t\right),\cdots,\,\mathcal{V}^{n}\left(\phi_{t},\, t\right)\right),\quad\phi_{-1}\left(\xi\right)=\xi\label{81}
\end{equation}
in which, by (\ref{9}) and (\ref{14}), we have

\begin{equation}
|\mathcal{V}^{i}|\,\leq\, C\left(n,\,\mathcal{C},\,\parallel F\parallel_{C^{3}\left(U\right)}\right)\,|\hat{X}|{}^{-1}\quad\forall\, i=1,\cdots,\, n\label{82}
\end{equation}
Thus, by taking $R$ sufficiently large, $\phi_{t}$ is well-defined
for $\xi\in B_{\frac{\boldsymbol{\rho}}{2}|\hat{X}|}^{n},\; t\in\left[-1,\,0\right]$.
It follows that the reparametrization $X=X\left(\phi_{t}\left(\xi\right),\, t\right)$
of the flow becomes a ``normal parametrization'' near $\hat{X}$
for $t\in\left[-1,\,0\right]$; that is, 
\[
\frac{\partial}{\partial t}\left(X\left(\phi_{t}\left(\xi\right),\, t\right)\right)=F\left(A^{\#}\left(\phi_{t}\left(\xi\right),\, t\right)\right)\, N\left(\phi_{t}\left(\xi\right),\, t\right)
\]

Let $g_{ij}\left(\xi,\, t\right)=\partial_{\xi_{i}}\left(X\left(\phi_{t}\left(\xi\right),\, t\right)\right)\cdot\partial_{\xi_{j}}\left(X\left(\phi_{t}\left(\xi\right),\, t\right)\right)$
be the pull-back metric associated with this ``normal parametrization'',
then by the evolution equation for the metric in \cite{A}, the homogeneity
of $F$ and the condition that $\phi_{-1}=\textrm{id}$, we have
\begin{equation}
\partial_{t}g_{ij}\left(\xi,\, t\right)\,=\,-2F\left(A^{\#}\left(\phi_{t}\left(\xi\right),\, t\right)\right)\, A_{ij}\left(\phi_{t}\left(\xi\right),\, t\right)\label{83}
\end{equation}
\[
=-2\Big|X\left(\phi_{t}\left(\xi\right),\, t\right)\Big|{}^{-1}F\left(\Big|X\left(\phi_{t}\left(\xi\right),\, t\right)\Big|\, A^{\#}\left(\phi_{t}\left(\xi\right),\, t\right)\right)\, A_{ij}\left(\phi_{t}\left(\xi\right),\, t\right)
\]
\begin{equation}
g_{ij}\left(\xi,\,-1\right)=\delta_{ij}+\partial_{i}\mathtt{u}\left(\xi,\,-1\right)\,\partial_{j}\mathtt{u}\left(\xi,\,-1\right)\label{84}
\end{equation}
where the second fundamental form $A_{t}\left(x\right)\sim A_{ij}\left(x,\, t\right)$
is equal to
\begin{equation}
A_{ij}\left(x,\, t\right)=\frac{\partial_{ij}^{2}\mathtt{u}\left(x,\, t\right)}{\sqrt{1+|\partial_{x}\mathtt{u}\left(x,\, t\right)|^{2}}}\label{85}
\end{equation}
By (\ref{85}), (\ref{6}), (\ref{7}), (\ref{8}), (\ref{17}) and
the comparability of $|X\left(x,\, t\right)|$ and $|\hat{X}|$, the
$\ell^{2}$ norm of the matrix $\partial_{t}g_{ij}\left(\xi,\, t\right)$
satisfies 
\begin{equation}
|\partial_{t}g_{ij}\left(\xi,\, t\right)|\,\leq\, C\left(n,\,\mathcal{C},\,\parallel F\parallel_{C^{1}\left(U\right)}\right)\,|\hat{X}|{}^{-2}\label{86}
\end{equation}
So by (\ref{84}), (\ref{6}), (\ref{8}) and (\ref{86}), the pull-back
metric $g_{ij}\left(\xi,\, t\right)$ is equivalent to the dot product
$\delta_{ij}$. 

Let $\Gamma_{ij}^{k}\left(\xi,\, t\right)$ be the Christoffel symbols
associated with the metric $g_{ij}\left(\xi,\, t\right)$, then we
have 
\begin{equation}
\partial_{t}\Gamma_{ij}^{k}\,=\,\frac{1}{2}g^{kl}\left(\nabla_{i}\overset{\boldsymbol{\cdot}}{g}_{lj}+\nabla_{j}\overset{\boldsymbol{\cdot}}{g}_{il}-\nabla_{l}\overset{\boldsymbol{\cdot}}{g}_{ij}\right)\label{87}
\end{equation}
 
\begin{equation}
\Gamma_{ij}^{k}\left(\xi,\,-1\right)\,=\,\frac{\partial_{k}\mathtt{u}\left(\xi,\,-1\right)\,\partial_{ij}^{2}\mathtt{u}\left(\xi,\,-1\right)}{1+|\partial_{x}\mathtt{u}\left(\xi,\,-1\right)|{}^{2}}\label{88}
\end{equation}
where 
$$\overset{\boldsymbol{\cdot}}{g}_{ij}=\partial_{t}g_{ij}=-2F\left(A^{\#}\right)A_{ij}$$
Similarly, and also by (\ref{20}), the homogeneity of the derivative
of $F$, the equivalence of $g_{ij}$ and $\delta_{ij}$, we have
\[
|\partial_{t}\Gamma_{ij}^{k}|\,\leq C\left(n,\,\mathcal{C},\,\parallel F\parallel_{C^{1}\left(U\right)}\right)\,|\hat{X}|{}^{-3}
\]
\[
|\Gamma_{ij}^{k}\left(\xi,\,-1\right)|\,\leq C\left(n,\,\mathcal{C},\,\parallel F\parallel_{C^{1}\left(U\right)}\right)\,|\hat{X}|{}^{-1}
\]
which implies 
\begin{equation}
|\Gamma_{ij}^{k}\left(\xi,\, t\right)|\,\leq C\left(n,\,\mathcal{C},\,\parallel F\parallel_{C^{1}\left(U\right)}\right)\,|\hat{X}|{}^{-1}\label{89}
\end{equation}

Now consider the deviation $h$ in the local coordinates $\left(\xi,\, t\right)$,
then the equation in Proposition \ref{p14} becomes
\begin{equation}
\Big|\partial_{t}h-\left\{ \partial_{\xi_{i}}\left(\mathbf{a}^{ij}\left(\xi,\, t\right)\partial_{\xi_{j}}h\right)+\Gamma_{ik}^{i}\left(\xi,\, t\right)\,\mathbf{a}^{kj}\left(\xi,\, t\right)\partial_{\xi_{j}}h\right\} \Big|\label{90}
\end{equation}
\[
\leq C\left(n,\,\mathcal{C},\,\parallel F\parallel_{C^{3}\left(U\right)}\right)\left(|\hat{X}|{}^{-1}|\partial_{\xi}h|\,+\,|\hat{X}|{}^{-2}|h|\right)
\]
\begin{equation*}
h\left(\xi,\,0\right)=0
\end{equation*}
where $\mathbf{a}^{ij}\left(\xi,\, t\right)=\mathbf{a}^{ji}\left(\xi,\, t\right)$
satisfies (by Proposition \ref{p15} and (\ref{89})) 

\noindent \resizebox{1.009\linewidth}{!}{
  \begin{minipage}{\linewidth}
  \begin{align}
\frac{\lambda}{C\left(n,\,\mathcal{C},\,\parallel F\parallel_{C^{3}\left(U\right)}\right)}\delta^{ij}\leq\,\frac{\lambda}{3}g^{ij}\left(\xi,\, t\right)\,\leq\,\mathbf{a}^{ij}\left(\xi,\, t\right)\,\leq\,\frac{3}{\lambda}g^{ij}\left(\xi,\, t\right)\,\leq\frac{C\left(n,\,\mathcal{C},\,\parallel F\parallel_{C^{3}\left(U\right)}\right)}{\lambda}\delta^{ij}\label{92}
\end{align}
  \end{minipage}
}\\

\begin{equation}
|\hat{X}|\Big|\partial_{\xi}\mathbf{a}^{ij}\left(\xi,\, t\right)\Big|\,+\,|\hat{X}|^{2}\Big|\partial_{t}\mathbf{a}^{ij}\left(\xi,\, t\right)\Big|\,\leq C\left(n,\,\mathcal{C},\,\parallel F\parallel_{C^{3}\left(U\right)},\,\lambda\right)\label{93}
\end{equation}
Thus, by (\ref{89}), (\ref{92}), (\ref{84}) and (\ref{86}), the
equation (\ref{90}) is equivalent to 
\begin{equation}
\Big|\partial_{t}h-\partial_{\xi_{i}}\left(\mathbf{a}^{ij}\left(\xi,\, t\right)\partial_{\xi_{j}}h\right)\Big|\,\label{94}
\end{equation}
\[
\leq C\left(n,\,\mathcal{C},\,\parallel F\parallel_{C^{3}\left(U\right)},\,\lambda\right)\left(|\hat{X}|{}^{-1}|\partial_{\xi}h|\,+\,|\hat{X}|{}^{-2}|h|\right)
\]
for $\left(\xi,\, t\right)\in B_{\frac{\boldsymbol{\rho}}{2}\mid\hat{X}\mid}^{n}\times\left[-1,\,0\right]$
\begin{equation*}
h\left(\xi,\,0\right)=0
\end{equation*}
\\
Let's consider the following change of variables:
\[
\left(\xi,\, t\right)=\Xi\left(\bar{\xi},\,\bar{t}\right)\equiv\left(\left(\frac{\boldsymbol{\rho}}{2}|\hat{X}|\right)\,\bar{\xi},\;\left(\frac{\boldsymbol{\rho}}{2}|\hat{X}|\right)^{2}\,\bar{t}\right)
\]
and let $\bar{h}=h\circ\Xi$, $\mathbf{\bar{a}}^{ij}=\mathbf{a}^{ij}\circ\Xi$.
Then (\ref{94}) in the new variables becomes 
\begin{equation}
\Big|\partial_{\bar{t}}\bar{h}-\partial_{\bar{\xi}_{i}}\left(\bar{\mathbf{a}}^{ij}\left(\bar{\xi,\,}\bar{t}\right)\partial_{\bar{\xi}_{j}}\bar{h}\right)\Big|\,\leq C\left(n,\,\mathcal{C},\,\parallel F\parallel_{C^{3}\left(U\right)},\,\lambda,\,\boldsymbol{\rho}\right)\left(|\partial_{\bar{\xi}}\bar{h}|\,+\,|\bar{h}|\right)\label{95}
\end{equation}
\begin{equation*}
\bar{h}\Big|_{\bar{t}=0}=0
\end{equation*}
and (\ref{92}), (\ref{93}) are translated into 
\begin{equation}
\frac{\lambda}{C\left(n,\,\mathcal{C},\,\parallel F\parallel_{C^{3}\left(U\right)}\right)}\delta^{ij}\,\leq\,\mathbf{\bar{a}}^{ij}\left(\bar{\xi,\,}\bar{t}\right)\,\leq\,\frac{C\left(n,\,\mathcal{C},\,\parallel F\parallel_{C^{3}\left(U\right)}\right)}{\lambda}\delta^{ij}\label{97}
\end{equation}
\begin{equation}
\Big|\partial_{\bar{\xi}}\mathbf{\bar{a}}^{ij}\left(\bar{\xi,\,}\bar{t}\right)\Big|\,+\,\Big|\partial_{\bar{t}}\mathbf{\bar{a}}^{ij}\left(\bar{\xi,\,}\bar{t}\right)\Big|\,\leq\, C\left(n,\,\mathcal{C},\,\parallel F\parallel_{C^{3}\left(U\right)},\,\lambda,\,\boldsymbol{\rho}\right)\label{98}
\end{equation}
for $\bar{\xi}\in B_{1}^{n},\;\bar{t}\in\left[-\left(\frac{\boldsymbol{\rho}}{2}|\hat{X}|\right)^{-2},\,0\right]$. 

Applying Proposition \ref{p20} to $\bar{h}\left(\bar{\xi},\,\bar{t}\right)$,
we may conclude that there exist $\tilde{\Lambda}=\tilde{\Lambda}\left(n,\,\mathcal{C},\,\parallel F\parallel_{C^{3}\left(U\right)},\,\lambda\right)>0$,
$\alpha=\alpha\left(n,\,\mathcal{C},\,\parallel F\parallel_{C^{3}\left(U\right)},\,\lambda\right)\in\left(0,\,1\right)$
for which the following holds:
\begin{equation}
|\partial_{\bar{\xi}}\bar{h}|\,+\,|\bar{h}|\label{99}
\end{equation}
\[
\leq\tilde{\Lambda}\,\exp\left(\frac{1}{\tilde{\Lambda}\bar{t}}\right)\,\left(\parallel\partial_{\bar{\xi}}\bar{h}\parallel_{L^{\infty}\left(B_{1}^{n}\times\left[-\left(\frac{\boldsymbol{\rho}}{2}\hat{|X}|\right)^{-2},\,0\right]\right)}+\parallel\bar{h}\parallel_{L^{\infty}\left(B_{1}^{n}\times\left[-\left(\frac{\boldsymbol{\rho}}{2}|\hat{X}|\right)^{-2},\,0\right]\right)}\right)
\]
for $\left(\bar{\xi},\,\bar{t}\right)\in B_{\nicefrac{1}{4}}^{n}\times[-\alpha\left(\frac{\boldsymbol{\rho}}{2}|\hat{X}|\right)^{-2},\,0)$.\\
Undoing change of variables, (\ref{99}) becomes
\begin{equation}
\frac{\boldsymbol{\rho}}{2}\,|\hat{X}|\,|\partial_{\xi}h|\,+\,|h|\label{100}
\end{equation}
\[
\leq\tilde{\Lambda}\,\exp\left(\frac{|\hat{X}|{}^{2}}{\tilde{\Lambda}t}\right)\,\left(\frac{\boldsymbol{\rho}}{2}|\hat{X}|\parallel\partial_{\xi}h\parallel_{L^{\infty}\left(B_{\frac{\boldsymbol{\rho}}{2}|\hat{X}|}^{n}\times\left[-1,\,0\right]\right)}+\parallel h\parallel_{L^{\infty}\left(B_{\frac{\boldsymbol{\rho}}{2}|\hat{X}|}^{n}\times\left[-1,\,0\right]\right)}\right)
\]
for $\left(\xi,\, t\right)\in B_{\frac{\boldsymbol{\rho}}{8}|\hat{X}|}^{n}\times[-\alpha,\,0)$.\\
Note that the pull-back metric $g_{ij}\left(\xi,\, t\right)$ is equivalent
to the dot product $\delta_{ij}$ and that $|X\left(x,\, t\right)|$
is comparable with $|\hat{X}|$. The conclusion follows immediately.
\end{proof}
Next, we would go from the exponential decay to identically vanishing
of the deviation $h$ outside a compact set. To this end, we have
to derive a different type of Carleman's inequality on the flow $\left\{ \Sigma_{t}\right\} _{-1\leq t\leq0}$,
which is done through two lemmas. The first lemma is a modification
of the integral equality in \cite{EF}.
\begin{lem}\label{l23}
Let $\left(\mathrm{M},\,\boldsymbol{g}_{t}\right)$ be a flow of Riemannian
manifolds and $P$ be a differential operator on the flow defined
by 
\[
P\mathrm{v}=\partial_{t}\mathrm{v}-\nabla_{\boldsymbol{g}_{t}}\cdot\left(a_{t}\, d\mathrm{v}\right)=\partial_{t}\mathrm{v}-\nabla_{i}\left(a^{ij}\left(\cdot,\, t\right)\nabla_{j}\mathrm{v}\right)
\]
where $a_{t}=a(\cdot,\, t)$ is a symmetric 2-tensor on $\left(\mathrm{M},\,\boldsymbol{g}_{t}\right)$.
Then given functions $G,\,\Psi\in C^{2,1}\left(\mathrm{M}\times\left[-T,\,0\right]\right)$
with $G>0$, define a function $\Phi$ as 
\begin{equation}
\Phi=\frac{\partial_{t}G+\nabla_{i}\left(a^{ij}\nabla_{j}G\right)+\,\frac{1}{2}\textrm{tr}\left(\partial_{t}\boldsymbol{g}\right)\, G}{G}\label{101}
\end{equation}
\[
=\partial_{t}\ln G\,+\,\nabla_{i}\left(a^{ij}\nabla_{j}\ln G\right)\,+\, a^{ij}\nabla_{i}\ln G\,\nabla_{j}\ln G\,+\,\frac{1}{2}\textrm{tr}\left(\partial_{t}\boldsymbol{g}\right)
\]
and a 2-tensor $\Upsilon$ as
\begin{equation}
\Upsilon^{ij}=a^{ik}a^{jl}\nabla_{kl}^{2}\ln G\,-\,\frac{1}{2}\partial_{t}a^{ij}\label{102}
\end{equation}
\[
+\,\frac{1}{2}\left(a^{ik}\nabla_{k}a^{jl}+a^{jk}\nabla_{k}a^{il}-a^{lk}\nabla_{k}a^{ij}\right)\,\nabla_{l}\ln G
\]
It follows that for any $\mathscr{\mathrm{u}}\in C_{\textrm{c}}^{2,1}\left(\mathrm{M}\times\left[-T,\,0\right]\right)$,
there holds

\noindent \resizebox{1.009\linewidth}{!}{
  \begin{minipage}{\linewidth}
  \begin{align*}
\int_{\mathrm{\mathrm{M}}}\left\{ \left(2\Upsilon^{ij}-\left(\Phi-\Psi\right)a^{ij}\right)\nabla_{i}\mathrm{u}\,\nabla_{j}\mathscr{\mathrm{u}}\,+\,\frac{1}{2}\left(\partial_{t}\Psi-\nabla_{i}\left(a^{ij}\nabla_{j}\Psi\right)+\left(\Phi-\Psi\right)\Psi\right)\mathscr{\mathrm{u}}^{2}\right\} G\, d\mu_{t}
\end{align*}
  \end{minipage}
}

\begin{equation}
=\{\int_{\mathrm{M}}2\, P\mathscr{\mathrm{u}}\,\left(\partial_{t}\mathscr{\mathrm{u}}+\, a^{ij}\nabla_{i}\ln G\,\nabla_{j}\mathscr{\mathrm{u}}+\,\frac{1}{2}\Psi\mathscr{\mathrm{u}}\right)G\, d\mu_{t}\label{103}
\end{equation}
\[
-\int_{\mathrm{M}}2\left(\partial_{t}\mathscr{\mathrm{u}}+\, a^{ij}\nabla_{i}\ln G\,\nabla_{j}\mathscr{\mathrm{u}}+\,\frac{1}{2}\Psi\mathscr{\mathrm{u}}\right)^{2}G\, d\mu_{t}
\]
\[
-\partial_{t}\left\{ \int_{\mathrm{M}}\left(a^{ij}\nabla_{i}\mathrm{u}\,\nabla_{j}\mathrm{u}\,-\,\frac{1}{2}\Psi\mathscr{\mathrm{u}}^{2}\right)G\, d\mu_{t}\right\} \}
\]
where $\mu_{t}$ is the volume form of \textup{$\left(\mathrm{M},\,\boldsymbol{g}_{t}\right)$.}\end{lem}

\begin{proof}
Let's begin with
\begin{equation}
\partial_{t}\left\{ \int_{M}a^{ij}\nabla_{i}\mathscr{\mathrm{u}}\,\nabla_{j}\mathscr{\mathrm{u}}\, G\, d\mu_{t}\right\} \label{104}
\end{equation}

\noindent \resizebox{1.009\linewidth}{!}{
  \begin{minipage}{\linewidth}
  \begin{align*}
=\int_{\mathrm{\mathrm{M}}}\left\{ 2a^{ij}\nabla_{j}\mathscr{\mathrm{u}}\,\nabla_{i}\partial_{t}\mathscr{\mathrm{u}}\, G+\, a^{ij}\nabla_{i}\mathscr{\mathrm{u}}\,\nabla_{j}\mathscr{\mathrm{u}}\left(\partial_{t}G+\frac{1}{2}\textrm{tr}\left(\partial_{t}g\right)\, G\right)+\,\partial_{t}a^{ij}\nabla_{i}\mathscr{\mathrm{u}}\,\nabla_{j}\mathscr{\mathrm{u}}\, G\right\} d\mu_{t}
\end{align*}
  \end{minipage}
}\\\\
in which we use the commutativity:
\[
\partial_{t}\, d\mathscr{\mathrm{u}}=d\,\partial_{t}\mathscr{\mathrm{u}},\quad \textrm{where}\quad d\mathrm{u}\sim\nabla_{i}\mathscr{\mathrm{u}}
\]
and the evolution equation of the volume form:
\begin{equation}
\partial_{t}\, d\mu_{t}=\frac{1}{2}\textrm{tr}\left(\partial_{t}\boldsymbol{g}\right)\, d\mu_{t}\label{105}
\end{equation}

Applying integration by parts on $\left(\mathrm{M},\,\boldsymbol{g}_{t}\right)$,
(\ref{104}) becomes 

\noindent \resizebox{1.009\linewidth}{!}{
  \begin{minipage}{\linewidth}
  \begin{align*}
\{\int_{\mathrm{\mathrm{M}}}-2\left(\nabla_{i}\left(a^{ij}\nabla_{j}\mathscr{\mathrm{u}}\right)+\, a^{ij}\nabla_{i}\ln G\,\nabla_{j}\mathscr{\mathrm{u}}\right)\partial_{t}\mathscr{\mathrm{u}}\, G\, d\mu_{t}+\int_{\mathrm{\mathrm{M}}}a^{ij}\nabla_{i}\mathscr{\mathrm{u}}\,\nabla_{j}\mathrm{u}\left(\partial_{t}G+\nabla_{k}\left(a^{kl}\nabla_{l}G\right)+\,\frac{1}{2}\textrm{tr}\left(\partial_{t}g\right)\, G\right)\, d\mu_{t}
\end{align*}
  \end{minipage}
}

\begin{equation}
-\int_{\mathrm{\mathrm{M}}}a^{ij}\nabla_{i}\mathrm{u}\,\nabla_{j}\mathrm{u}\,\nabla_{k}\left(a^{kl}\nabla_{l}G\right)\, d\mu_{t}+\int_{\mathrm{\mathrm{M}}}\partial_{t}a^{ij}\nabla_{i}\mathscr{\mathrm{u}}\,\nabla_{j}\mathscr{\mathrm{u}}\, G\, d\mu_{t}\}\label{106}
\end{equation}
By (\ref{101}), integration by parts twice and the symmetry of $a_{t}$,
(\ref{106}) becomes
\begin{equation}
\{-2\int_{\mathrm{\mathrm{M}}}\left(\nabla_{i}\left(a^{ij}\nabla_{j}\mathscr{\mathrm{u}}\right)+\, a^{ij}\nabla_{i}\ln G\,\nabla_{j}\mathrm{u}\right)\partial_{t}\mathscr{u}\, G\, d\mu_{t}+\int_{\mathrm{\mathrm{M}}}a^{ij}\nabla_{i}\mathscr{\mathrm{u}}\,\nabla_{j}\mathscr{\mathrm{u}}\,\Phi G\, d\mu_{t}\label{107}
\end{equation}

\noindent \resizebox{1.009\linewidth}{!}{
  \begin{minipage}{\linewidth}
  \begin{align*}
+\int_{\mathrm{\mathrm{M}}}\left\{ \nabla_{k}a^{ij}\nabla_{i}\mathscr{\mathrm{u}}\,\nabla_{j}\mathscr{\mathrm{u}}\, a^{kl}\nabla_{l}\ln G-\,2\nabla_{j}\left(a^{ij}\nabla_{i}\mathscr{\mathrm{u}}\right)\nabla_{k}\mathrm{u}\, a^{kl}\nabla_{l}\ln G-\,2a^{ij}\nabla_{i}\mathscr{\mathrm{u}}\,\nabla_{k}\mathscr{\mathrm{u}}\,\nabla_{j}a^{kl}\nabla_{l}\ln G\right\} G\, d\mu_{t}
\end{align*}
  \end{minipage}
}

\[
-2\int_{\mathrm{\mathrm{M}}}a^{ij}\nabla_{i}\mathscr{\mathrm{u}}\,\nabla_{k}\mathscr{\mathrm{u}}\, a^{kl}\nabla_{jl}^{2}G\, d\mu_{t}+\,\int_{\mathrm{\mathrm{M}}}\partial_{t}a^{ij}\nabla_{i}\mathrm{u}\,\nabla_{j}\mathscr{\mathrm{u}}\, G\, d\mu_{t}\}
\] 

Then we reorganize (\ref{107}) (in order to make up the term $P\mathscr{u}$) to get

\noindent \resizebox{1.009\linewidth}{!}{
  \begin{minipage}{\linewidth}
  \begin{align}
\{2\int_{\mathrm{\mathrm{M}}}\left\{ \left(\partial_{t}\mathrm{u}-\nabla_{i}\left(a^{ij}\nabla_{j}\mathrm{u}\right)\right)\left(\partial_{t}\mathrm{u}+a^{kl}\nabla_{k}\ln G\,\nabla_{l}\mathscr{\mathrm{u}}\right)-\,\left(\partial_{t}\mathscr{\mathrm{u}}\right)^{2}-\,2a^{ij}\nabla_{i}\ln G\,\nabla_{j}\mathscr{\mathrm{u}}\,\partial_{t}\mathscr{\mathrm{u}}\right\} G\, d\mu_{t}\label{108}
\end{align}
  \end{minipage}
}

\[
+\int_{\mathrm{\mathrm{M}}}\Phi a^{ij}\nabla_{i}\mathscr{\mathrm{u}}\,\nabla_{j}\mathscr{\mathrm{u}}\, G\, d\mu_{t}-\,2\int_{\mathrm{\mathrm{M}}}a^{ij}a^{kl}\left(\nabla_{jl}^{2}\ln G+\nabla_{j}\ln G\,\nabla_{l}\ln G\right)\nabla_{i}\mathscr{\mathrm{u}}\,\nabla_{k}\mathscr{\mathrm{u}}\, G\, d\mu_{t}
\]

\noindent \resizebox{1.009\linewidth}{!}{
  \begin{minipage}{\linewidth}
  \begin{align*}
+\int_{\mathrm{\mathrm{M}}}\left\{ a^{kl}\nabla_{k}a^{ij}\,\nabla_{l}\ln G\nabla_{i}\mathscr{\mathrm{u}}\,\nabla_{j}\mathrm{u}-\,2a^{ij}\nabla_{j}a^{kl}\,\nabla_{l}\ln G\,\nabla_{i}\mathrm{u}\,\nabla_{k}\mathscr{\mathrm{u}}+\,\partial_{t}a^{ij}\nabla_{i}\mathscr{\mathrm{u}}\,\nabla_{j}\mathscr{\mathrm{u}}\right\} G\, d\mu_{t}\}
\end{align*}
  \end{minipage}
}\\\\
By (\ref{102}), (\ref{108}) becomes 

\noindent \resizebox{1.009\linewidth}{!}{
  \begin{minipage}{\linewidth}
  \begin{align*}
\{2\int_{\mathrm{\mathrm{M}}}\left\{ \left(\partial_{t}\mathscr{\mathrm{u}}-\nabla_{i}\left(a^{ij}\nabla_{j}\mathrm{u}\right)\right)\left(\partial_{t}\mathrm{u}+a^{kl}\nabla_{k}\ln G\,\nabla_{l}\mathrm{u}\right)-\,\left(\partial_{t}\mathscr{\mathrm{u}}+a^{ij}\nabla_{i}\ln G\,\nabla_{j}\mathscr{\mathrm{u}}\right)^{2}\right\} G\, d\mu_{t}
\end{align*}
  \end{minipage}
}

\[
+\int_{\mathrm{\mathrm{M}}}\Phi a^{ij}\nabla_{i}\mathscr{\mathrm{u}}\,\nabla_{j}\mathscr{\mathrm{u}}\, G\, d\mu_{t}-\,2\int_{\mathrm{\mathrm{M}}}\Upsilon^{ij}\nabla_{i}\mathscr{\mathrm{u}}\,\nabla_{j}\mathscr{\mathrm{u}}\, G\, d\mu_{t}\}
\]

\noindent \resizebox{1.009\linewidth}{!}{
  \begin{minipage}{\linewidth}
  \begin{align*}
=\{2\int_{\mathrm{\mathrm{M}}}P\mathscr{\mathrm{u}}\,\left(\partial_{t}\mathscr{\mathrm{u}}+\, a^{ij}\nabla_{i}\ln G\,\nabla_{j}\mathscr{\mathrm{u}}+\,\frac{1}{2}\Psi\mathscr{\mathrm{u}}\right)G\, d\mu_{t}-\,\int_{\mathrm{\mathrm{M}}}\left(\partial_{t}\mathrm{u}-\,\nabla_{i}\left(a^{ij}\nabla_{j}\mathrm{u}\right)\right)\Psi\mathscr{\mathrm{u}}\, G\, d\mu_{t}
\end{align*}
  \end{minipage}
}

\noindent \resizebox{1.009\linewidth}{!}{
  \begin{minipage}{\linewidth}
  \begin{align*}
-2\int_{\mathrm{\mathrm{M}}}\left(\partial_{t}\mathscr{\mathrm{u}}+\, a^{ij}\nabla_{i}\ln G\,\nabla_{j}\mathrm{u}+\frac{1}{2}\Psi\mathscr{\mathrm{u}}\right)^{2}G\, d\mu_{t}+2\int_{\mathrm{\mathrm{M}}}\left(\partial_{t}\mathscr{\mathrm{u}}\,+a^{ij}\nabla_{i}\ln G\,\nabla_{j}\mathscr{\mathrm{u}}+\,\frac{1}{2}\Psi\mathscr{\mathrm{u}}\right)\Psi\mathscr{\mathrm{u}}\, G\, d\mu_{t}
\end{align*}
  \end{minipage}
}

\begin{equation}
-\frac{1}{2}\int_{\mathrm{\mathrm{M}}}\Psi^{2}\mathscr{\mathrm{u}}^{2}G\, d\mu_{t}-\int_{\mathrm{\mathrm{M}}}\left(2\Upsilon^{ij}-\Phi a^{ij}\right)\nabla_{i}\mathscr{\mathrm{u}}\,\nabla_{j}\mathscr{\mathrm{u}}\, G\, d\mu_{t}\}\label{109}
\end{equation}
For the second term of (\ref{109}), by the product rule and integration
by parts, we get
\begin{equation}
-\int_{\mathrm{\mathrm{M}}}\left(\partial_{t}\mathrm{u}-\,\nabla_{i}\left(a^{ij}\nabla_{j}\mathrm{u}\right)\right)\mathscr{\mathrm{u}}\,\Psi G\, d\mu_{t}\label{110}
\end{equation}
\[
=-\frac{1}{2}\int_{\mathrm{\mathrm{M}}}\left(\partial_{t}\mathscr{\mathrm{u}}^{2}-\,\nabla_{i}\left(a^{ij}\nabla_{j}\mathscr{\mathrm{u}}^{2}\right)+\,2a^{ij}\nabla_{i}\mathscr{\mathrm{u}}\,\nabla_{j}\mathscr{\mathrm{u}}\right)\Psi G\, d\mu_{t}
\]
\[
=\{\frac{1}{2}\int_{\mathrm{\mathrm{M}}}\left(\partial_{t}\Psi\, G+\,\Psi\left(\partial_{t}G+\frac{1}{2}\textrm{tr}\left(\partial_{t}g\right)\, G\right)\right)\mathscr{\mathrm{u}}^{2}\, d\mu_{t}-\,\partial_{t}\left(\int_{\mathrm{\mathrm{M}}}\frac{1}{2}\Psi^{2}\mathscr{\mathrm{u}}^{2}G\, d\mu_{t}\right)
\]

\noindent \resizebox{1.009\linewidth}{!}{
  \begin{minipage}{\linewidth}
  \begin{align*}
-\,\int_{\mathrm{\mathrm{M}}}a^{ij}\nabla_{i}\mathscr{\mathrm{u}}\,\nabla_{j}\mathscr{\mathrm{u}}\,\Psi G\, d\mu_{t}
\,+\,\frac{1}{2}\int_{\mathrm{\mathrm{M}}}\left\{ \nabla_{j}\left(a^{ij}\nabla_{i}\Psi\right)G+\,2a^{ij}\nabla_{i}G\,\nabla_{j}\Psi+\,\Psi\nabla_{j}\left(a^{ij}\nabla_{i}G\right)\right\} \mathscr{\mathrm{u}}^{2}\, d\mu_{t}\}
\end{align*}
  \end{minipage}
}

\[
=\{\frac{1}{2}\int_{\mathrm{\mathrm{M}}}\left(\partial_{t}\Psi+\,\nabla_{j}\left(a^{ij}\nabla_{i}\Psi\right)+\,\Phi\Psi+\, a^{ij}\nabla_{i}\ln G\,\nabla_{j}\Psi\right)\mathscr{\mathrm{u}}^{2}G\, d\mu_{t}
\]
\[
-\,\int_{\mathrm{\mathrm{M}}}\Psi a^{ij}\nabla_{i}\mathscr{\mathrm{u}}\,\nabla_{j}\mathscr{\mathrm{u}}\, G\, d\mu_{t}
\,-\,\partial_{t}\left(\int_{\mathrm{\mathrm{M}}}\frac{1}{2}\Psi^{2}\mathscr{\mathrm{u}}^{2}G\, d\mu_{t}\right)\}
\]
Likewise, for the fourth term of (\ref{109}), we have 
\begin{equation}
2\int_{\mathrm{\mathrm{M}}}\left(\partial_{t}\mathscr{\mathrm{u}}+\, a^{ij}\nabla_{i}\ln G\,\nabla_{j}\mathscr{\mathrm{u}}+\,\frac{1}{2}\Psi\mathscr{\mathrm{u}}\right)\Psi\mathscr{\mathrm{u}}\, G\, d\mu_{t}\label{111}
\end{equation}
\[
=\int_{\mathrm{\mathrm{M}}}\partial_{t}\mathscr{\mathrm{u}}^{2}\,\Psi G\, d\mu_{t}+\,\int_{\mathrm{\mathrm{M}}}a^{ij}\nabla_{i}G\,\nabla_{j}\mathscr{\mathrm{u}}^{2}\,\Psi\, d\mu_{t}+\,\int_{\mathrm{\mathrm{M}}}\Psi^{2}\mathrm{u}^{2}G\, d\mu_{t}
\]
\[
=\{-\int_{\mathrm{\mathrm{M}}}\left(\partial_{t}\Psi\, G+\,\Psi\left(\partial_{t}G+\,\frac{1}{2}\textrm{tr}\partial_{t}g\, G\right)\right)\mathscr{\mathrm{u}}^{2}\, d\mu_{t}+\,\partial_{t}\left(\int_{\mathrm{\mathrm{M}}}\Psi\mathscr{\mathrm{u}}^{2}G\, d\mu_{t}\right)
\]
\[
+\,\int_{\mathrm{\mathrm{M}}}\Psi^{2}\mathscr{\mathrm{u}}^{2}G\, d\mu_{t}\,-\,\int_{\mathrm{\mathrm{M}}}\left(\nabla_{j}\left(a^{ij}\nabla_{i}G\right)\Psi+\, a^{ij}\nabla_{i}G\,\nabla_{j}\Psi\right)\mathscr{\mathrm{u}}^{2}\, d\mu_{t}\}
\]
\[
=-\int_{\mathrm{\mathrm{M}}}\left(\partial_{t}\Psi+\,\Phi\Psi+\, a^{ij}\nabla_{i}\ln G\,\nabla_{j}\Psi-\,\Psi^{2}\right)\mathrm{u}^{2}G\, d\mu_{t}+\,\partial_{t}\left(\int_{\mathrm{\mathrm{M}}}\Psi\mathscr{\mathrm{u}}^{2}G\, d\mu_{t}\right)
\]
Combining (\ref{109}), (\ref{110}), (\ref{111}) to get (\ref{103}).
\end{proof}
We hereafter consider the Riemannian manifold in Lemma \ref{l23} to be the
time-slice $\Sigma_{t}$ with the induced metric $g_{t}$, which evolves
(in ``normal parametrization'') like 
$$\partial_{t}g=-2F\left(A^{\#}\right)A$$
(see \cite{A}). The differential operator in Lemma \ref{l23} is taken to be
the one in Proposition \ref{p14}.

For the second lemma, we would choose suitable weight function $G$
and auxiliary function $\Psi$ in Lemma \ref{l23} in order to bound the LHS
of (\ref{103}) from below. The choice of $G$ is due to \cite{ESS}
and \cite{W}. As for $\Psi$, it is not shown in \cite{W} but is
used here in order to deal with the last term in (\ref{102}), which comes
from the nonlinear nature of $F$. Note that in
the linear case when $F\left(S\right)=\textrm{tr}\left(S\right)$
(see \cite{W}), the coefficients of the differential operator in
Proposition \ref{p14} becomes $\mathbf{a}^{ij}=g^{ij}$; besides, (\ref{102})
is reduced to 
\[
\Upsilon^{ij}=g^{ik}g^{jl}\nabla_{kl}^{2}\ln G\,-\, HA^{ij}
\]
The idea of using an auxiliary function for the nonlinear case is motivated by \cite{N}. 

\begin{lem}\label{l24}
Assume that $\varkappa\leq6^{-4}\lambda^{3}$ in (\ref{1}) and (\ref{2}).
Then there exists $R=R\left(\Sigma,\,\tilde{\Sigma},\mathcal{\, C},\, U,\,\parallel F\parallel_{C^{3}\left(U\right)},\,\lambda,\,\varkappa\right)\geq1$
so that for any constants $M\geq1,\;\tau\in(0,\,1]$, let
\begin{equation}
G=\exp\left(M\left(t+\tau\right)|X|^{\frac{3}{2}}+\,|X|^{2}\right)\label{112}
\end{equation}
\begin{equation}
\Psi=\{\left(\frac{3}{2}M\left(t+\tau\right)|X|^{-\frac{1}{2}}+2\right)^{2}\mathbf{a}^{ij}\left(X\cdot\partial_{i}X\right)\left(X\cdot\partial_{j}X\right)+\, M|X|^{\frac{3}{2}}\label{113}
\end{equation}
\[
+\frac{1}{2}\left(\frac{3}{2}M\left(t+\tau\right)|X|^{-\frac{1}{2}}+2\right)\left(\textrm{tr}\left(\mathbf{a}\right)-\frac{\lambda}{3}\right)
\]
\[
+\left(\textrm{tr}\left(\mathbf{a}\right)-\frac{\lambda}{3}\right)+\,\frac{3}{4}M\left(t+\tau\right)|X|^{-\frac{5}{2}}\left(\textrm{tr}\left(\mathbf{a}\right)|X|{}^{2}-\,\mathbf{a}^{ij}\left(X\cdot\partial_{i}X\right)\left(X\cdot\partial_{j}X\right)\right)\}
\]
(note that $G>0$ and $\Psi\geq0$), there hold
\begin{equation}
2\Upsilon^{ij}-\left(\Phi-\Psi\right)\mathbf{a}^{ij}\,\geq\,\frac{\lambda^{2}}{9}g^{ij}\label{114}
\end{equation}
\begin{equation}
\frac{1}{2}\left(\partial_{t}\Psi-\,\nabla_{i}\left(\mathbf{a}^{ij}\nabla_{j}\Psi\right)+\left(\,\Phi-\Psi\right)\Psi\right)\,\geq\,\frac{\lambda^{2}}{9}|X|{}^{2}\label{115}
\end{equation}
for $X\in\Sigma_{t}\setminus\bar{B}_{R},\; t\in[-\tau,\,0)$, where
\textup{$\textrm{tr}\left(\mathbf{a}\right)=g_{ij}\mathbf{a}^{ij}$,}
$\Phi$ and \textup{$\Upsilon^{ij}$} are defined in (\ref{101})
and (\ref{102}), respectively, with the covariant derivative is taken
w.r.t $\Sigma_{t}$, $\partial_{t}g=-2F\left(A^{\#}\right)A$, and
$a^{ij}=\mathbf{a}^{ij}$. \end{lem}

\begin{rem}\label{r25}
In view of (\ref{58}), the hypothesis that $\varkappa\leq6^{-4}\lambda^{3}$
amounts to requiring the smallness of $|X|\,|\nabla_{\Sigma_{t}}\mathbf{a}|$
(compared with the ellipticity of $\mathbf{a}$). Similar hypothesis
also appears in \cite{N} and \cite{WZ} when using Carleman's inequalities
to prove the backward uniqueness of parabolic equations. 
\end{rem}

\begin{proof}
Let's start with computing the covariant derivatives of $\ln G$:
\begin{equation}
\nabla_{i}\ln G=\left(\frac{3}{2}M\left(t+\tau\right)|X|{}^{-\frac{1}{2}}+2\right)\left(X\cdot\partial_{i}X\right)\label{116}
\end{equation}
\begin{equation}
\nabla_{ij}^{2}\ln G=\{\left(\frac{3}{2}M\left(t+\tau\right)|X|{}^{-\frac{1}{2}}+2\right)\left(g_{ij}+\, X\cdot N\, A_{ij}\right)\label{117}
\end{equation}
\[
-\frac{3}{4}M\left(t+\tau\right)|X|^{-\frac{5}{2}}\left(|X|{}^{2}g_{ij}-\,\left(X\cdot\partial_{i}X\right)\left(X\cdot\partial_{j}X\right)\right)
\]
\[
+2t\left(\frac{3}{2}M\left(t+\tau\right)|X|{}^{-\frac{1}{2}}+2\right)F\left(A^{\#}\right)A_{ij}\}
\]
and its evolution 
\begin{equation}
\partial_{t}\ln G=M|X|{}^{\frac{3}{2}}+\left(\frac{3}{2}M\left(t+\tau\right)|X|^{-\frac{1}{2}}+2\right)\left(X\cdot\partial_{t}X\right)\label{118}
\end{equation}
\[
=M|X|^{\frac{3}{2}}+2t\left(\frac{3}{2}M\left(t+\tau\right)|X|{}^{-\frac{1}{2}}+2\right)F\left(A^{\#}\right)^{2}
\]
in which we use the $F$ curvature flow equation in normal parametrization
(see Definition \ref{d13}): 
\[
\partial_{t}X=F\left(A^{\#}\right)N
\]
and the $F$ self-shrinker equation for $\Sigma_{t}=\sqrt{-t}\,\Sigma$
(in Definition \ref{d4}): 
\[
X\cdot N=2tF\left(A^{\#}\right)
\]
Thus, by (\ref{101}), (\ref{116}), (\ref{117}) and (\ref{118}),
we have 

\begin{equation}
\Phi=\{\left(\frac{3}{2}M\left(t+\tau\right)|X|{}^{-\frac{1}{2}}+2\right)^{2}\mathbf{a}^{ij}\left(X\cdot\partial_{i}X\right)\left(X\cdot\partial_{j}X\right)+\, M|X|^{\frac{3}{2}}\label{119}
\end{equation}
\[
+\frac{1}{2}\left(\frac{3}{2}M\left(t+\tau\right)|X|{}^{-\frac{1}{2}}+2\right)\textrm{tr}\left(\mathbf{a}\right)
\]
\[
+\textrm{tr}\left(\mathbf{a}\right)+\,\frac{3}{4}M\left(t+\tau\right)|X|{}^{-\frac{5}{2}}\left(\textrm{tr}\left(\mathbf{a}\right)|X|{}^{2}-\,\mathbf{a}^{ij}\left(X\cdot\partial_{i}X\right)\left(X\cdot\partial_{j}X\right)\right)
\]

\noindent \resizebox{1.009\linewidth}{!}{
  \begin{minipage}{\linewidth}
  \begin{align*}
  +\left(\frac{3}{2}M\left(t+\tau\right)|X|{}^{-\frac{1}{2}}+2\right)\left\{ \left(\nabla_{i}\mathbf{a}^{ij}\right)\left(X\cdot\partial_{j}X\right)+\,2tF\left(A^{\#}\right)\left(F\left(A^{\#}\right)+\mathbf{a}^{ij}A_{ij}\right)\right\} -\, F\left(A^{\#}\right)H\}
\end{align*}
  \end{minipage}
}\\
which, together with (\ref{113}), implies that 
\begin{equation}
\Phi-\Psi=\{\frac{\lambda}{2}\left(\frac{3}{2}M\left(t+\tau\right)|X|{}^{-\frac{1}{2}}+2\right)+\,\frac{\lambda}{3}\label{120}
\end{equation}

\noindent \resizebox{1.009\linewidth}{!}{
  \begin{minipage}{\linewidth}
  \begin{align*}
+\left(\frac{3}{2}M\left(t+\tau\right)|X|{}^{-\frac{1}{2}}+2\right)\left\{ \left(\nabla_{k}\mathbf{a}^{kl}\right)\left(X\cdot\partial_{l}X\right)+2tF\left(A^{\#}\right)\left(F\left(A^{\#}\right)+\mathbf{a}^{kl}A_{kl}\right)\right\} -\, F\left(A^{\#}\right)H\}
\end{align*}
  \end{minipage}
}\\\\

By (\ref{102}), (\ref{116}), (\ref{117}) and (\ref{120}), 
\begin{equation}
2\Upsilon^{ij}-\left(\Phi-\Psi\right)\mathbf{a}^{ij}=\{\left(\frac{3}{2}M\left(t+\tau\right)|X|{}^{-\frac{1}{2}}+2\right)\left(\mathbf{a}^{ik}\mathbf{a}^{jl}g_{kl}-\,\frac{\lambda}{6}\mathbf{a}^{ij}\right)\label{121}
\end{equation}
\[
+\left(2\mathbf{a}^{ik}\mathbf{a}^{jl}g_{kl}-\frac{\lambda}{3}\mathbf{a}^{ij}\right)+\,\frac{3}{2}M\left(t+\tau\right)|X|^{-\frac{5}{2}}\mathbf{a}^{ik}\mathbf{a}^{jl}\left(|X|{}^{2}g_{kl}-\,\left(X\cdot\partial_{k}X\right)\left(X\cdot\partial_{l}X\right)\right)
\]

\noindent \resizebox{1.009\linewidth}{!}{
  \begin{minipage}{\linewidth}
  \begin{align*}
+\left(\frac{3}{2}M\left(t+\tau\right)|X|{}^{-\frac{1}{2}}+2\right)\left\{ \mathbf{a}^{ik}\nabla_{k}\mathbf{a}^{jl}+\mathbf{a}^{jk}\nabla_{k}\mathbf{a}^{il}-\mathbf{a}^{lk}\nabla_{k}\mathbf{a}^{ij}-\mathbf{a}^{ij}\nabla_{k}\mathbf{a}^{kl}\right\} \left(X\cdot\partial_{l}X\right)
\end{align*}
  \end{minipage}
}

\[
+\left(\frac{3}{2}M\left(t+\tau\right)|X|{}^{-\frac{1}{2}}+2\right)\left(2\mathbf{a}^{ik}\mathbf{a}^{jl}A_{kl}-\mathbf{a}^{ij}\mathbf{a}^{kl}A_{kl}-F\left(A^{\#}\right)\mathbf{a}^{ij}\right)2t\, F\left(A^{\#}\right)
\]
\[
-\partial_{t}\mathbf{a}^{ij}+\, F\left(A^{\#}\right)H\,\mathbf{a}^{ij}\}
\]
which can be estimated from below (using (\ref{57}), (\ref{58}),
(\ref{60}), (\ref{17}), (\ref{20}) and the homogeneity of $F$) by
\begin{equation}
2\Upsilon^{ij}-\left(\Phi-\Psi\right)\mathbf{a}^{ij}\,\geq\,\{\left(\frac{3}{2}M\left(t+\tau\right)|X|{}^{-\frac{1}{2}}+2\right)\left(\left(\frac{\lambda^{2}}{18}-36\frac{\varkappa}{\lambda}\right)g^{ij}+O\left(|X|{}^{-2}\right)\right)\label{122}
\end{equation}
\[
+\frac{\lambda^{2}}{9}g^{ij}\,+O\left(|X|{}^{-2}\right)\}
\]
where the notation $O\left(|X|^{-2}\right)$ means that 
\[
\Big|O\left(|X|^{-2}\right)\Big|\,\leq\, C\left(n,\,\mathcal{C},\,\parallel F\parallel_{C^{3}\left(U\right)}\right)\,|X|{}^{-2}
\]
Then (\ref{114}) follows from (\ref{112}) and the hypothesis ($\varkappa\leq6^{-4}\lambda^{3}$)
provided that $R\gg1$ (independing of $M$ and $\tau$). 

On the other hand, by (\ref{57}), (\ref{58}), (\ref{17}), (\ref{20}),
the homogeneity of $F$, the hypothesis that $\varkappa\leq6^{-4}\lambda^{3}$
(note that $\lambda\in(0,\,1]$) and $R\gg1$ (independent of $M$
and $\tau$), we can estimate (\ref{120}) from below by 
\[
\Phi-\Psi\geq\left(\frac{3}{2}M\left(t+\tau\right)|X|{}^{-\frac{1}{2}}+2\right)\left(\frac{\lambda}{6}\,-3\varkappa\,+O\left(|X|{}^{-2}\right)\right)+\frac{\lambda}{3}+O\left(|X|{}^{-2}\right)
\]
\begin{equation}
\geq\left(\frac{3}{2}M\left(t+\tau\right)|X|{}^{-\frac{1}{2}}+2\right)\frac{\lambda}{9}+\,\frac{\lambda}{6}\label{123}
\end{equation}
Similarly, from the $F$ self-shrinker equation for $\Sigma_{t}$,
the tangential component of the position vector can be estimated by
\begin{equation}
|X^{\top}|^{2}=|X|{}^{2}-\left(X\cdot N\right){}^{2}=|X|{}^{2}-\left(2t\, F\left(A^{\#}\right)\right)^{2}\label{124}
\end{equation}
\[
=|X|^{2}-\left(2t\, F\left(|X|\, A^{\#}\right)\right)^{2}|X|{}^{-2}=\,|X|^{2}+O\left(|X|^{-2}\right)
\]
Consequently, (\ref{113}) can be estimated from below, using (\ref{57}) and (\ref{124}), by
\begin{equation}
\Psi\geq\left(\frac{3}{2}M\left(t+\tau\right)|X|{}^{-\frac{1}{2}}+2\right)^{2}\mathbf{a}^{ij}\left(X\cdot\partial_{i}X\right)\left(X\cdot\partial_{j}X\right)+\, M\,|X|{}^{\frac{3}{2}}\label{125}
\end{equation}
\[
\geq\left(\frac{3}{2}M\left(t+\tau\right)|X|{}^{-\frac{1}{2}}+2\right)^{2}\left(\frac{\lambda}{3}|X|{}^{2}+O\left(|X|^{-2}\right)\right)+\, M\,|X|^{\frac{3}{2}}
\]
Multiply (\ref{123}) and (\ref{125}) to get 
\begin{equation}
\left(\Phi-\Psi\right)\Psi\,\geq\,\{\left(\frac{3}{2}M\left(t+\tau\right)|X|{}^{-\frac{1}{2}}+2\right)^{3}\frac{1}{36}\lambda^{2}|X|{}^{2}\label{126}
\end{equation}
 
 \noindent \resizebox{1.009\linewidth}{!}{
  \begin{minipage}{\linewidth}
  \begin{align*}
+\left(\frac{3}{2}M\left(t+\tau\right)|X|{}^{-\frac{1}{2}}+2\right)^{2}\frac{1}{27}\lambda^{2}|X|^{2}\,+\left(\frac{3}{2}M\left(t+\tau\right)|X|{}^{-\frac{1}{2}}+2\right)\frac{\lambda}{9}M\,|X|{}^{\frac{3}{2}}+\,\frac{\lambda}{6}M\,|X|{}^{\frac{3}{2}}\}
\end{align*}
  \end{minipage}
}\\\\

To achieve (\ref{115}), let's first rearrange (\ref{113}) to get
\begin{equation}
\Psi=\{\left(\frac{3}{2}M\left(t+\tau\right)|X|{}^{-\frac{1}{2}}+2\right)^{2}\mathbf{a}^{kl}\left(X\cdot\partial_{k}X\right)\left(X\cdot\partial_{l}X\right)+\, M\,|X|{}^{\frac{3}{2}}\label{127}
\end{equation}
\[
+\left(\frac{3}{2}M\left(t+\tau\right)|X|^{-\frac{1}{2}}+2\right)\left(\textrm{tr}\left(\mathbf{a}\right)-\frac{\mathbf{a}^{kl}\left(X\cdot\partial_{k}X\right)\left(X\cdot\partial_{l}X\right)}{2|X|{}^{2}}-\,\frac{\lambda}{6}\right)
\]
\[
+\frac{\mathbf{a}^{kl}\left(X\cdot\partial_{k}X\right)\left(X\cdot\partial_{l}X\right)}{|X|^{2}}-\,\frac{\lambda}{3}\}
\]
Then we take time-derivative of (\ref{127}) and estimate the result
by Proposition \ref{p15}, (\ref{17}), (\ref{20}), the homogeneity
of $F$ and its derivatives, the $F$ self-shrinker equation for $\Sigma_{t}$
(i.e. $X\cdot N=2tF\left(A^{\#}\right)$) and the $F$ curvature flow
equation (i.e. $\partial_{t}X=F\left(A^{\#}\right)N$)), and also
assuming that $R\gg1$ (depending on $\lambda$). Note that we could
simplify the computation by taking ``normal coordinates'' of $\Sigma_{t}$.
For instance, let's compute and estimate the time-derivative of the
first term in (\ref{127}) as follows:
\begin{equation}
\partial_{t}\left\{ \left(\frac{3}{2}M\left(t+\tau\right)|X|{}^{-\frac{1}{2}}+2\right)^{2}\mathbf{a}^{kl}\left(X\cdot\partial_{k}X\right)\left(X\cdot\partial_{l}X\right)\right\} \label{128}
\end{equation}

 \noindent \resizebox{1.009\linewidth}{!}{
  \begin{minipage}{\linewidth}
  \begin{align*}
=\{2\left(\frac{3}{2}M\left(t+\tau\right)|X|{}^{-\frac{1}{2}}+2\right)\left\{ \frac{3}{2}M\,|X|^{-\frac{1}{2}}+\,\frac{3}{2}M\left(t+\tau\right)\left(-\frac{1}{2}|X|^{-\frac{3}{2}}\right)\frac{X\cdot F\left(A^{\#}\right)N}{|X|}\right\} \mathbf{a}^{kl}\left(X\cdot\partial_{k}X\right)\left(X\cdot\partial_{l}X\right)
\end{align*}
  \end{minipage}
}

 \noindent \resizebox{1.009\linewidth}{!}{
  \begin{minipage}{\linewidth}
  \begin{align*}
+\left(\frac{3}{2}M\left(t+\tau\right)|X|^{-\frac{1}{2}}+2\right)^{2}\,\left\{ \left(\partial_{t}\mathbf{a}^{kl}\right)\left(X\cdot\partial_{k}X\right)\left(X\cdot\partial_{l}X\right)+\,2\mathbf{a}^{kl}\left(X\cdot\partial_{k}X\right)\left(X\cdot\partial_{l}\left(F\left(A^{\#}\right)N\right)\right)\right\} \}
\end{align*}
  \end{minipage}
}\\

By taking normal coordinates, we may assume (at the point of consideration) that $g_{ij}=\delta_{ij}$ (so the norm is Proposition \ref{p15}
becomes $\ell^{2}$ norm), $\left\{ \partial_{1}X,\cdots,\,\partial_{n}X,\, N\right\} $
is an orthonormal basis for $\mathbb{R}^{n+1}$, and the last term
in (\ref{128}) can be computed and estimated by 
\[
\partial_{l}\left(F\left(A^{\#}\right)N\right)=\frac{\partial F}{\partial S_{i}^{j}}\left(A^{\#}\right)\,\left(\partial_{l}A_{i}^{j}\right)N\,+\, F\left(A^{\#}\right)\left(-A_{l}^{k}\,\partial_{k}X\right)
\]
\[
=\frac{\partial F}{\partial S_{i}^{j}}\left(|X|\, A^{\#}\right)\,\left(\nabla_{l}A_{i}^{j}\right)N\,+\,|X|{}^{-1}F\left(|X|\, A^{\#}\right)\left(-A_{l}^{k}\,\partial_{k}X\right)=O\left(|X|{}^{-2}\right)
\]\\
Thus,  (\ref{128}) can be estimated by
\[
\left(\frac{3}{2}M\left(t+\tau\right)|X|{}^{-\frac{1}{2}}+2\right)\left(3M|X|{}^{-\frac{1}{2}}+M\cdot O\left(|X|{}^{-\frac{9}{2}}\right)\right)\,\mathbf{a}^{kl}\left(X\cdot\partial_{k}X\right)\left(X\cdot\partial_{l}X\right)
\]
\[
+\left(\frac{3}{2}M\left(t+\tau\right)|X|{}^{-\frac{1}{2}}+2\right)^{2}O\left(1\right)
\]\\
Doing the same thing to other terms in (\ref{127}) leads to

 \noindent \resizebox{1.009\linewidth}{!}{
  \begin{minipage}{\linewidth}
  \begin{align}\label{129}
\partial_{t}\Psi=\{\left(\frac{3}{2}M\left(t+\tau\right)|X|^{-\frac{1}{2}}+2\right)\left(3M|X|{}^{-\frac{1}{2}}+M\cdot O\left(|X|{}^{-\frac{9}{2}}\right)\right)\,\mathbf{a}^{kl}\left(X\cdot\partial_{k}X\right)\left(X\cdot\partial_{l}X\right)
\end{align}
  \end{minipage}
}

 \noindent \resizebox{1.009\linewidth}{!}{
  \begin{minipage}{\linewidth}
  \begin{align*}
+\left(\frac{3}{2}M\left(t+\tau\right)|X|{}^{-\frac{1}{2}}+2\right)^{2}O\left(1\right)\,+\, M\cdot O\left(|X|^{-\frac{1}{2}}\right)+\,\left(\frac{3}{2}M\left(t+\tau\right)|X|{}^{-\frac{1}{2}}+2\right)O\left(|X|^{-2}\right)\,+\, O\left(|X|{}^{-2}\right)\}
\end{align*}
  \end{minipage}
}

 \noindent \resizebox{1.009\linewidth}{!}{
  \begin{minipage}{\linewidth}
  \begin{align*}
\geq\left(\frac{3}{2}M\left(t+\tau\right)|X|{}^{-\frac{1}{2}}+2\right)\left(\frac{2}{3}\lambda M|X|^{\frac{3}{2}}\right)\,+\,\left(\frac{3}{2}M\left(t+\tau\right)|X|{}^{-\frac{1}{2}}+2\right)^{2}O\left(1\right)\,+\, M\cdot O\left(|X|^{-\frac{1}{2}}\right)
\end{align*}
  \end{minipage}
}\\\\

Similarly, we can compute $\nabla_{i}\left(\mathbf{a}^{ij}\nabla_{j}\Psi\right)$
and estimate it by 
\begin{equation}
\nabla_{i}\left(\mathbf{a}^{ij}\nabla_{j}\Psi\right)=\mathbf{a}^{ij}\nabla_{ij}^{2}\Psi+\,\left(\nabla_{i}\mathbf{a}^{ij}\right)\left(\nabla_{j}\Psi\right)\label{130}
\end{equation}

 \noindent \resizebox{1.009\linewidth}{!}{
  \begin{minipage}{\linewidth}
  \begin{align*}
=\left(\frac{3}{2}M\left(t+\tau\right)|X|{}^{-\frac{1}{2}}+2\right)^{2}O\left(1\right)\,+\,\left(\frac{3}{2}M\left(t+\tau\right)|X|^{-\frac{1}{2}}+2\right)O\left(|X|{}^{-2}\right)\,+\: M\cdot O\left(|X|{}^{-\frac{1}{2}}\right)
\end{align*}
  \end{minipage}
} \\
Then (\ref{115}) follows from (\ref{126}), (\ref{129}) and (\ref{130}).
\end{proof}

Using the above two lemmas, we can derive the following Carleman's
inequality on the flow $\left\{ \Sigma_{t}\right\} _{-1\leq t\leq0}$
(with $\Sigma_{0}=\mathcal{C}$).

\begin{prop}[Carleman's inequality]\label{p26}
Assume that $\varkappa\leq6^{-4}\lambda^{3}$ in (\ref{4}) and (\ref{5}).
Then there exists $R\geq1$ (depending on $\Sigma,\,\tilde{\Sigma},\mathcal{\, C},\, U,\,\parallel F\parallel_{C^{3}\left(U\right)},\,\lambda,\,\varkappa$)
so that for any constants $M\geq1$, $\tau\in(0,\,1]$, and one-parameter
family of $C^{2}$ functions $\mathscr{\mathrm{u}}_{t}=\mathscr{\mathrm{u}}\left(\cdot,\, t\right)$
which is compactly supported in $\Sigma_{t}\setminus\bar{B}_{R}$
for each $t\in[-\tau,\,0]$ and is differentiable in time, there holds
\begin{equation}
\frac{\lambda^{2}}{9}\int_{-\tau}^{0}\int_{\Sigma_{t}}\left(|\nabla_{\Sigma_{t}}\mathscr{\mathrm{u}}|^{2}+\mathscr{\mathrm{u}}^{2}\right)G\, d\mathcal{H}^{n}dt\label{131}
\end{equation}
\[
\leq\,\{\int_{-\tau}^{0}\int_{\Sigma_{t}}|\mathbf{P}\mathscr{\mathrm{u}}|{}^{2}G\, d\mathcal{H}^{n}dt\,+\,\frac{3}{\lambda}\int_{\Sigma_{-\tau}}|\nabla_{\Sigma_{-\tau}}\mathscr{\mathrm{u}}_{-\tau}|^{2}G\left(\cdot,\,-\tau\right)\, d\mathcal{H}^{n}
\]
\[
+\frac{1}{2}\int_{\mathcal{C}}\Psi\left(\cdot,\,0\right)\,\mathscr{\mathrm{u}}^{2}\left(\cdot,\,0\right)\, G\left(\cdot,\,0\right)\, d\mathcal{H}^{n}\}
\]
where $\mathcal{H}^{n}$ is the $n-$dimensional Hausdorff measure;
$\mathbf{P}$, $G$ and $\Psi$ are defined in (\ref{53}), (\ref{112}),
(\ref{113}), respectively. \end{prop}
\begin{proof}
Apply Lemma \ref{l23} to the hypersurface $\Sigma_{t}$ (with $\partial_{t}g=-2F\left(A^{\#}\right)A$),
the differential operator $\mathbf{P}$ and the function $u_{t}$
to get 

 \noindent \resizebox{1.009\linewidth}{!}{
  \begin{minipage}{\linewidth}
  \begin{align*}
\int_{\mathrm{\Sigma_{t}}}\left\{ \left(2\Upsilon^{ij}-\left(\Phi-\Psi\right)\mathbf{a}^{ij}\right)\nabla_{i}\mathscr{\mathrm{u}}\,\nabla_{j}\mathscr{\mathrm{u}}+\,\frac{1}{2}\left(\partial_{t}\Psi-\nabla_{i}\left(a^{ij}\nabla_{j}\Psi\right)+\left(\Phi-\Psi\right)\Psi\right)\mathscr{\mathrm{u}}^{2}\right\} G\, d\mathcal{H}^{n}
\end{align*}
  \end{minipage}
}

 \noindent \resizebox{1.009\linewidth}{!}{
  \begin{minipage}{\linewidth}
  \begin{align*}
=\{\int_{\mathrm{\Sigma_{t}}}2\,\mathbf{P}\mathscr{\mathrm{u}}\,\left(\partial_{t}\mathscr{\mathrm{u}}+\,\mathbf{a}^{ij}\nabla_{i}\ln G\,\nabla_{j}\mathrm{u}+\,\frac{1}{2}\Psi\mathscr{\mathrm{u}}\right)G\, d\mathcal{H}^{n}-\,\int_{\mathrm{\Sigma_{t}}}2\left(\partial_{t}\mathscr{\mathrm{u}}+\,\mathbf{a}^{ij}\nabla_{i}\ln G\,\nabla_{j}\mathscr{\mathrm{u}}+\,\frac{1}{2}\Psi\mathscr{\mathrm{u}}\right)^{2}G\, d\mathcal{H}^{n}
\end{align*}
  \end{minipage}
} \\

\begin{equation}
-\partial_{t}\left\{ \int_{\mathrm{\Sigma_{t}}}\left(\mathbf{a}^{ij}\nabla_{i}\mathscr{\mathrm{u}}\,\nabla_{j}\mathscr{\mathrm{u}}-\,\frac{1}{2}\Psi\mathscr{\mathrm{u}}^{2}\right)G\, d\mathcal{H}^{n}\right\}\} \label{132}
\end{equation}
By Cauchy-Schwarz inequality, the RHS of (\ref{132}) is bounded from
above by
\begin{equation}
\int_{\Sigma_{t}}|\mathbf{P}\mathscr{\mathrm{u}}|{}^{2}G\, d\mathcal{H}^{n}dt\,-\,\partial_{t}\left\{ \int_{\Sigma_{t}}\left(\mathbf{a}^{ij}\nabla_{i}\mathscr{\mathrm{u}}\,\nabla_{j}\mathscr{\mathrm{u}}-\,\frac{1}{2}\Psi\mathscr{\mathrm{u}}^{2}\right)G\, d\mathcal{H}^{n}\right\} \label{133}
\end{equation}
By Lemma  \ref{l24} and $R\geq1$, the LHS of (\ref{132}) is bounded from
below by
\begin{equation}
\frac{\lambda^{2}}{9}\int_{\Sigma_{t}}\left(|\nabla_{\Sigma_{t}}\mathscr{\mathrm{u}}|{}^{2}+\mathscr{\mathrm{u}}^{2}\right)G\, d\mathcal{H}^{n}\label{134}
\end{equation}
Combining (\ref{132}), (\ref{133}), (\ref{134}), we get
\begin{equation}
\frac{\lambda^{2}}{9}\int_{\Sigma_{t}}\left(|\nabla_{\Sigma_{t}}\mathscr{\mathrm{u}}|{}^{2}+\mathscr{\mathrm{u}}^{2}\right)G\, d\mathcal{H}^{n}\label{135}
\end{equation}
\[
\leq\int_{\Sigma_{t}}\mathbf{|P}\mathscr{\mathrm{u}}|{}^{2}G\, d\mathcal{H}^{n}dt\,-\,\partial_{t}\left\{ \int_{\Sigma_{t}}\left(\mathbf{a}^{ij}\nabla_{i}\mathscr{\mathrm{u}}\,\nabla_{j}\mathscr{\mathrm{u}}-\,\frac{1}{2}\Psi\mathscr{\mathrm{u}}^{2}\right)G\, d\mathcal{H}^{n}\right\} 
\]
Integrate (\ref{135}) in time from $-\tau$ to $0$ and then use
(\ref{57}) and $\Psi\geq0$ to conclude (\ref{131}).
\end{proof}

Now we are ready to show that $h$ vanishes outside a compact set.
We basically follows the proof in \cite{ESS} (which is also used
in \cite{W}).
\begin{thm}\label{t27}
Suppose that $\varkappa\leq6^{-4}\lambda^{3}$ in (\ref{4}) and (\ref{5}),
then there exits $\boldsymbol{R}=\boldsymbol{R}\left(\Sigma,\,\tilde{\Sigma},\mathcal{\, C},\, U,\,\parallel F\parallel_{C^{3}\left(U\right)},\,\lambda,\,\varkappa\right)\geq1$
so that the deviation $h\left(\cdot,\,-1\right)$ of $\tilde{\Sigma}$
from $\Sigma$ vanishes on $\Sigma\setminus\bar{B}_{\boldsymbol{R}}$.
In other words, $\tilde{\Sigma}=\Sigma$ outside the ball $B_{\boldsymbol{R}}$.\end{thm}
\begin{proof}
Choose $R\gg1$ (depending on $\Sigma,\,\tilde{\Sigma},\mathcal{\, C},\, U,\,\parallel F\parallel_{C^{3}\left(U\right)},\,\lambda$)
so that Proposition \ref{p14}, Proposition \ref{p15}, Proposition \ref{p22}, Proposition
\ref{p26} and (\ref{20}) hold; in particular, we may assume that for all
$X\in\Sigma_{t}\setminus\bar{B}_{R}$, $t\in\left[-\tau,\,0\right]$
\begin{equation}
|\mathbf{P}h|\,\leq\frac{\lambda}{6}\left(|\nabla_{\Sigma_{t}}h|\,+\,|h|\right)\label{136}
\end{equation}
\begin{equation}
|\nabla_{\Sigma_{t}}h|\,+\,|h|\,\leq\Lambda\,\exp\left(\frac{|X|^{2}}{\Lambda t}\right)\label{137}
\end{equation}
where $\Lambda=\Lambda\left(n,\,\mathcal{C},\,\parallel F\parallel_{C^{3}\left(U\right)},\,\lambda\right)>0$,
$\tau\equiv\textrm{min}\left\{ \alpha\left(n,\,\mathcal{C},\,\parallel F\parallel_{C^{3}\left(U\right)},\,\lambda\right),\,\frac{1}{\Lambda}\right\} $
(see Proposition \ref{p22}). 

For any given $M\geq1$ and $\mathcal{R}\geq4R+1$, choose a smooth
cut-off function $\zeta=\zeta\left(X\right)$ so that 
\begin{equation}
\chi_{B_{\mathcal{R}-1}\setminus\bar{B}_{R+1}}\,\leq\zeta\leq\,\chi_{B_{\mathcal{R}}\setminus\bar{B}_{R}}\label{138}
\end{equation}
\[
|D\zeta|\,+\,|D^{2}\zeta|\,\leq3
\]
Note that $D\zeta$ is supported in $E=\left\{ X\in\mathbb{R}^{n+1}\Big|\, R\leq\,|X|\,\leq R+1\;\textrm{or}\;\mathcal{R}-1\leq\,|X|\,\leq\mathcal{R}\right\} $. 

Let $\mathscr{\mathrm{u}}\left(\cdot,\, t\right)=\zeta\, h\left(\cdot,\, t\right)$,
then $\mathscr{\mathrm{u}}\left(\cdot,\, t\right)$ is compactly supported
in $\Sigma_{t}\setminus\bar{B}_{R}$ for each $t\in[-\tau,\,0]$,
and we have, by (\ref{136}), (\ref{137}), (\ref{138}) 
\begin{equation}
\Big|\mathbf{P}\mathscr{\mathrm{u}}\Big|\,=\,\Big|\zeta\,\mathbf{P}h-h\,\mathbf{P}\zeta-2\mathbf{a}^{ij}\nabla_{i}\zeta\,\nabla_{j}h\Big|\label{139}
\end{equation}
\[
\leq\frac{\lambda}{6}\left(|\nabla_{\Sigma_{t}}\mathscr{\mathrm{u}}|\,+\,|\mathscr{\mathrm{u}}|\right)+\, C\left(n,\,\mathcal{C},\,\parallel F\parallel_{C^{3}\left(U\right)}\right)\left(|\nabla_{\Sigma_{t}}h|\,+\,|h|\right)\,\chi_{E}
\]
\[
\leq\frac{\lambda}{6}\left(|\nabla_{\Sigma_{t}}\mathscr{\mathrm{u}}|\,+\,|\mathscr{\mathrm{u}}|\right)+\, C\left(n,\,\mathcal{C},\,\parallel F\parallel_{C^{3}\left(U\right)},\,\lambda\right)\,\exp\left(\frac{|X|{}^{2}}{\Lambda t}\right)\,\chi_{E}
\]
\begin{equation}
\mathscr{\mathrm{u}}\left(\cdot,\,0\right)=0\label{140}
\end{equation}
By (\ref{139}), (\ref{140}), Proposition \ref{p26} and (\ref{137}), we get 
\[
\frac{\lambda^{2}}{9}\int_{-\tau}^{0}\int_{\Sigma_{t}}\left(|\nabla_{\Sigma_{t}}\mathscr{\mathrm{u}}|^{2}+\mathscr{\mathrm{u}}^{2}\right)G\, d\mathcal{H}^{n}dt\,\leq\,\{\frac{\lambda^{2}}{18}\int_{-\tau}^{0}\int_{\Sigma_{t}}\left(|\nabla_{\Sigma_{t}}\mathscr{\mathrm{u}}|^{2}+\mathscr{\mathrm{u}}^{2}\right)G\, d\mathcal{H}^{n}dt
\]
\begin{equation}
+\, C\left(n,\,\mathcal{C},\,\parallel F\parallel_{C^{3}\left(U\right)},\,\lambda\right)\,\int_{-\tau}^{0}\int_{\Sigma_{t}\cap E}\exp\left(2\frac{|X|{}^{2}}{\Lambda t}\right)\, G\, d\mathcal{H}^{n}dt\label{141}
\end{equation}
\[
+C\left(n,\,\mathcal{C},\,\parallel F\parallel_{C^{3}\left(U\right)},\,\lambda\right)\,\int_{\Sigma_{-\tau}}\exp\left(-2\frac{|X|{}^{2}}{\Lambda\tau}\right)\, G\left(\cdot,\,-\tau\right)\, d\mathcal{H}^{n}\}
\]
where $G$ is defined in (\ref{112}). Note that by the choice $\tau\leq\frac{1}{\Lambda}$,
we can estimate the last two terms on the RHS of (\ref{141}) by
\begin{equation}
\int_{-\tau}^{0}\int_{\Sigma_{t}\cap E}\exp\left(2\frac{|X|{}^{2}}{\Lambda t}\right)\, G\, d\mathcal{H}^{n}dt\,\leq\,\int_{-\tau}^{0}\int_{\Sigma_{t}\cap E}\exp\left(M\tau|X|{}^{\frac{3}{2}}-\,|X|{}^{2}\right)\, d\mathcal{H}^{n}dt\label{142}
\end{equation}
and
\begin{equation}
\int_{\Sigma_{-\tau}}\exp\left(-2\frac{|X|{}^{2}}{\Lambda\tau}\right)\, G\left(\cdot,\,-\tau\right)\, d\mathcal{H}^{n}\,\leq\,\int_{\Sigma_{-\tau}}\exp\left(-|X|{}^{2}\right)\, d\mathcal{H}^{n}\label{143}
\end{equation}
Consequently, by (\ref{142}), (\ref{143}) and noting that the first
term on the RHS of (\ref{141}) can be absorbed by its LHS, we get
from (\ref{141}) that 
\begin{equation}
\frac{\lambda^{2}}{18}\int_{-\tau}^{0}\int_{\Sigma_{t}}\left(|\nabla_{\Sigma_{t}}\mathscr{\mathrm{u}}|{}^{2}+\mathscr{\mathrm{u}}^{2}\right)G\, d\mathcal{H}^{n}dt\,\label{144}
\end{equation}
\[
\leq \, \{C\left(n,\,\mathcal{C},\,\parallel F\parallel_{C^{3}\left(U\right)},\,\lambda\right)\int_{-\tau}^{0}\int_{\Sigma_{t}\cap E}\exp\left(M\tau|X|{}^{\frac{3}{2}}-|X|{}^{2}\right)\, d\mathcal{H}^{n}dt
\]
\[
+C\left(n,\,\mathcal{C},\,\parallel F\parallel_{C^{3}\left(U\right)},\,\lambda\right)\int_{\Sigma_{-\tau}}\exp\left(-|X|^{2}\right)\, d\mathcal{H}^{n}\}
\]
 
\[
\leq\, \{C\left(n,\,\mathcal{C},\,\parallel F\parallel_{C^{3}\left(U\right)},\,\lambda\right)\int_{-\tau}^{0}\int_{\Sigma_{t}\cap\left(B_{\mathcal{R}-1}\setminus\bar{B}_{\mathcal{R}}\right)}\exp\left(M\tau\mathcal{R}{}^{\frac{3}{2}}-\left(\mathcal{R}-1\right){}^{2}\right)\, d\mathcal{H}^{n}dt
\]
\[
+C\left(n,\,\mathcal{C},\,\parallel F\parallel_{C^{3}\left(U\right)},\,\lambda\right)\int_{-\tau}^{0}\int_{\Sigma_{t}\cap\left(B_{R}\setminus\bar{B}_{R+1}\right)}\exp\left(M\tau\left(R+1\right){}^{\frac{3}{2}}-R{}^{2}\right)\, d\mathcal{H}^{n}dt
\]
\[
+C\left(n,\,\mathcal{C},\,\parallel F\parallel_{C^{3}\left(U\right)},\,\lambda\right)\int_{\Sigma_{-\tau}}\exp\left(-|X|{}^{2}\right)\, d\mathcal{H}^{n}\}
\]

The first term on the RHS of (\ref{144}) goes away as $\mathcal{\mathcal{R}}\nearrow\infty$;
the last term is bounded from above by $C\left(n,\,\mathcal{C},\,\parallel F\parallel_{C^{3}\left(U\right)},\,\lambda\right)$
because of (\ref{10}). For the LHS of (\ref{144}), we have
\[
\frac{\lambda^{2}}{18}\int_{-\tau}^{0}\int_{\Sigma_{t}}\left(|\nabla_{\Sigma_{t}}\mathscr{\mathrm{u}}|^{2}+\mathscr{\mathrm{u}}^{2}\right)G\, d\mathcal{H}^{n}dt\,\geq\,\frac{\lambda^{2}}{18}\int_{-\frac{\tau}{2}}^{0}\int_{\Sigma_{t}\cap\left(B_{\mathcal{R}-1}\setminus\bar{B}_{4R}\right)}\mathscr{\mathrm{u}}^{2}G\, d\mathcal{H}^{n}dt
\]
\[
\geq\frac{\lambda^{2}}{18}\,\exp\left(4M\tau R{}^{\frac{3}{2}}\right)\,\int_{-\frac{\tau}{2}}^{0}\int_{\Sigma_{t}\cap\left(B_{\mathcal{R}-1}\setminus\bar{B}_{4R}\right)}h^{2}\, d\mathcal{H}^{n}dt
\]
Therefore, let $\mathcal{\mathcal{R}}\nearrow\infty$ in (\ref{144}),
we arrive at
\begin{equation}
\int_{-\frac{\tau}{2}}^{0}\int_{\Sigma_{t}\setminus\bar{B}_{4R}}h^{2}\, d\mathcal{H}^{n}dt\label{145}
\end{equation}
\[
\leq\exp\left(-4M\tau R{}^{\frac{3}{2}}\right)\, C\left(n,\,\mathcal{C},\,\parallel F\parallel_{C^{3}\left(U\right)},\,\lambda\right)\left\{ \exp\left(2\sqrt{2}M\tau R{}^{\frac{3}{2}}\right)\,+1\right\} 
\]
Let $M\nearrow\infty$ in (\ref{145}), we get $h_{t}=h\left(\cdot,\, t\right)$
vanishes on $\Sigma_{t}\setminus\bar{B}_{4R}$ for $t\in\left[-\frac{\tau}{2},\,0\right]$,
and hence $\tilde{\Sigma}_{-\frac{\tau}{2}}=\sqrt{\frac{\tau}{2}}\,\tilde{\Sigma}$
coincides with $\Sigma_{-\frac{\tau}{2}}=\sqrt{\frac{\tau}{2}}\,\Sigma$
outside $B_{4R}$, which in turn shows that $\tilde{\Sigma}$ coincide
with $\Sigma$ outside the ball of radius $\boldsymbol{R}=\frac{4R}{\sqrt{\nicefrac{\tau}{2}}}$.
\end{proof}
By the previous theorem and the ``unique continuation principle''
in Proposition \ref{p20} (see Remark \ref{r21}), we have the following characterization of
the overlap region of $\Sigma$ and $\tilde{\Sigma}$. 
\begin{thm}\label{t28}
Under the same hypothesis of Theorem \ref{t27}, let 
\[
\Sigma^{0}=\left\{ X\in\Sigma\cap\tilde{\Sigma}\Big|\,\Sigma\textrm{ coincides with }\tilde{\Sigma}\textrm{ in a neighborhood of }X\right\} 
\]
then \textup{$\Sigma^{0}$ is a nonempty hypersurface and} $\partial\Sigma^{0}\subseteq\left(\partial\Sigma\,\cup\,\partial\tilde{\Sigma}\right)$. \end{thm}
\begin{proof}
Note that $\Sigma^{0}$ is a nonempty hypersurface follows from Theorem \ref{t27}.

Suppose that $\partial\Sigma^{0}\nsubseteq\left(\partial\Sigma\,\cup\,\partial\tilde{\Sigma}\right)$,
then pick $\hat{X}\in\partial\Sigma^{0}\setminus\left(\partial\Sigma\,\cup\,\partial\tilde{\Sigma}\right)$
and choose a sequence $\left\{ \hat{X}_{m}\in\Sigma^{0}\right\} $
converging to $\hat{X}$. Note that $N\left(\hat{X}\right)=\tilde{N}\left(\hat{X}\right)$
since $N\left(\hat{X}_{m}\right)=\tilde{N}\left(\hat{X}_{m}\right)$
for all $m\in\mathbb{N}$, where $N$ , $\tilde{N}$ are the unit-normal
of $\Sigma$ and $\tilde{\Sigma}$, respectively. Thus, near $\hat{X}$,
$\Sigma$ and $\tilde{\Sigma}$ can be regraded as graphes of $\mathfrak{u}$
and $\tilde{\mathfrak{u}}$, respectively, over $B_{\boldsymbol{\varrho}}^{n}\subset T_{\hat{X}}\Sigma=T_{\hat{X}}\tilde{\Sigma}$
for some $\boldsymbol{\varrho}\in\left(0,\,1\right)$. That is, $\Sigma$
and $\tilde{\Sigma}$ can be respectively parametrized by 
\[
X=X\left(x\right)\equiv\hat{X}+\left(x,\,\mathfrak{u}\left(x\right)\right),\quad\tilde{X}=\tilde{X}\left(x\right)\equiv\hat{X}+\left(x,\,\tilde{\mathfrak{u}}\left(x\right)\right)\quad\textrm{for}\; x\in B_{\boldsymbol{\varrho}}^{n}
\]
in which we assume that $N\left(\hat{X}\right)=\tilde{N}\left(\hat{X}\right)=\left(0,\,1\right)$
for ease of notation. Note also that $A_{i}^{j}\left(0\right)=\tilde{A_{i}}^{j}\left(0\right)$
since $A_{i}^{j}\left(x_{m}\right)=\tilde{A_{i}}^{j}\left(x_{m}\right)$
for all $m\in\mathbb{N}$, where $x_{m}$ is the coordinates of $\hat{X}_{m}$
(i.e. $X\left(x_{m}\right)=\hat{X}_{m}=\tilde{X}\left(x_{m}\right)$) and 
\begin{equation}
A^{\#}\left(x\right)\sim A_{i}^{j}\left(x\right)=\partial_{i}\left(\frac{\partial_{j}\mathfrak{u}\left(x\right)}{\sqrt{1+|\partial_{x}\mathfrak{u}|{}^{2}}}\right),\quad\tilde{A}^{\#}\left(x\right)\sim\tilde{A}_{i}^{j}\left(x\right)=\partial_{i}\left(\frac{\partial_{j}\mathfrak{\tilde{u}}\left(x\right)}{\sqrt{1+|\partial_{x}\mathfrak{\tilde{u}}|{}^{2}}}\right)\label{146}
\end{equation}
are the shape operators of $\Sigma$ and $\tilde{\Sigma}$, respectively.
As a result, we may assume (by choosing $\boldsymbol{\varrho}$ small
if necessary) that $\tilde{A}_{i}^{j}\left(x\right)$ is so close to $A_{i}^{j}\left(x\right)$ that the set 
\[
\mathfrak{U}=\left\{ \left(1-\theta\right)A_{i}^{j}\left(x\right)+\theta\tilde{A}_{i}^{j}\left(x\right)\Big|\, x\in B_{\boldsymbol{\varrho}}^{n},\,\theta\in\left[0,\,1\right]\right\} 
\]
is a bounded subset of $\boldsymbol{\Omega}$ and there holds 
\[
\bar{\lambda}\,\leq\frac{\partial F}{\partial S_{i}^{j}}\left(\left(1-\theta\right)A^{\#}\left(x\right)+\theta\tilde{A}^{\#}\left(x\right)\right)\leq\,\frac{1}{\bar{\lambda}}
\]
for some $\bar{\lambda}\in(0,\,1]$. 

From the $F$ shrinker equation in Definition \ref{d4}, we get

\noindent \resizebox{1.0\linewidth}{!}{
  \begin{minipage}{\linewidth}
  \begin{align}
\sqrt{1+|\partial_{x}\mathfrak{u}|^{2}}\, F\left(A_{i}^{j}\left(x\right)\right)+\,\frac{1}{2}\left(\mathfrak{u}-x\cdot\partial_{x}\mathfrak{u}\right)=0,\quad\sqrt{1+|\partial_{x}\tilde{\mathfrak{u}}|^{2}}\, F\left(\tilde{A}_{i}^{j}\left(x\right)\right)+\,\frac{1}{2}\left(\tilde{\mathfrak{u}}-\, x\cdot\partial_{x}\tilde{\mathfrak{u}}\right)=0\label{147}
\end{align}
  \end{minipage}
}\\

Subtracting (\ref{147}) and using (\ref{146}), we then get an equation for $\mathtt{v=\tilde{\mathfrak{u}}-\mathfrak{u}}$:
\begin{equation}
\mathfrak{a}^{ij}\partial_{ij}^{2}\mathtt{v}+\,\mathfrak{b}^{j}\partial_{j}\mathtt{v}+\,\frac{1}{2}\mathtt{v}=0\label{148}
\end{equation}
where

\noindent \resizebox{1.0\linewidth}{!}{
  \begin{minipage}{\linewidth}
  \begin{align}
\mathfrak{a}^{ij}\left(x\right)=\int_{0}^{1}\left\{ \frac{\partial F}{\partial S_{i}^{j}}\left(\left(1-\theta\right)A^{\#}\left(x\right)+\theta\tilde{A}^{\#}\left(x\right)\right)-\,\frac{\partial F}{\partial S_{i}^{k}}\left(\left(1-\theta\right)A^{\#}\left(x\right)+\theta\tilde{A}^{\#}\left(x\right)\right)\,\frac{\partial_{k}\mathfrak{u}_{\theta}\,\partial_{j}\mathfrak{u}_{\theta}}{1+|\partial_{x}\mathfrak{u}_{\theta}|{}^{2}}\right\} \, d\theta\label{149}
\end{align}
  \end{minipage}
}

\begin{equation}
\mathfrak{b}^{j}\left(x\right)=\{-\int_{0}^{1}\frac{\partial F}{\partial S_{i}^{j}}\left(\left(1-\theta\right)A^{\#}\left(x\right)+\theta\tilde{A}^{\#}\left(x\right)\right)\,\frac{\partial_{k}\mathfrak{u}_{\theta}\,\partial_{ik}^{2}\mathfrak{u}_{\theta}}{1+|\partial_{x}\mathfrak{u}_{\theta}|^{2}}\, d\theta\label{150}
\end{equation}
\[
-\int_{0}^{1}\frac{\partial F}{\partial S_{i}^{k}}\left(\left(1-\theta\right)A^{\#}\left(x\right)+\theta\tilde{A}^{\#}\left(x\right)\right)\,\frac{\partial_{j}\mathfrak{u}_{\theta}\,\partial_{ik}^{2}\mathfrak{u}_{\theta}+\partial_{k}\mathfrak{u}_{\theta}\,\partial_{ij}^{2}\mathfrak{u}_{\theta}}{1+|\partial_{x}\mathfrak{u}_{\theta}|{}^{2}}\, d\theta
\]
\[
+3\int_{0}^{1}\frac{\partial F}{\partial S_{i}^{k}}\left(\left(1-\theta\right)A^{\#}\left(x\right)+\theta\tilde{A}^{\#}\left(x\right)\right)\,\frac{\partial_{j}\mathfrak{u}_{\theta}\,\partial_{k}\mathfrak{u}_{\theta}\,\partial_{l}\mathfrak{u}_{\theta}\,\partial_{il}^{2}\mathfrak{u}_{\theta}}{\left(1+|\partial_{x}\mathfrak{u}_{\theta}|^{2}\right)^{\frac{3}{2}}}\, d\theta
\]
\[
+\int_{0}^{1}F\left(\left(1-\theta\right)A^{\#}\left(x\right)+\theta\tilde{A}^{\#}\left(x\right)\right)\,\frac{\partial_{j}\mathfrak{u}_{\theta}}{\sqrt{1+|\partial_{x}\mathfrak{u}_{\theta}|{}^{2}}}\, d\theta-\,\frac{1}{2}x_{j}\}
\]
and 
$$\mathfrak{u}_{\theta}=\left(1-\theta\right)\mathfrak{u}+\theta\tilde{\mathfrak{u}}$$
Note that (\ref{148}) is equivalent to the following divergence form
equation:
\begin{equation}
-\partial_{i}\left(\frac{\mathfrak{a}^{ij}+\mathfrak{a}^{ji}}{2}\partial_{j}\mathtt{v}\right)\,=\left(-\partial_{i}\left(\frac{\mathfrak{a}^{ij}+\mathfrak{a}^{ji}}{2}\right)+\mathfrak{b}^{j}\right)\partial_{j}\mathtt{v}+\,\frac{1}{2}\mathtt{v}\label{151}
\end{equation}
And by (\ref{149}), (\ref{150}) and (\ref{146}), we have the following
estimates for the coefficients of (\ref{151}): 
\begin{equation}
\frac{\bar{\lambda}}{1+\parallel\partial_{x}\mathfrak{u}_{\theta}\parallel_{L^{\infty}\left(B_{\boldsymbol{\varrho}}^{n}\right)}^{2}}\,\leq\,\frac{\mathfrak{a}^{ij}+\mathfrak{a}^{ji}}{2}\,\leq\, C\left(\parallel F\parallel_{C^{1}\left(\mathfrak{U}\right)},\;\parallel\mathfrak{u}\parallel_{C^{2}\left(B_{\boldsymbol{\varrho}}^{n}\right)}\right)\label{152}
\end{equation}
\begin{equation}
|\partial_{x}\mathfrak{a}^{ij}|\,+\,|\mathfrak{b}^{j}|\,\leq\, C\left(\parallel F\parallel_{C^{2}\left(\mathfrak{U}\right)},\;\parallel\mathfrak{u}\parallel_{C^{3}\left(B_{\boldsymbol{\varrho}}^{n}\right)}\right)\label{153}
\end{equation}

On the other hand, since $\hat{X}_{m}\in\Sigma^{0}$ and $\hat{X}_{m}\rightarrow\hat{X}$
as $m\nearrow\infty$, $\mathtt{v}$ is vanishing at each neighborhood
of $x_{m}$ and $x_{m}\rightarrow0$ as $m\nearrow\infty$. Thus,
by Proposition \ref{p20} and Remark \ref{r21}, $\mathtt{v}$ vanishes on $B^{n}\left(x_{m},\,\frac{1}{4}\left(\boldsymbol{\varrho}-|x_{m}|\right)\right)$
for all $m\in\mathbb{N}$, which implies that $\mathtt{v}$ vanishes
on $B^{n}\left(0,\,\frac{1}{4}\boldsymbol{\varrho}\right)$. In other
words, $\Sigma$ coincides with $\tilde{\Sigma}$ in a neighborhood
of $\hat{X}$, which contradicts with $\hat{X}\in\partial\Sigma^{0}$. 
\end{proof}

Lastly, we give an estimate of $\varkappa$ (defined in (\ref{5}))
in the rotationally symmetric case to conclude this section. From now on, we assume that the cone $\mathcal{C}$ (in
Definition \ref{d1}) is rotationally symmetric, say 
\[
\mathcal{C}=\left\{ \left(\sigma s\,\nu,\, s\right)\Big|\,\nu\in\mathcal{\mathbf{S}}^{n-1},\, s\in\mathbb{R}_{+}\right\} 
\]
for some constant $\sigma>0$, where $\mathcal{\mathbf{S}}^{n-1}$
is the unit-sphere in $\mathbb{R}^{n}$. To derive the estimate, we have to first
compute the covariant derivatives of the second fundamental form of
$\mathcal{C}$. 
\begin{lem}\label{l34}
At each point $X_{\mathcal{C}}=\left(\sigma s\,\nu,\, s\right)$\textup{$\in\mathcal{C}$
(where $\nu\in\mathcal{\mathbf{S}}^{n-1},\, s>0$), }pick an orthonormal
basis $\left\{ e_{1}^{\mathcal{C}},\cdots,\, e_{n}^{\mathcal{C}}\right\} $
for $T_{X_{\mathcal{C}}}\mathcal{C}$ so that $e_{n}^{\mathcal{C}}=\frac{\left(\sigma\nu,\,1\right)}{\sqrt{1+\sigma^{2}}}$,
then we have
\begin{equation}
A_{\mathcal{C}}\left(e_{i}^{\mathcal{C}},\, e_{j}^{\mathcal{C}}\right)=\kappa_{i}^{\mathcal{C}}\delta_{ij},\quad\textrm{with}\;\kappa_{1}^{\mathcal{C}}=\cdots=\kappa_{n-1}^{\mathcal{C}}=\frac{1}{\sigma|X_{\mathcal{C}}|},\;\kappa_{n}^{\mathcal{C}}=0\label{237}
\end{equation}
\begin{equation}
\nabla_{\mathcal{C}}A_{\mathcal{C}}\left(e_{i}^{\mathcal{C}},\, e_{j}^{\mathcal{C}},\, e_{n}^{\mathcal{C}}\right)=\frac{-1}{\sigma|X_{\mathcal{C}}|^{2}}\delta_{ij}=-\frac{\kappa_{i}^{\mathcal{C}}}{|X_{\mathcal{C}}|}\delta_{ij},\quad\forall\, i,\, j\neq n\label{238}
\end{equation}
\begin{equation}
\nabla_{\mathcal{C}}A_{\mathcal{C}}\left(e_{i}^{\mathcal{C}},\, e_{j}^{\mathcal{C}},\, e_{k}^{\mathcal{C}}\right)=\nabla_{\mathcal{C}}A_{\mathcal{C}}\left(e_{i}^{\mathcal{C}},\, e_{n}^{\mathcal{C}},\, e_{n}^{\mathcal{C}}\right)=\nabla_{\mathcal{C}}A_{\mathcal{C}}\left(e_{n}^{\mathcal{C}},\, e_{n}^{\mathcal{C}},\, e_{n}^{\mathcal{C}}\right)=0\quad\forall\, i,\, j,\, k\neq n\label{239}
\end{equation}
where $A_{\mathcal{C}}$ is the second fundamental form of $\mathcal{C}$
and $\nabla_{\mathcal{C}}A_{\mathcal{C}}$ is its covariant derivative.
Note that $A_{\mathcal{C}}$ and $\nabla_{\mathcal{C}}A_{\mathcal{C}}$
are totally symmetric tensors (by Codazzi equation).\end{lem}
\begin{proof}
Let's parametrize $\mathcal{C}$ by 
\[
X_{\mathcal{C}}=\left(\sigma s\,\nu,\, s\right)\quad\textrm{for}\;\nu\in\mathcal{\mathbf{S}}^{n-1},\, s\in\mathbb{R}_{+}
\]
and take an othornomal local frame $\left\{ e_{1}^{\mathcal{C}},\cdots,\, e_{n}^{\mathcal{C}}\right\} $
of $\mathcal{C}$ so that 
\begin{equation}
e_{n}^{\mathcal{C}}\,=\,\frac{\partial_{s}X_{\mathcal{C}}}{|\partial_{s}X_{\mathcal{C}}|}\,=\,\frac{\left(\sigma\nu,\,1\right)}{\sqrt{1+\sigma^{2}}}\label{240}
\end{equation}
By a simple calculation, the principal curvatures of the cone $\mathcal{C}$ are given by 
\begin{equation}
\kappa_{1}^{\mathcal{C}}=\cdots=\kappa_{n-1}^{\mathcal{C}}=\frac{1}{\sigma s\sqrt{1+\sigma^{2}}}=\frac{1}{\sigma|X_{\mathcal{C}}|},\quad\kappa_{n}^{\mathcal{C}}=0\label{241}
\end{equation}
Since $\left\{ e_{1}^{\mathcal{C}},\cdots,\, e_{n}^{\mathcal{C}}\right\} $
forms a principal basis at each point, so by (\ref{241}) we have
\begin{equation}
A_{ii}^{\mathcal{C}}=\kappa_{i}^{\mathcal{C}}=\frac{1}{\sigma s\sqrt{1+\sigma^{2}}}=\frac{1}{\sigma|X_{\mathcal{C}}|}\quad\textrm{whenever}\: i\neq n\label{242}
\end{equation}
\[
A_{ij}^{\mathcal{C}}=0=A_{nn}^{\mathcal{C}}\quad\textrm{whenever}\: i\neq j
\]
where $A_{ij}^{\mathcal{C}}\equiv A_{\mathcal{C}}\left(e_{i}^{\mathcal{C}},\, e_{j}^{\mathcal{C}}\right)$.\\

By the orthonormality of $\left\{ e_{1}^{\mathcal{C}},\cdots,\, e_{n}^{\mathcal{C}}\right\} $
and the product rule, the Christoffel symbols $\mathcal{^{C}}\Gamma_{ij}^{k}\equiv\left(D_{e_{i}^{\mathcal{C}}}\, e_{j}^{\mathcal{C}}\right)\cdot e_{k}^{\mathcal{C}}$
satisfy
\begin{equation}
^{\mathcal{C}}\Gamma_{ki}^{j}=\left(D_{e_{k}^{\mathcal{C}}}\, e_{i}^{\mathcal{C}}\right)\cdot e_{j}^{\mathcal{C}}=-\left(D_{e_{k}^{\mathcal{C}}}\, e_{j}^{\mathcal{C}}\right)\cdot e_{i}^{\mathcal{C}}=-^{\mathcal{C}}\Gamma_{kj}^{i}\label{243}
\end{equation}
Thus, from (\ref{242}) and (\ref{243}), we deduce that whenever
$i,\, j\neq n$ or $i=j=n$, there holds
\begin{equation}
\nabla_{k}^{\mathcal{C}}A_{ij}^{\mathcal{C}}=D_{e_{k}^{\mathcal{C}}}\left(A_{ij}^{\mathcal{C}}\right)\,-\,{}^{\mathcal{C}}\Gamma_{ki}^{j}A_{jj}^{\mathcal{C}}\,-\,{}^{\mathcal{C}}\Gamma_{kj}^{i}A_{ii}^{\mathcal{C}}=D_{e_{k}^{\mathcal{C}}}\left(A_{ij}^{\mathcal{C}}\right)\label{244}
\end{equation}

By (\ref{244}), (\ref{242}) and (\ref{240}), we get 
\[
\nabla_{n}^{\mathcal{C}}A_{ij}^{\mathcal{C}}=D_{e_{n}^{\mathcal{C}}}\left(\kappa_{i}^{\mathcal{C}}\delta_{ij}\right)=\frac{1}{\sqrt{1+\sigma^{2}}}\partial_{s}\left(\frac{1}{\sigma s\sqrt{1+\sigma^{2}}}\right)\,\delta_{ij}
\]
\[
=\frac{-1}{\sigma\left(1+\sigma^{2}\right)s^{2}}\delta_{ij}=\frac{-1}{\sigma|X_{\mathcal{C}}|{}^{2}}\delta_{ij}\quad\textrm{if}\; i,\, j\neq n
\]
which verifies (\ref{238}).

By (\ref{244}), (\ref{242}) and noting that $|X_{\mathcal{C}}|$
is invariant along $e_{k}^{\mathcal{C}}$ for $k\neq n$ , we get
\begin{equation}
\nabla_{k}^{\mathcal{C}}A_{ij}^{\mathcal{C}}=D_{e_{k}^{\mathcal{C}}}\left(\kappa_{i}^{\mathcal{C}}\delta_{ij}\right)=D_{e_{k}^{\mathcal{C}}}\left(\frac{1}{\sigma|X_{\mathcal{C}}|}\right)\,\delta_{ij}=0\quad\textrm{if}\; i,\, j,\, k\neq n\label{245}
\end{equation}
From (\ref{244}) and (\ref{242}), we have 
\begin{equation}
\nabla_{i}^{\mathcal{C}}A_{nn}^{\mathcal{C}}=D_{e_{i}^{\mathcal{C}}}\left(A_{nn}^{\mathcal{C}}\right)=0\quad\forall\, i\label{246}
\end{equation}
Then (\ref{239}) follows from (\ref{245}) and (\ref{246}).
\end{proof}

Combining (\ref{1}), (\ref{2}), (\ref{3}) with Lemma  \ref{l34}, we get
the following:
\begin{prop}\label{p35}
The constant $\varkappa$ defined in (\ref{5}) can be estimated by
\begin{equation}
\varkappa\,\leq\, C\left(n\right)\left(\Big|\partial^{2}f\left(\overrightarrow{1},\,0\right)\Big|\,+\,\Big|\partial_{1}f\left(\overrightarrow{1},\,0\right)-\partial_{n}f\left(\overrightarrow{1},\,0\right)\Big|\right)\label{247}
\end{equation}
Note that here we assume that $\mathcal{C}$ is rotationally symmetric.
\end{prop}
\begin{proof}
At each point $X_{\mathcal{C}}\in\mathcal{C}$, take an orthonormal
basis $\left\{ e_{1}^{\mathcal{C}},\cdots,\, e_{n}^{\mathcal{C}}\right\} $
for $T_{X_{\mathcal{C}}}\mathcal{C}$ so that $e_{n}^{\mathcal{C}}=\frac{\left(\sigma\nu,\,1\right)}{\sqrt{1+\sigma^{2}}}$.
Then by (\ref{2}), (\ref{3}), Lemma  \ref{l34} and the homogeneity of the
derivatives of $f$, we get 
\[
\Big|\frac{\partial^{2}F}{\partial S_{i}^{j}\partial S_{k}^{l}}\left(A_{\mathcal{C}}^{\#}\right)\Big|\,\leq\,\left(\Big|\partial^{2}f\left(\kappa_{1}^{\mathcal{C}},\cdots,\,\kappa_{n}^{\mathcal{C}}\right)\Big|\,+\,\Big|\frac{\partial_{1}f\left(\kappa_{1}^{\mathcal{C}},\cdots,\,\kappa_{n}^{\mathcal{C}}\right)-\partial_{n}f\left(\kappa_{1}^{\mathcal{C}},\cdots,\,\kappa_{n}^{\mathcal{C}}\right)}{\kappa_{1}^{\mathcal{C}}-\kappa_{n}^{\mathcal{C}}}\Big|\right)
\]
\[
=\frac{1}{\kappa_{1}^{\mathcal{C}}}\left(\Big|\partial^{2}f\left(\overrightarrow{1},\,0\right)\Big|\,+\,\Big|\partial_{1}f\left(\overrightarrow{1},\,0\right)-\partial_{k}f\left(\overrightarrow{1},\,0\right)\Big|\right)
\]
which implies that
\[
|X_{\mathcal{C}}|\Big|\sum_{k,\, l}\frac{\partial^{2}F}{\partial S_{i}^{j}\partial S_{k}^{l}}\left(A_{\mathcal{C}}^{\#}\right)\,\left(\nabla_{\mathcal{C}}A_{\mathcal{C}}^{\#}\right)_{k}^{l}\Big|
\]
\[
\leq|X_{\mathcal{C}}|\,\frac{C\left(n\right)}{\kappa_{1}^{\mathcal{C}}}\left(\Big|\partial^{2}f\left(\overrightarrow{1},\,0\right)\Big|\,+\,\Big|\partial_{1}f\left(\overrightarrow{1},\,0\right)-\partial_{k}f\left(\overrightarrow{1},\,0\right)\Big|\right)\,\frac{\kappa_{1}^{\mathcal{C}}}{|X_{\mathcal{C}}|}
\]
\[
=C\left(n\right)\left(\Big|\partial^{2}f\left(\overrightarrow{1},\,0\right)\Big|\,+\,\Big|\partial_{1}f\left(\overrightarrow{1},\,0\right)-\partial_{k}f\left(\overrightarrow{1},\,0\right)\Big|\right)
\]
Therefore,
\[
\varkappa=\sup_{X_{\mathcal{C}}\in\mathcal{C}\cap\left(B_{3}\setminus\bar{B}_{\frac{1}{3}}\right)}\Big|\sum_{k,\, l}\frac{\partial^{2}F}{\partial S_{i}^{j}\partial S_{k}^{l}}\left(A_{\mathcal{C}}^{\#}\right)\,\left(\nabla_{\mathcal{C}}A_{\mathcal{C}}^{\#}\right)_{k}^{l}\Big|
\]
\[
\leq C\left(n\right)\left(\Big|\partial^{2}f\left(\overrightarrow{1},\,0\right)\Big|\,+\,\Big|\partial_{1}f\left(\overrightarrow{1},\,0\right)-\partial_{k}f\left(\overrightarrow{1},\,0\right)\Big|\right).
\]
\end{proof}

\vspace{0.3in}
\email{
\noindent Department of Mathematics, Rutgers University - Hill Center for the Mathematical Sciences 
110 Frelinghuysen Rd., Piscataway, NJ 08854-8019\\\\
E-mail address: \textsf{showhow@math.rutgers.edu}
}
\end{document}